\documentclass[reqno, 11pt]{amsart}
\usepackage[utf8]{inputenc}
\usepackage[all]{xy}
\usepackage{enumerate}
\usepackage[english]{babel}
\usepackage{graphicx}
\usepackage{xcolor}
\usepackage{setspace}
\usepackage{cancel}
\usepackage[left=2.5cm, right=2.5cm, bottom=2cm]{geometry}
\usepackage{amsmath, amssymb, amsthm, epsfig, mathtools}
\usepackage{latexsym}
\usepackage{url}
\usepackage{float}
\usepackage{bbm}
\usepackage[pagewise]{lineno}
\usepackage{comment}
\usepackage{array,booktabs}
\usepackage{tikz}
\usepackage{placeins}
\usetikzlibrary{positioning,arrows.meta}
\usepackage[colorlinks=true, linkcolor=blue, citecolor=green, urlcolor=purple]{hyperref}
\usepackage[nameinlink,noabbrev]{cleveref}
\def\today{\ifcase\month\or
  January\or February\or March\or April\or May\or June\or
  July\or August\or September\or October\or November\or December\fi
  \space\number\day, \number\year}

\newcommand{\mc}{\mathcal}

\newcommand{\R}{\mathbb{R}}
\newcommand{\N}{\mathbb{N}}
\newcommand{\T}{\mathbb{T}}

\newcommand{\C}{\mathbb{C}}
\newcommand{\eps}{\varepsilon}
\newcommand{\dis}{\displaystyle}

\newtheorem{theorem}{Theorem}[section]
\newtheorem{lemma}{Lemma}[section]
\newtheorem{proposition}{Proposition}[section]
\newtheorem{corollary}{Corollary}[section]
\newtheorem{assumption}{Assumption}[section]

\theoremstyle{definition}
\newtheorem{definition}{Definition}[section]

\theoremstyle{remark}
\newtheorem{remark}{Remark}[section]

\crefname{theorem}{Theorem}{Theorems}
\crefname{lemma}{Lemma}{Lemmas}
\crefname{proposition}{Proposition}{Propositions}
\crefname{corollary}{Corollary}{Corollaries}
\crefname{assumption}{Assumption}{Assumptions}
\crefname{definition}{Definition}{Definitions}
\crefname{example}{Example}{Examples}
\crefname{remark}{Remark}{Remarks}
\begin{document}
\color{black}

\date{\today}

\title[Saturation on compact manifolds and forced nonlinear PDEs]{Saturation on compact manifolds via  Stone--Weierstrass type theorem, and forced nonlinear PDEs}

\author[R. Buffe]{Rémi Buffe}
\address{Université de Lorraine, CNRS, Inria, IECL, F-54000 Nancy, France.}
\email{remi.buffe@univ-lorraine.fr}

\author[A. Duca]{Alessandro Duca}
\address{Université de Lorraine, CNRS, Inria, IECL, F-54000 Nancy, France.}
\email{alessandro.duca@inria.fr}

\author[H. Parada]{Hugo Parada}
\address{Université de Lorraine, CNRS, Inria, IECL, F-54000 Nancy, France.}
\email{hugo.parada@inria.fr}
\thanks{This project is funded by the Agence Nationale de la Recherche by the QuBiCCS project (ANR-24-CE40-3008). R. Buffe is also funded by Agence Nationale de la Recherche by the COINS project (ANR-26-CE40-7933-01)}

\begin{abstract}
We consider polynomial nonlinear PDEs on compact manifolds with finite-dimensional forcing terms. We provide necessary and sufficient conditions, such that the forcing trough the nonlinear dynamics attains a dense subspace on manifolds with and without boundary. Our key argument is the Nachbin's theorem, which is an analogue of the well-known Stone-Weierstrass theorem in differentiable spaces. As applications, we prove small-time approximate controllability by additive finite-dimensional controls, topological irreducibility and almost sure unboundedness of Sobolev norms.
\end{abstract}
\maketitle

\setcounter{tocdepth}{1}
\tableofcontents


\section{Introduction and main results}
\label{sec:intro}

We study the density of power saturations generated by finitely many smooth functions. Roughly speaking, starting from a vector space spanned by finitely many real-valued functions on a compact manifold, we allow iterations of convex combinations and the nonlinear operation $f\mapsto f^p$, where $p\geq2$ is an integer. We call this procedure power saturation, and say that an initial family is saturating in a given Sobolev space if its power saturation is dense in that space. Our first goal is to construct finite families that are saturating in every Sobolev space. 

Our construction uses the coordinate functions of a smooth embedding of the manifold into an Euclidean space. Together with the constant function, these coordinates have dense power saturation for every integer degree $p\geq2$. For the cone saturation with even $p$, one additional smooth function is sufficient: it completes the coordinate family so that its sum of squares is constant. This identity makes it possible to recover both signs and then generate a dense space. The resulting families are independent of the Sobolev exponent and work for all degrees in the corresponding saturation.\par\medskip

\noindent These approximation results are motivated in part by finite-dimensional forcing of nonlinear evolution equations. In the geometric control approach of Agrachev and Sarychev \cite{AS05,AS06}, short-time operations generate
new spatial directions through the nonlinearity. We use the density theorem to implement this principle for Schr\"odinger, Klein--Gordon and parabolic equations on compact manifolds. The approximation problem is treated first, independently of these dynamical applications. Throughout this paper, we consider manifolds of dimension greater than one. \footnote{The one-dimensional case has particular features and admits direct trigonometric arguments that differ from the
geometric approach developed in this paper.}\par\medskip\color{black}

\subsection{Density of saturations: Manifolds without boundary} 
In this subsection, $(M,g)$ denotes a smooth compact connected Riemannian manifold without boundary of dimension $d:=\dim M>1$. To define the two saturations, let
\begin{equation}\label{eq:Control_Direction_Space}
\begin{aligned}
&E^\R \text{ be a finite-dimensional real vector space containing }1
\text{ and contained in }C^\infty(M;\R),\\
&E:=E^\R\otimes_\R\C\subset C^\infty(M;\C).
\end{aligned}
\end{equation}
Fix an integer $p\geq2$. Starting from $E^\R$, we define its \emph{power saturation} through the recursion:
\begin{equation}
\label{eq:odd-saturation}
\begin{aligned}
&\mathsf E_0^\R:=E^\R,\\
&\mathsf E_{n+1}^\R:=\operatorname{span}_{\R}\left(\mathsf E_n^\R\cup \{\varphi^{p}:\varphi\in \mathsf E_n^\R\}
\right),\qquad n\ge0,\\
& \mathsf E_n:=\mathsf E_n^\R\otimes_\R\C,\qquad n\ge0,\\
&\mathsf E_\infty^\R :=\bigcup_{n\ge0}\mathsf E_n^\R,\qquad \mathsf E_\infty:=\mathsf E_\infty^\R\otimes_\R\C.
\end{aligned}
\end{equation}

\noindent For even $p$, we also define the convex \emph{cone saturation}, which records the restriction to one sign of the newly generated powers. Fix $r\in\R$. Starting from $E^\R$, we define recursively the following cones and spaces: 
\begin{equation}
\label{eq:even-saturation}
\begin{aligned}
&\mathsf F_0^\R:=E^\R,\\
& \mathsf C_n:=\Big\{\eta-\sum_{j=1}^N\varphi_j^p:N\geq1,\quad \eta,\varphi_1,\ldots,\varphi_N\in \mathsf F_n^\R\Big\}, \quad \mathsf F_{n+1}^\R:=\overline{\mathsf C_n}^{H^r(M;\R)}\cap\left(-\overline{\mathsf C_n}^{H^r(M;\R)}\right),\\
& \mathsf F_n:=\mathsf F_n^\R\otimes_\R\C,\qquad n\ge0,\\
&\mathsf F_\infty^\R :=\bigcup_{n\ge0}\mathsf F_n^\R,\qquad \mathsf F_\infty:=\mathsf F_\infty^\R\otimes_\R\C.
\end{aligned}
\end{equation}
Observe that $\mathsf F_{n+1}^\R$ is the largest vector subspace contained in $\overline{\mathsf C_n}^{H^r(M;\R)}$. We call $\mathsf E_\infty^\R$ and $\mathsf F_\infty^\R$ the power and cone saturations of $E^\R$, respectively. An initial family is called \emph{saturating} in a given function space if its corresponding saturation is dense in that space. With this in mind, we state one of the main results of this work. Given $V_1,\ldots,V_m\in C^\infty(M;\R)$, define the auxiliary function
\begin{equation}\label{eq:V_m+1}
V_{m+1}:=\sqrt{R^2-\sum_{j=1}^{m}V_j^2}, \qquad R^2>\max_{x\in M}\sum_{j=1}^{m}V_j^2.
\end{equation}
Such a choice of $R$ is always possible since the manifold $M$ is compact. Moreover, $V_{m+1}$ is smooth.
\begin{theorem}[Density of power and cone saturations]
\label{th:odd-saturation-density}
Assume that the map
\[\mathcal V:M\to\R^m, \qquad \mathcal V(x):=(V_1(x),\dots,V_m(x)),\]
is a smooth embedding. For every $r\in\R$, the following assertions hold.
\begin{enumerate}
\item The spaces $\mathsf E_\infty^\R$ and $\mathsf E_\infty$ defined  by \eqref{eq:odd-saturation}, with $\mathsf E_0^\R=\operatorname{span}_{\R}\{1,V_1,\dots,V_m\}$ satisfy
\[\overline{\mathsf E_\infty^\R}^{H^r(M;\R)}=H^r(M;\R),
\qquad\overline{\mathsf E_\infty}^{H^r(M;\C)}=H^r(M;\C).\]
\item If $p$ is even, the spaces $\mathsf F_\infty^\R$ and $\mathsf F_\infty$ defined by
\eqref{eq:even-saturation}, with $\mathsf F_0^\R=\operatorname{span}_{\R}\{1,V_1,\dots,V_{m+1}\}$ satisfy
\[\overline{\mathsf F_\infty^\R}^{H^r(M;\R)}=H^r(M;\R),
\qquad \overline{\mathsf F_\infty}^{H^r(M;\C)}=H^r(M;\C).\]
\end{enumerate}
\end{theorem}
The proof is given in \Cref{sec:prelim}. Note that such a smooth embedding always exists by Whitney's embedding theorem (see for instance \cite[Theorem 6.15]{LeeSM}). In \Cref{sec:prelim}, we provide examples of initial families based on smooth embeddings. Once an embedding has been chosen, the functions $1,V_1,\dots,V_m$ work for every degree in \eqref{eq:odd-saturation}, and the single enlargement \eqref{eq:V_m+1} works for every even degree in \eqref{eq:even-saturation}.\par\medskip

This theorem has to be linked with well-known results concerning density of subalgebras of functions defined on manifolds.

\begin{theorem}[Stone--Weierstrass]
\label{thm:StoneWeierstrass}
Let $\mathcal A\subset C(M)$ be a real subalgebra. Assume that $1\in\mathcal A$, and that $\mathcal A$ separates points on $M$, namely, for every $x\ne y$ there exists $f\in\mathcal A$ such that $f(x)\ne f(y)$. Then
\[\overline{\mathcal A}^{\|\cdot\|_\infty}=C(M).\]
\end{theorem}
Observe that \Cref{thm:StoneWeierstrass} (see \textit{e.g.} \cite[Theorem~4.45]{FollandRA}) only provides density in $C(M)$. In the proof of \Cref{th:odd-saturation-density}, we shall need density in $C^r(M)$ (or even in $H^r(M)$), which is provided by Nachbin's theorem (see \cite{Nachbin1949} for the case $\partial M =\emptyset$, and for instance \cite{CG20} in the case $\partial M \neq \emptyset$).
\begin{theorem}[Nachbin]\label{thm:Nachbin}
Let $M$ be a smooth compact manifold with or without boundary. Let $\mathcal A\subset C^\infty(M)$ be a real unital subalgebra. Then $\mathcal A$ is dense in $C^\infty(M)$ if and only if the following conditions are satisfied
\begin{enumerate}[(i)]
\item for any $x,y\in M$, $x\neq y$, there exists $f\in \mathcal A$ such that $f(x)\neq f(y)$;
\item for any $x\in M$ and any $\xi\in T_xM$, $\xi\neq 0$, there exists $f\in \mathcal A$ such that $d_x f(\xi)\neq 0.$
\end{enumerate}
\end{theorem}
\noindent Let $\mathcal A$ be the real unital algebra generated by
$1,V_1,\dots,V_m$. On a compact manifold, the components of a smooth
embedding satisfy both requirements of \Cref{thm:Nachbin}, so $\mathcal A$
is dense in $C^\infty(M)$ and consequently in every $H^r(M)$.\par\medskip

\noindent For \eqref{eq:odd-saturation}, polarization shows that closure
under a fixed power, together with linear combinations and the constant
function, generates all products of the initial functions. Thus
$\mathsf E_\infty^\R=\mathcal A$. The cone saturation
\eqref{eq:even-saturation} requires an additional argument: one must
recover both signs before using linear combinations of the generated
powers. The auxiliary function \eqref{eq:V_m+1} supplies the identity
\[\sum_{j=1}^{m+1}V_j^2=R^2.\]
This identity is the geometric analogue of $\sin^2\theta+\cos^2\theta=1$, used in trigonometric saturation arguments on the torus; see \cite{Ner21}. In the quadratic case, the cone contains the directions $-V_j^2$, and the constant function allows us to recover the opposite direction through
\[V_i^2=R^2+\sum_{\substack{1\leq j\leq m+1\\j\neq i}}(-V_j^2).\]
For a general even degree, a rescaling around a constant first produces
quadratic directions. Iterating the recovery of signs then produces the
algebra $\mathcal A$; the details are given in \Cref{lem:algebra-odd-sat}. The proof also explains why the closures in \eqref{eq:even-saturation} do not introduce a dependence on $r$. At each finite stage, the cone is contained in a finite-dimensional space of smooth functions, on which all Sobolev norms induce the same topology. In fact, both saturations are dense in $C^\infty(M)$ itself. The Sobolev formulation in \Cref{th:odd-saturation-density} is convenient for the applications below. For the power saturation defined by \eqref{eq:odd-saturation}, the embedding condition is also necessary for density in $C^\infty(M)$. More precisely, \Cref{prop:power-saturation-characterization} characterizes the finite families containing $1$ for which $\mathsf E_\infty^\R$ is dense, and \Cref{cor:stable-minimal-generators} records stability under small $C^1$ perturbations and the relation with the embedding dimension of $M$.\\

\noindent We also obtain finite families of Laplace--Beltrami eigenfunctions with dense power saturation in every Sobolev space. Let $(\phi_k)_{k\geq0}$ be a real orthonormal basis of $L^2(M)$ satisfying
\[-\Delta_g\phi_k=\lambda_k\phi_k,\qquad 0=\lambda_0<\lambda_1\leq\lambda_2\leq\cdots,\]
with $\phi_0$ constant. A sufficiently accurate finite spectral
approximation of the coordinates of an embedding remains an embedding. This gives the following consequence.

\begin{corollary}[Density by finitely many eigenfunctions]
\label{cor:eigenfunction-controls}
There exists $N\in\N$ such that for every integer $p\geq2$ and every $r\in\R$, if
\[E^\R_N:=\operatorname{span}_{\R}\{1,\phi_1,\dots,\phi_N\},\qquad E_N:=E^\R_N\otimes_\R\C,\]
the power saturation $\mathsf E_\infty^\R$ of $E^\R_N$, defined by \eqref{eq:odd-saturation}, is dense in $H^r(M;\R)$,  and its complexification is dense in $H^r(M;\C)$.
\end{corollary}

\subsection[Density of saturations: Manifolds with boundary]{Density of saturations: Manifolds with boundary}
\label{sec:intro-boundary-approximation}
In this section, $(M,g)$ denotes a smooth compact connected Riemannian manifold with boundary. The approximation problem changes when the generating functions must satisfy a boundary condition. Under the Neumann condition, all their differentials annihilate the normal direction at the boundary, so they cannot be the coordinates of a smooth embedding there. We construct a family using boundary and interior profiles and a normal coordinate squared. These profiles generate dense subspaces of the Neumann Sobolev spaces. Set $A_N=-\Delta_{g,N}$, the Neumann realization of $-\Delta_g$, and denote by $H_N^s:=D((I-\Delta_{g,N})^{s/2})$ the corresponding Neumann Sobolev scale (see \Cref{sec:admissible_control_boundary}).

\begin{theorem}[Density in the Neumann Sobolev scale]
\label{th:Neumann-approximation}
Let $(M,g)$ be a smooth compact connected Riemannian manifold with boundary and dimension $d$. There exist a finite-dimensional real space independent of $p$
\[E_N^\R\subset\{f\in C^\infty(\overline M;\R):\partial_\nu f|_{\partial M}=0\}\]
containing $1$, and one additional smooth real-valued profile $\Psi$ satisfying $\partial_\nu\Psi|_{\partial M}=0$, with the following properties. For every $0\leq s<7/2$, $s\neq3/2$:
\begin{enumerate}
\item For every integer $p\geq2$, the space $\mathsf E_{N,\infty}^\R$ defined by \eqref{eq:odd-saturation}, with $\mathsf E_{N,0}^\R=E_N^\R$, satisfies
\[\overline{\mathsf E_{N,\infty}^\R}^{H_N^s(M;\R)}=H_N^s(M;\R).\]
\item For every even integer $p\geq2$, the space $\mathsf F_{N,\infty}^\R$ defined by \eqref{eq:even-saturation}, with $\mathsf F_{N,0}^\R=E_N^\R+\operatorname{span}_{\R}\{\Psi\}$ and closures taken in $H_N^s(M;\R)$, satisfies
\[\overline{\mathsf F_{N,\infty}^\R}^{H_N^s(M;\R)}=H_N^s(M;\R).\]
\end{enumerate}
Both complex-valued density statements follow by complexification.
\end{theorem}
The construction and proof are given in \Cref{sec:admissible_control_boundary}. The family $E_N^\R$ and the additional profile $\Psi$ are independent of the degree $p$ and the Sobolev exponent $s$. The smooth closure of the generated algebra is identified in \Cref{th:Neumann-smooth-closure}, and the additional profile is given in \Cref{cor:Neumann-even-saturation}. \par\medskip
\noindent Dirichlet conditions lead to a different conclusion for power saturation. Powers of functions vanishing on the boundary do not generate new first normal traces. In dimension at least two and for $s>3/2$, this prevents both
saturations from being dense in the zero-trace
Sobolev space $\mathcal X_D^s:=\{u\in H^s(M;\R):u|_{\partial M}=0\}$, starting from any finite family of smooth functions vanishing on
the boundary; see \Cref{prop:Dirichlet-power-obstruction}.\color{black}\par\medskip

\noindent \Cref{tab:saturating-families} summarizes the first part of our main results, concerning saturating families and the Dirichlet obstruction.

\begin{table}[!ht]
\centering
\resizebox{0.9\textwidth}{!}{%
\renewcommand{\arraystretch}{1.25}
\setlength{\tabcolsep}{3.5pt}
\begin{tabular}{
p{0.12\textwidth}
p{0.19\textwidth}
p{0.22\textwidth}
p{0.22\textwidth}
p{0.17\textwidth}}
\hline
\textbf{Space}
& \textbf{Manifold}
& \textbf{Initial family}
& \textbf{Target space}
& \textbf{Result} \\
\hline
$\mathsf E_\infty^\R$ \newline 
$\mathsf F_\infty^\R$
& Boundaryless
& $1,V_1,\ldots,V_m$\newline
$1,V_1,\ldots,V_{m+1}$
& $H^r(M;\R)$, $r\in\R$
& \Cref{th:odd-saturation-density}. \\

\hline
$\mathsf E_{N,\infty}^\R$\newline  $\mathsf F_{N,\infty}^\R$
& With boundary \newline Neumann case
& $E_N^\R$ \newline $E_N^\R+\R\Psi$
& $H_N^s(M;\R)$, $s\ne \tfrac 32$ \newline
$0\le s<\tfrac 72$
& \Cref{th:Neumann-approximation}. \\
\hline 

$\mathsf E_{D,\infty}^\R$ \newline $\mathsf F_{D,\infty}^\R$
& With boundary \newline 
Dirichlet case
& Any finite smooth \newline 
Dirichlet family
& $\mathcal X_D^s$, $s>\tfrac 32$
& \Cref{prop:Dirichlet-power-obstruction}\newline (Not dense). \\
\hline
\end{tabular}}
\vspace{1mm}
\caption{\small Saturating families and the Dirichlet obstruction.}
\label{tab:saturating-families}
\end{table}
\FloatBarrier

The proofs of the preceding density results are given in \Cref{sec:prelim}. We now present their applications to the control of nonlinear evolution equations.

\subsection{Applications to finite-dimensional control of nonlinear PDEs}
In many control problems for PDEs, the control is not allowed to act freely on the whole state space. Typical examples include boundary controls or localized internal controls, where the action is restricted to a prescribed part of the domain or its boundary. Here, we consider a different degenerate situation: the control acts on the whole manifold, but only through finitely many spatial profiles. Thus, the source term is finite-dimensional in space, and the question is whether the nonlinear dynamics can propagate these few controlled directions to a dense set of states.\medskip

\noindent We are interested in small-time approximate controllability by means of finite-dimensional source terms. The strategy is inspired by the geometric control approach introduced by Agrachev and Sarychev for the Navier-Stokes and Euler equations on the torus \cite{AS05,AS06}. Their method relies on large controls acting on small time intervals and shows how a nonlinear equation can propagate the directly controlled modes to new directions. These ideas have their finite-dimensional antecedent on polynomial control systems \cite{JK85}, and have since been extended in several infinite-dimensional directions; see, for instance, the abstract framework developed in \cite{GHHM18}.\medskip

 In our setting, the saturation method consists in generating new control directions through the map $f\mapsto f^p$ and using the density of the resulting spaces to obtain approximate controllability. We use the power saturation \eqref{eq:odd-saturation} or the cone saturation \eqref{eq:even-saturation} in the Sobolev phase space of the equation.\medskip

\noindent Related results have been obtained for several nonlinear PDEs, mainly on periodic domains. Examples include three-dimensional Navier--Stokes equations \cite{Shi06}, the nonlinear heat and Cahn--Hilliard equations \cite{Ner21,HSMT26}, the Kuramoto--Sivashinsky and viscous Burgers equations \cite{Gao22,Shi14,Shi18}, as well as the Benjamin--Bona--Mahony, Kadomtsev--Petviashvili--I, Korteweg--de Vries, Kawahara, Zakharov--Kuznetsov--Burgers and Camassa--Holm equations \cite{Jel23,Jel22,Che23,AM25,Che25,CDM25}. In these periodic domains, the construction of saturating families typically uses the interaction of trigonometric modes.\medskip

\noindent The closest result to our Schrödinger result is \cite{Sar12}, where approximate controllability of the defocusing cubic Schrödinger equation on $\T^d$ is obtained by means of a low-dimensional additive source term. There, the controlled directions are Fourier modes, and the saturation mechanism is based on the algebraic interaction of the cubic nonlinearity with the Fourier structure. Our contribution can be viewed as a geometric counterpart of this result on arbitrary compact manifolds. Instead of using Fourier-arithmetic, we take finitely many smooth profiles whose joint map is an embedding of $M$ together with polynomial density along this embedding. We also mention the bilinear Schrödinger result of \cite{DN24}, where a multiplicative version of the Agrachev–Sarychev strategy is developed for the nonlinear Schrödinger equation on $\T^d$.\medskip

\noindent In contrast, we are not aware of a direct analogue of additive finite-dimensional saturation from \cite{Sar12} for wave-type equations. The closest results are in the bilinear setting, where related saturation ideas have been used for wave-type equations; see, for instance, \cite{Poz24,BPP26}. For the Klein–Gordon equation, the second-order structure requires an additional step: velocity translations are first converted into position translations, and the nonlinear kick then generates new velocity directions \cite{Poz24}.\medskip

\noindent We also emphasize that our argument is local in time and only uses the short-time odd- and even-power expansion; in particular, the sign of the nonlinearity (focusing or defocusing) plays no role in the controllability mechanism. The short-time strategy used below is a concrete instance of the scaling and saturation approach developed in \cite{GHHM18}. In this abstract framework, large controls acting on small time intervals are used to generate new effective directions through the highest-order nonlinear term, under sufficiently strong regularity assumptions ensuring well-posedness and the validity of the saturation procedure, as verified in several examples. To the authors' knowledge, there is no explicit general construction of finite saturating families available for arbitrary compact Riemannian manifolds, beyond the torus, sphere, rectangles and cylinders. For the control applications, our goal is to overcome these limitations and provide explicit control profiles on general compact manifolds ensuring the saturation property.\medskip

\noindent Let $X$ be a Banach phase space of the equation under consideration and let $E$ be the finite-dimensional space of admissible spatial control profiles. We consider a controlled evolution equation, formally
written as
\begin{equation}
\label{eq:general_evolution}
    \dot z(t)=f(z(t),\eta(t)), \qquad z(0)=z_0,
\end{equation}
where $z$ is an $X$-valued state and
$\eta\in L^\infty(0,T;E)$ is the control. Solutions are understood in the appropriate sense for each equation considered below and will be mostly defined locally in time.
\begin{definition}[Small-time approximate controllability]
\label{def:STAC}
Equation \eqref{eq:general_evolution} is said to be small-time approximately controllable (STAC) in $X$ by controls with values in $E$ if, for every $z_0,z_1\in X$, every
$\varepsilon>0$, and every $\varkappa>0$, there exist $T\in(0,\varkappa)$ and $\eta\in L^\infty(0,T;E)$ such that the corresponding solution is defined on $[0,T]$ and satisfies
\[ \|z(T)-z_1\|_X<\varepsilon.\]
\end{definition}

The proof strategies illustrated in \Cref{fig:saturation-mechanism,fig:KG-saturation-mechanism} rely on the following notion of small-time approximately realizable translations. For $\xi\in X$, we write $\tau_\xi(z):=z+\xi$.
\begin{definition}
The translation $\tau_\xi$ is said to be small-time approximately realizable if, for every $z_0\in X$, every $\varepsilon>0$, and every $\varkappa>0$, there exist $T\in(0,\varkappa)$ and $\eta\in L^\infty(0,T;E)$ such that the corresponding solution is defined on $[0,T]$ and satisfies
\[\|z(T)-(z_0+\xi)\|_X<\varepsilon.\]
\end{definition}

With \Cref{th:odd-saturation-density} in hand, we first record some basic consequences of this mechanism in the Sobolev-algebra regime $s>d/2$. In this range, $H^s(M)$ is a Banach algebra, the polynomial nonlinearities are locally Lipschitz in the relevant spaces, and the short-time large-amplitude limits are easier to justify. We emphasize that our control is \emph{finite-dimensional}, in the sense that $E$ has finite dimension. For the Schr\"odinger equation, we use $E=\operatorname{span}_{\C}\{1,V_1,\dots,V_m\}$ for both parities. For the real Klein--Gordon and parabolic equations, we use $E^\R=\operatorname{span}_{\R}\{1,V_1,\dots,V_m\}$ for odd powers and add the profile $V_{m+1}$ for even powers. Recall that as $V_1,\dots,V_m$ are the coordinates of an embedding, we have $m\geq d$ ($m\geq d+1$ in the boundaryless case). \medskip

\subsubsection*{Schrödinger equation}
For the Schr\"odinger equation with gauge-invariant nonlinearity
\begin{equation}
\label{eq:Sch_intro}
\left\{
\begin{aligned}
&i\partial_t u=(-\Delta_g+W)u+\kappa  |u|^{p-1}u+\eta(t,x), \qquad (t,x)\in(0,T)\times M,\\
&u(0)=u_0,
\end{aligned}
\right.
\end{equation}
we obtain the following small-time approximate controllability result. 
\begin{theorem}
\label{thm:intro-NLS-STAC}
Let $(M,g)$ be a smooth connected compact Riemannian manifold without boundary. Let $s>d/2$ and let $p\ge2$ be an integer. Let $W\in C^\infty(M;\R)$ and $\kappa\in\R\setminus\{0\}$. Assume that $\mathcal V=(V_1,\ldots,V_m):M\to\R^m$ is a smooth embedding and set $E=\operatorname{span}_{\C}\{1,V_1,\ldots,V_m\}$. Then the Schr\"odinger equation \eqref{eq:Sch_intro} is small-time approximately controllable in $H^s(M;\C)$ by controls $\eta\in L^\infty(0,T;E)$  in the following cases:
\begin{itemize}
\item $p$ is odd;
\item $p$ is even and $s<p-1$.
\end{itemize}
\end{theorem}
The complete proof of \Cref{thm:intro-NLS-STAC} is given in \Cref{sec:additive-NLS}. The main steps in the cubic case are summarized in \Cref{fig:saturation-mechanism}.

\begin{figure}[ht]
\centering
\resizebox{0.8\textwidth}{!}{%
\begin{tikzpicture}[
    font=\small,
    box/.style={
        draw,
        align=center,
        inner xsep=6pt,
        inner ysep=6pt,
        minimum height=1.65cm,
        text width=4.5cm
    },
    arrow/.style={->, thick, >=stealth}
]
\node[box] (setup) at (0,0) {
    \textbf{Setup}\\
    Cubic Schrödinger equation\\
    controls in $L^\infty(0,T;E)$
};
\node[box] (direct) at (5.5,0) {
    \textbf{Direct translations}\\
    for $\xi\in E$,\\
    $\tau_\xi$ is realizable
};
\node[box] (asymptotic) at (11,0) {
    \textbf{Asymptotic limit}\\
    translate by $\delta^{-1/3}\alpha\varphi$, $\forall\alpha\in\C$, $\varphi$ real\\
    free evolution on $[0,\delta]$\\
    remove translation
};
\node[box] (kick) at (11,-2.9) {
    \textbf{Cubic kick}\\
    $\tau_{\pm\delta^{-1/3}\alpha\varphi}$ realizable $\forall\alpha\in\C$\\
    $\Longrightarrow$ $\tau_{\varphi^3}$ realizable
};
\node[box, text width=4.9cm] (saturation) at (5.5,-2.9) {
    \textbf{Saturation criterion}\\
    $\mathsf E_0=E$\\
    iterate the cubic kick and take\\
    complex linear combinations
};
\node[box] (conclusion) at (0,-2.9) {
    \textbf{Conclusion}\\
    cubic Schrödinger equation\\
    is STAC in $H^s$
};
\draw[arrow] (setup) -- (direct);
\draw[arrow] (direct) -- (asymptotic);
\draw[arrow] (asymptotic) -- (kick);
\draw[arrow] (kick) -- (saturation);
\draw[arrow] (saturation) -- (conclusion);
\end{tikzpicture}}
\caption{\small Structure of the proof for the cubic Schrödinger equation.}
\label{fig:saturation-mechanism}
\end{figure}
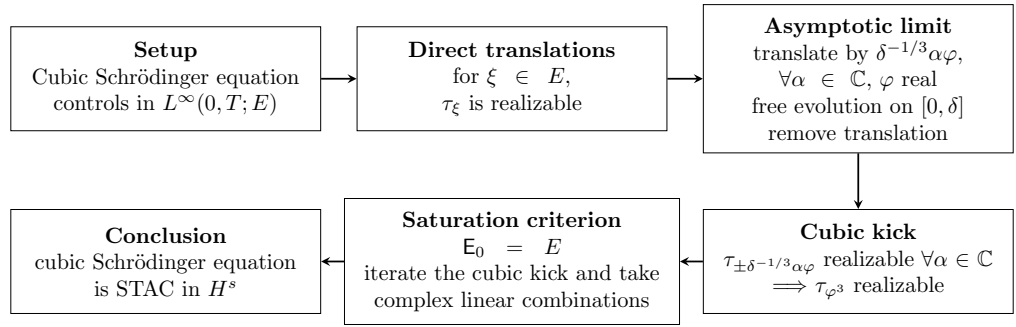

\subsubsection*{Klein--Gordon equation} For the real-valued Klein--Gordon equation with nonlinearity $y^{p}$
\[\left\{
\begin{aligned}
&\partial_t^2 y-\Delta_g y+Wy+\kappa y^{p}=\eta(t,x), \qquad (t,x)\in(0,T)\times M,\\
&(y(0),\partial_t y(0))= (y_0,y_1),
\end{aligned}
\right.\]
we work in $\mathcal H^s:=H^s(M;\R)\times H^{s-1}(M;\R)$, $s>d/2$. 
\begin{theorem}
\label{thm:intro-KG-STAC}
Let $(M,g)$ be a smooth connected compact Riemannian manifold without boundary. Let $s>d/2$ and let $p\ge2$ be an integer. Let $W\in C^\infty(M;\R)$ and $\kappa\in\R\setminus\{0\}$. Assume that $\mathcal V=(V_1,\ldots,V_m):M\to\R^m$ is a smooth embedding, and set

\[E^\R=
\begin{cases}
\operatorname{span}_{\R}\{1,V_1,\ldots,V_m\},& \text{if }p\text{ is odd},\\
\operatorname{span}_{\R}\{1,V_1,\ldots,V_{m+1}\},& \text{if }p\text{ is even}.
\end{cases}
\]
Then the nonlinear Klein--Gordon equation is small-time approximately controllable in $\mathcal H^s$ by controls $\eta\in L^\infty(0,T;E^\R)$.
\end{theorem}
The complete proof of \Cref{thm:intro-KG-STAC} is given in \Cref{sec:additive-KG}. The main steps in the cubic case are given in \Cref{fig:KG-saturation-mechanism}.

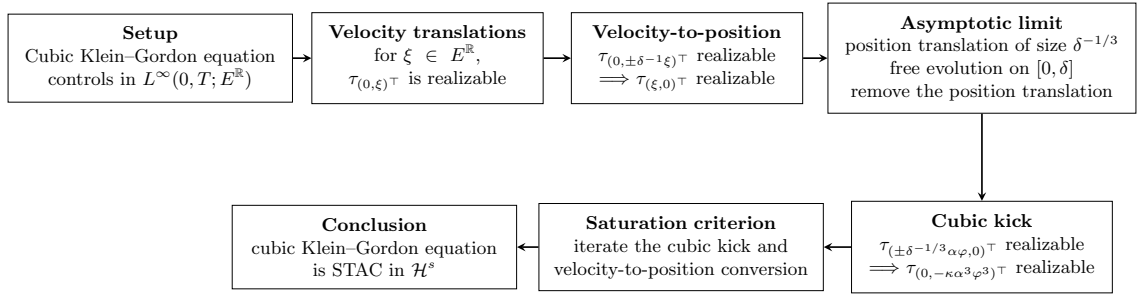
\begin{figure}[ht]
\centering
\resizebox{0.9\textwidth}{!}{%
\begin{tikzpicture}[
    font=\footnotesize,
    box/.style={
        draw,
        align=center,
        inner xsep=6pt,
        inner ysep=6pt,
        minimum height=1.35cm,
        text width=3.5cm
    },
    arrow/.style={
        ->,
        thick,
        >=stealth
    }
]

\node[box, text width=4.4cm] (setup) at (0,0) {
    \textbf{Setup}\\
    Cubic Klein--Gordon equation\\
    controls in $L^\infty(0,T;E^\R)$
};

\node[box] (velocity) at (4.7,0) {
    \textbf{Velocity translations}\\
    for $\xi\in E^\R$,\\
    $\tau_{(0,\xi)^\top}$ is realizable
};

\node[box] (position) at (9.1,0) {
    \textbf{Velocity-to-position}\\
    $\tau_{(0,\pm\delta^{-1}\xi)^\top}$ realizable\\
    $\Longrightarrow$ $\tau_{(\xi,0)^\top}$ realizable
};

\node[box, text width=4.8cm] (asymptotic) at (14.1,0) {
    \textbf{Asymptotic limit}\\
    position translation of size $\delta^{-1/3}$\\
    free evolution on $[0,\delta]$\\
    remove the position translation
};

\node[box, text width=4.0cm] (kick) at (14.1,-3.2) {
    \textbf{Cubic kick}\\
    $\tau_{(\pm\delta^{-1/3}\alpha\varphi,0)^\top}$ realizable\\
    $\Longrightarrow$
    $\tau_{(0,-\kappa\alpha^3\varphi^3)^\top}$ realizable
};

\node[box, text width=4.4cm] (saturation) at (9.0,-3.2) {
    \textbf{Saturation criterion}\\
    iterate the cubic kick and\\
    velocity-to-position conversion
};

\node[box, text width=4.4cm] (conclusion) at (3.8,-3.2) {
    \textbf{Conclusion}\\
    cubic Klein--Gordon equation\\
    is STAC in $\mathcal H^s$
};

\draw[arrow] (setup) -- (velocity);
\draw[arrow] (velocity) -- (position);
\draw[arrow] (position) -- (asymptotic);
\draw[arrow] (asymptotic) -- (kick);
\draw[arrow] (kick) -- (saturation);
\draw[arrow] (saturation) -- (conclusion);

\end{tikzpicture}%
}
\caption{\small Structure of the proof for the cubic Klein--Gordon equation.}
\label{fig:KG-saturation-mechanism}
\end{figure}

\subsubsection*{Parabolic equation} We finally point out that the same additive saturation mechanism applies to parabolic equations, as shown on $\T^d$ in \cite{Ner21}. Let $P\in\R[X]$ be a real polynomial such that
\[P(\lambda)\to+\infty \qquad\text{as }\lambda\to+\infty,\]
and set $A:=P(-\Delta_g)+W$. We consider the controlled parabolic equation with polynomial nonlinearity
\begin{equation*}
\left\{
\begin{aligned}
&\partial_t y+Ay+\kappa y^{p}=\eta(t,x),\qquad (t,x)\in(0,T)\times M,\\
&y(0)=y_0.
\end{aligned}
\right.
\end{equation*}
\begin{theorem}
\label{thm:heat-STAC}
Let $(M,g)$ be a smooth connected compact Riemannian manifold without boundary. Let $s>d/2$ and let $p\ge2$ be an integer. Let $W\in C^\infty(M;\R)$ and $\kappa\in\R\setminus\{0\}$. Assume that
$\mathcal V=(V_1,\ldots,V_m):M\to\R^m$ is a smooth embedding, and set 
\[E^\R=
\begin{cases}
\operatorname{span}_{\R}\{1,V_1,\ldots,V_m\},& \text{if }p\text{ is odd},\\
\operatorname{span}_{\R}\{1,V_1,\ldots,V_{m+1}\},& \text{if }p\text{ is even}.
\end{cases}\]
Then the parabolic equation with nonlinearity $y^p$ is small-time approximately controllable in $H^s(M;\R)$ by controls $\eta\in L^\infty(0,T;E^\R)$.
\end{theorem}
The complete proof of \Cref{thm:heat-STAC} is given in \Cref{sec:additive-Heat}. The main steps in the cubic case are summarized in \Cref{fig:saturation-mechanism-parabolic}.

\begin{figure}[H]
\centering
\resizebox{0.8\textwidth}{!}{%
\begin{tikzpicture}[
    font=\small,
    box/.style={
        draw,
        align=center,
        inner xsep=6pt,
        inner ysep=6pt,
        minimum height=1.65cm,
        text width=4.5cm
    },
    arrow/.style={->, thick, >=stealth}
]
\node[box] (setup) at (0,0) {
    \textbf{Setup}\\
    Cubic parabolic equation\\
    controls in $L^\infty(0,T;E^\R)$
};
\node[box] (direct) at (5.5,0) {
    \textbf{Direct translations}\\
    for $\xi\in E^\R$,\\
    $\tau_\xi$ is realizable
};
\node[box] (asymptotic) at (11,0) {
    \textbf{Asymptotic limit}\\
    translation of size $\delta^{-1/3}$\\
    free evolution on $[0,\delta]$\\
    remove translation
};
\node[box] (kick) at (11,-2.9) {
    \textbf{Cubic kick}\\
    $\tau_{\pm\delta^{-1/3}\varphi}$ realizable\\
    $\Longrightarrow$ $\tau_{\varphi^3}$ realizable
};
\node[box] (saturation) at (5.5,-2.9) {
    \textbf{Saturation criterion}\\
    $\mathsf E_0^\R=E^\R$\\
    iterate the cubic kick and\\
    take real linear combinations
};
\node[box] (conclusion) at (0,-2.9) {
    \textbf{Conclusion}\\
    cubic parabolic equation\\
    is STAC in $H^s(M;\R)$
};
\draw[arrow] (setup) -- (direct);
\draw[arrow] (direct) -- (asymptotic);
\draw[arrow] (asymptotic) -- (kick);
\draw[arrow] (kick) -- (saturation);
\draw[arrow] (saturation) -- (conclusion);
\end{tikzpicture}}
\caption{\small Structure of the proof for the cubic parabolic equation.}
\label{fig:saturation-mechanism-parabolic}
\end{figure}
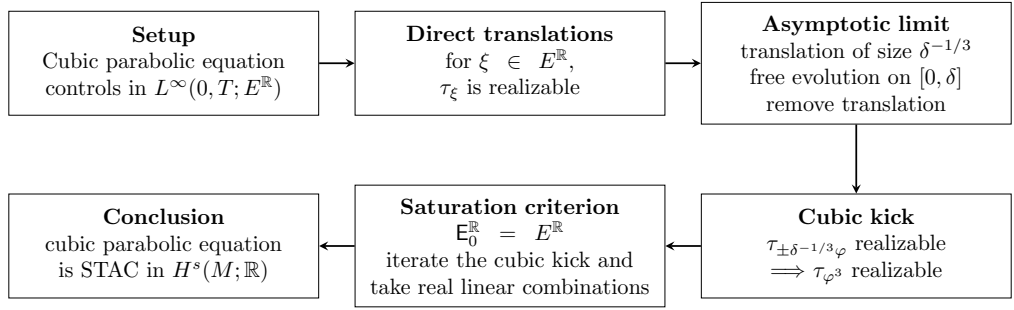

Another purpose of the article is to show that the preceding mechanism is not restricted to the Sobolev-algebra setting or to boundaryless manifolds. We consider compact manifolds with boundary and the main additional difficulties come from the compatibility conditions. We start by analyzing the Neumann case, where the saturation mechanism remains valid by choosing sufficiently smooth control profiles satisfying the boundary conditions. After, we show that the construction provided in the Neumann case is not valid for the Dirichlet boundary conditions by proving some partial obstruction to the saturation property. The approximation theorem is proved in \Cref{sec:admissible_control_boundary}, and its application to the Neumann Schr\"odinger equation is given in \Cref{sec:manifold_boundary}.\medskip

\noindent Finally, we discuss less regular initial data for the cubic Schr\"odinger and Klein--Gordon equations in \Cref{sec:less_regular}. The assumption $s>d/2$ used in the above results is mainly imposed in order to work in a Sobolev-algebra and to handle the nonlinear terms in a straightforward way. We explain how this regularity threshold can be lowered by combining the fixed-point argument with suitable Strichartz estimates, as in the corresponding well-posedness theory.\medskip

\noindent Each case requires a different additional ingredient, but the underlying principle remains the same: large short-time additive translations generate new effective directions through the leading nonlinear term, and the density of the resulting directions is obtained by an algebraic saturation argument. \medskip

\noindent We now summarize the second part of our main results. In \Cref{tab:STAC-algebra}, we present the small-time approximate controllability results on compact manifolds.

\begin{table}[H]
\centering
\resizebox{0.9\textwidth}{!}{%
\renewcommand{\arraystretch}{1.25}
\begin{tabular}{
p{0.15\textwidth}
p{0.16\textwidth}
p{0.17\textwidth}
p{0.27\textwidth}
p{0.13\textwidth}}
\hline
\textbf{Equation}
&\textbf{Manifold}
& \textbf{Nonlinearity}
& \textbf{State space}
& \textbf{Result} \\
\hline
Schr\"odinger
& Boundaryless 
& $|u|^{p-1}u$,
  $p$ odd\newline 
  $|u|^{p-1}u$,
  $p$ even
& $H^s(M;\C)$,
  $s>\tfrac d2$\newline
  $H^s(M;\C)$, $\tfrac d2<s<p-1$
& \Cref{thm:intro-NLS-STAC}. \\
\hline
Schr\"odinger
& With boundary
& $|u|^{p-1}u$, $p$ odd \newline 
  $|u|^{p-1}u$, $p$ even
& $H_N^s$, $s\neq \tfrac 32$,
$\tfrac d2<s<\tfrac 72-\tfrac 2p$\newline
$H_N^s$, $s\neq \tfrac 32$,
$\tfrac d2<s<p^{\sharp}$
& \Cref{th:bdry-manifold}. \\
\hline
Schr\"odinger
& Boundaryless
& $|u|^{2}u$. 
& $H^s(M;\C)$,\newline $s>\tfrac d2-\tfrac 1b$, $b>2$
& \Cref{thm:NLS-STAC2}. \\
\hline
Klein--Gordon
& Boundaryless 
& $y^{p}$
& {\small$H^s(M;\R)\times H^{s-1}(M;\R)$},\newline
  $s>\tfrac d2$
& \Cref{thm:intro-KG-STAC}. \\
\hline
Klein--Gordon
& Boundaryless
& $y^{3}$
& {\small$H^s(M;\R)\times H^{s-1}(M;\R)$},\newline
  $s>\tfrac d2-\tfrac 1b-\tfrac da$, $b>2$, $a>2d$
& \Cref{thm:KG-STAC2}. \\
\hline
Parabolic
& Boundaryless 
& $y^{p}$
& $H^s(M;\R)$,\newline
  $s>\tfrac d2$
& \Cref{thm:heat-STAC}. \\
\hline
\end{tabular}}
\caption{\small Small-time approximate controllability results on compact manifolds. Here $p^{\sharp}=\min\{7/2-2/p, p-1\}$. The parameters $a$ and $b$ are the Strichartz admissible exponents for the Schrödinger equation \eqref{eq:Strichartz}  and for the Klein--Gordon equation \eqref{eq:wave-Strichartz}.}
\label{tab:STAC-algebra}
\end{table}

\subsection{Random forcing}

\noindent As a final application of our density and deterministic controllability results, we derive two consequences for randomly forced equations: topological irreducibility and almost sure unboundedness of Sobolev norms. The relation between deterministic controllability and support properties of randomly forced equations is well-known; see, for instance, \cite{Barbu06,GHHM18}. For probabilistic consequences of low-mode controllability and saturation in nonlinear PDEs, we refer to \cite{AKSS07,Gao22,KNS20,Che25}. Here the randomness acts on the temporal coefficients of finitely many fixed spatial profiles. On each unit time interval, we randomly choose new time-dependent coefficients, with the same probability law and independently of the choices on the previous intervals.\medskip

\noindent Topological irreducibility means that every nonempty open subset of the phase space can be reached with positive probability from any initial state. Here, this holds at the end of a single forcing block. In the block-noise setting considered here, the passage to topological irreducibility is direct: the full support of the forcing law and the continuity of the resolving map imply that every deterministic controlled trajectory can be approximated with positive probability. \medskip

Almost sure unboundedness means that, for each initial state, the Sobolev norm of the solution reaches arbitrarily large values along almost every realization of the forcing. The proof uses compactness in a weaker Sobolev topology to obtain a uniform positive probability of leaving each ball in the stronger space. The Markov property then allows this estimate to be iterated over successive forcing blocks. The use of controllability to obtain almost sure unboundedness of Sobolev norms was developed for randomly perturbed Schr\"odinger equations in \cite{Ner09,DN24} through a compactness and stopping-time argument. We adapt this mechanism to finite-dimensional additive forcing for the Schr\"odinger and Klein--Gordon equations in the cubic defocusing case where we can ensure global well-posedness on compact manifolds without boundary. 

\begin{table}[H]
\centering
\resizebox{0.9\textwidth}{!}{%
\renewcommand{\arraystretch}{1.25}
\begin{tabular}{p{0.15\textwidth} p{0.14\textwidth} p{0.26\textwidth} p{0.32\textwidth}}
\hline
\textbf{Equation} 
& \textbf{Manifold} 
& \textbf{Hypotheses}
& \textbf{Result}\\
\hline
Schr\"odinger\newline
Klein--Gordon
& Boundaryless 
& {\small Global well-posedness\newline
  Controllability, full support}
& \Cref{prop:irreducibility} \newline Topological irreducibility. \\
\hline
Schr\"odinger\newline
Klein--Gordon
& Boundaryless 
& {\small Global well-posedness\newline
  Controllability, full support\newline
  Compatible indices $s_0<s$}
& \Cref{thm:unbounded-Sobolev-trajectories} \newline Unbounded Sobolev trajectory. \\
\hline
\end{tabular}}
\caption{\small Stochastic consequences for the cubic defocusing Schr\"odinger and Klein--Gordon equations on compact manifolds without boundary.}
\label{tab:extensions}
\end{table}
\FloatBarrier

\subsection{Open questions} 
The extensions considered above are not exhaustive. Other variants could be considered, for instance by combining the situations discussed above.
\begin{itemize}
\item A first question is which nonlinear maps, beyond the power maps considered here, generate dense spaces from finitely many smooth functions. For the cone saturation, one may also ask which initial families yield density and when the auxiliary profile can be omitted. A related question is how these approximation results can be used to prove controllability for other nonlinear PDEs, by identifying short-time limits that produce the corresponding nonlinear maps.

\item A second direction concerns the geometry of the underlying space. The present work is set on compact manifolds, with some extensions to manifolds with boundary. It would be natural to investigate non-compact manifolds, as well as more general boundary conditions, where both the density argument and the well-posedness theory may require substantial modifications.

\item A final direction concerns the functional framework. We have kept the basic arguments in Sobolev spaces above the algebra threshold, except when Strichartz estimates are explicitly used. It would be interesting to determine sharp regularity thresholds for the saturation mechanism, both in the dispersive case and in the parabolic case.
\end{itemize}
\section{Density of saturation spaces}
\label{sec:prelim}

\subsection{Polynomial approximation and algebraic saturation}
\label{sec:closed-approximation}

In this subsection, let $(M,g)$ be a smooth compact connected Riemannian manifold without boundary and set
$d:=\dim M$. Let $-\Delta_g$ be the nonnegative Laplace--Beltrami operator. It is self-adjoint on $L^2(M)$ with compact resolvent. Hence, there exist eigenvalues $(\lambda_k)_{k\ge0}$ and an orthonormal basis $(\phi_k)_{k\ge0}$ of $L^2(M)$ consisting of real-valued functions such that
\[-\Delta_g\phi_k=\lambda_k\phi_k; \qquad 0=\lambda_0<\lambda_1\le\lambda_2\le\cdots.\]
where $\phi_0$ is constant. For $r\in\R$, define
\[H^r(M)=\{u \in \mathcal{D}'(M): \ \|u\|_{H^r(M)}<\infty\}, \qquad \text{ where } \qquad
\|u\|_{H^r(M)}^2:=\sum_{k\ge0}(1+\lambda_k)^r
|\langle u,\phi_k\rangle_{L^2}|^2.\]
By Whitney's embedding theorem \cite[Theorem 10.15]{LeeSM}, there exist $m\in\N$ and functions
$V_1,\dots,V_m\in C^\infty(M;\R)$ such that
\[\mathcal V:M\to\R^m, \qquad \mathcal V(x):=(V_1(x),\dots,V_m(x)),\]
is a smooth embedding. We shall use the following elementary consequence of
\Cref{thm:Nachbin}.

\begin{lemma}[Polynomial density along an embedding]
\label{lem:polynomial-density-embedding}
Let $V_1,\dots,V_m\in C^\infty(M;\R)$ and assume that $\mathcal V:M\to\R^m$, $\mathcal V(x):=(V_1(x),\dots,V_m(x))$, is a smooth embedding. Let $\mathcal A$ be the real unital algebra generated by
$\{1,V_1,\dots,V_m\}$. Then, for every $r\in\R$,
\[\overline{\mathcal A}^{H^r(M;\R)}=H^r(M;\R).\]
Consequently,
\[\overline{\mathcal A\otimes_\R\C}^{H^r(M;\C)}=H^r(M;\C).\]
\end{lemma}

\begin{proof}
From \Cref{thm:Nachbin}, $\mathcal A$ is dense in $C^\infty(M)$. We conclude by using the fact that $C^\infty(M)$ is dense in $H^r(M)$. The complex-valued statement follows immediately from the real-valued one.
\end{proof}

\noindent We now prove the algebraic inclusions needed to apply \Cref{lem:polynomial-density-embedding}. We begin with the polarization identity.
\begin{lemma}[Polarization]
\label{lem:odd-polarization}
Let $p\geq2$ be an integer and let $f_1,\ldots,f_p$ be real-valued functions. Then
\[f_1\cdots f_p=\frac{1}{2^p p!}\sum_{\varepsilon\in\{-1,1\}^p}\left(\prod_{j=1}^p\varepsilon_j\right)\left(\sum_{j=1}^p\varepsilon_j f_j\right)^p, \qquad \varepsilon=(\varepsilon_1,\ldots,\varepsilon_p)\in\{-1,1\}^p.\]
Consequently, if $f_1,\ldots,f_p\in{\mathsf E_n^\R}$, then $f_1\cdots f_p\in{\mathsf E_{n+1}^\R}$.
\end{lemma}
\begin{proof}
The identity follows by expanding the powers in the sum. Since $\mathsf E_n^\R$ is a vector space, for every $\varepsilon\in\{-1,1\}^p$ the function 
\[\sum_{j=1}^p\varepsilon_jf_j\]
belongs to ${\mathsf E_n^\R}$. By definition, its $p$-th power belongs to ${\mathsf E_{n+1}^\R}$. The polarization identity therefore gives $f_1\cdots f_p\in{\mathsf E_{n+1}^\R}$.
\end{proof}

\begin{lemma}
\label{lem:algebra-odd-sat}
Let $p\ge2$ be an integer and let $\mathcal A$ be the real unital
algebra generated by $1,V_1,\dots,V_m$.
\begin{itemize}
\item The space $\mathsf E_\infty^\R$ defined by \eqref{eq:odd-saturation}, with $\mathsf E_0^\R=\operatorname{span}_{\R}\{1,V_1,\dots,V_m\}$,  satisfies $\mathcal A\subset\mathsf E_\infty^\R$.
\item If $p$ is even, the space $\mathsf F_\infty^\R$ defined by \eqref{eq:even-saturation}, with $\mathsf F_0^\R=\operatorname{span}_{\R}\{1,V_1,\dots,V_{m+1}\}$, where $V_{m+1}$ is given by \eqref{eq:V_m+1}, satisfies $\mathcal A\subset\mathsf F_\infty^\R$.
\end{itemize} 
\end{lemma}
\begin{proof}
\noindent\textit{Power saturation.}
First, we prove that $\mathsf E_\infty^\R$ is stable under pointwise multiplication. Let $f,g\in\mathsf E_\infty^\R$. Since the spaces $\mathsf E_n^\R$ are nested, there exists $n\geq0$ such that $f,g\in\mathsf E_n^\R$. Moreover, $1\in\mathsf E_0^\R\subset\mathsf E_n^\R$.  Applying \Cref{lem:odd-polarization} to $(f_1,\dots,f_p)=(f,g,1,\ldots,1)$, we obtain $fg\in{\mathsf E_{n+1}^\R\subset\mathsf E_\infty^\R}$. Since $\mathsf E_\infty^\R$ contains $1,V_1,\dots,V_m$, it follows that it contains all monomials
\[V_1^{\alpha_1}\cdots V_m^{\alpha_m},\qquad \alpha_j\in\N.\]
By linearity, it contains every finite linear combination of such monomials. Hence $\mathcal A\subset\mathsf E_\infty^\R$.
\medskip

\noindent\textit{Cone saturation.}
First, the spaces $\mathsf F_n^\R$ are nested and contain $1$. Moreover, they are finite-dimensional spaces of smooth functions. Indeed, by induction, $\mathsf C_n$ is contained in the finite-dimensional space spanned by $\mathsf F_n^\R$ and the products of $p$ basis elements, and its closure remains in that space. Since all Sobolev norms induce the same topology on this finite-dimensional space, the cone saturation is independent of $r$. We first show that, for every $\varphi\in\mathsf F_n^\R$,
\[-\varphi^2\in\overline{\mathsf C_n}^{H^r(M;\R)}.\]
For $p=2$, this follows directly from the definition of $\mathsf C_n$. For even $p\geq4$, the binomial formula gives
\[\begin{aligned}
\varepsilon^{-\frac{2p}{p-2}}+p\varepsilon^{-\frac{p}{p-2}}\varphi-\left(\varepsilon^{-\frac{2}{p-2}}+\varepsilon\varphi\right)^p&=-\frac{p(p-1)}2\varphi^2-\sum_{k=3}^p\binom{p}{k}\varepsilon^{\frac{p(k-2)}{p-2}}\varphi^k\\
&=-\frac{p(p-1)}2\varphi^2+O_{H^r}\left(\varepsilon^{\frac{p}{p-2}}\right).
\end{aligned}\]
Thus
\[\varepsilon^{-\frac{2p}{p-2}}+p\varepsilon^{-\frac{p}{p-2}}\varphi-\left(\varepsilon^{-\frac{2}{p-2}}+\varepsilon\varphi\right)^p\longrightarrow-\frac{p(p-1)}2\varphi^2\qquad\text{in }H^r(M;\R)\]
as $\varepsilon\to0^+$. Each expression on the left belongs to $\mathsf C_n$, which proves the claim. Now let $f_1,\ldots,f_N\in\mathsf F_n^\R$ satisfy
$\sum_{j=1}^N f_j^2=c$, where $c$ is constant. Then
\[f_i^2=c+\sum_{j\neq i}(-f_j^2)\in\overline{\mathsf C_n}^{H^r(M;\R)}.\]
Since both signs belong to this closed cone, we obtain
$f_i^2\in\mathsf F_{n+1}^\R$. Applying the preceding claim at the next stage gives
\[\mp2f_if_j=f_i^2+f_j^2-(f_i\pm f_j)^2\in\overline{\mathsf C_{n+1}}^{H^r(M;\R)}.\]
Hence $f_if_j\in\mathsf F_{n+2}^\R$. Moreover,
\[\left(\sum_{j=1}^Nf_j^2\right)^2=\sum_{i,j=1}^N(f_if_j)^2=c^2,\]
so the family of all pairwise products satisfies the same constant-square property. We apply this argument first to the family $1,V_1,\ldots,V_{m+1}$, whose sum of squares is $1+R^2$, and then repeatedly to the indexed families of all pairwise products, retaining repetitions. Thus every product of $2^k$
initial profiles belongs to $\mathsf F_{2k}^\R$. Since the initial family contains $1$, every monomial in $V_1,\ldots,V_m$ belongs to some $\mathsf F_{2k}^\R$. By linearity,
$\mathcal A\subset\mathsf F_\infty^\R$.
\end{proof}

We are now in a position to prove \Cref{th:odd-saturation-density}.
\medskip

\textit{Proof of \Cref{th:odd-saturation-density}: }
By \Cref{lem:algebra-odd-sat}, $\mathsf E_\infty^\R$ contains the real
algebra generated by $1,V_1,\dots,V_m$. For even $p$, the same holds
for $\mathsf F_\infty^\R$. The conclusion follows from
\Cref{lem:polynomial-density-embedding}.
\qed
\medskip

\subsection{Characterization and stability of saturating families}
\label{sec:characterization}
We now show that the embedding condition is also necessary for density in $C^\infty(M)$ and deduce stability under small $C^1$ perturbations of the generating functions.
\begin{proposition}[Characterization of power saturation]
\label{prop:power-saturation-characterization}
Let $M$ be a smooth compact manifold without boundary, let $p\geq2$ be an integer, and let
\[E^\R=\operatorname{span}_{\R}\{1,V_1,\ldots,V_m\}\subset C^\infty(M;\R).\]
The following assertions are equivalent:
\begin{enumerate}[(i)]
\item $\mathcal V=(V_1,\ldots,V_m):M\to\R^m$ is a smooth embedding;
\item the space $\mathsf E_\infty^\R$ is dense in $C^\infty(M;\R)$;
\item $\mathsf E_\infty^\R$ is dense in $H^r(M;\R)$ for every $r\in\R$.
\end{enumerate}
\end{proposition}

\begin{proof}
Let $\mathcal A$ be the real unital algebra generated by $V_1,\ldots,V_m$. By \Cref{lem:algebra-odd-sat}, $\mathcal A\subset\mathsf E_\infty^\R$. Conversely, $\mathcal A$ contains $E^\R$ and is closed under linear combinations and powers, so induction gives $\mathsf E_\infty^\R\subset\mathcal A$. Thus $\mathsf E_\infty^\R=\mathcal A$. Assertion (i) implies (ii) by \Cref{thm:Nachbin}, and (ii) implies (iii) by the density of smooth functions in Sobolev spaces. Assume (iii) and choose $r>\dim M/2+1$, so that $H^r(M)\hookrightarrow C^1(M)$. If $\mathcal V(x)=\mathcal V(y)$ for two distinct points, then every element of $\mathcal A$ has the same values at $x$ and $y$. The continuous functional $u\mapsto u(x)-u(y)$ would therefore vanish on all of $H^r(M)$, a contradiction. If $d_x\mathcal V(\xi)=0$ for some nonzero $\xi\in T_xM$, the chain rule gives $d_x a(\xi)=0$ for every $a\in\mathcal A$. The continuous functional $u\mapsto d_x u(\xi)$ would again vanish on $H^r(M)$, a contradiction. Hence $\mathcal V$ is an injective immersion. Compactness of $M$ makes it a smooth embedding, proving (i).
\end{proof}

The identity $\mathsf E_\infty^\R=\mathcal A$ holds for every initial family containing $1$, without an embedding assumption. It is density with derivatives that forces the initial map to be an embedding. In fact, the proof shows that density in a single $H^r(M)$ with $r>\dim M/2+1$ already suffices.

\begin{corollary}[Stability and number of generators]
\label{cor:stable-minimal-generators}The following assertions holds.

\begin{itemize}
\item Let $\mathcal V=(V_1,\ldots,V_m):M\to\R^m$ be a smooth embedding. Then every smooth perturbation, sufficiently small for the $C^1$--topology, of the profiles $V_1,\ldots,V_m$ has dense power saturation. 
\item Let $e(M)$ be the least integer $m$ for which $M$ admits a smooth embedding into $\R^m$. The smallest possible dimension of a real space containing $1$ whose power saturation is dense in $C^\infty(M;\R)$ is $e(M)+1$.
\end{itemize}
\end{corollary}
\begin{proof}
Smooth embeddings of a compact manifold are stable under sufficiently small $C^1$ perturbations. The first assertion therefore follows from \Cref{prop:power-saturation-characterization}.
If a subspace containing 1 of dimension $k$ has dense power saturation, choose a basis $1,V_1,\ldots,V_{k-1}$. The same proposition gives an embedding into $\R^{k-1}$, so $k\geq e(M)+1$. Conversely, the coordinates of an embedding into $\R^{e(M)}$, together with $1$, span a space of dimension at most $e(M)+1$ whose power saturation is dense in $C^\infty(M;\R)$.
\end{proof}

\subsection{Finite spectral families}
\label{sec:spectral-families}
We now prove \Cref{cor:eigenfunction-controls}, whose statement was given in the introduction.

\begin{proof}[Proof of \Cref{cor:eigenfunction-controls}]
Let $\mathcal V=(V_1,\dots,V_m):M\to\R^m$ be a smooth embedding. Let $\Pi_N$ be the orthogonal projection onto $\operatorname{span}\{\phi_0,\phi_1,\dots,\phi_N\}$. Since $\Pi_N V_j\to V_j$ in $C^1(M)$, by \Cref{cor:stable-minimal-generators}, for $N$ sufficiently large, the map $\mathcal V_N(x):=(\Pi_N V_1(x),\dots,\Pi_N V_m(x))$, is a smooth embedding.
\end{proof}
\subsection{Approximation with Neumann-compatible profiles}
\label{sec:admissible_control_boundary}

\noindent Let $(M,g)$ be a smooth compact connected Riemannian manifold with boundary, and let $A_N=-\Delta_{g,N}+W$ be the Neumann realization of $-\Delta_g+W$, where
$W\in C^\infty(\overline M;\R)$. Choose $c=1+\|W\|_{L^\infty(\overline M)}$; then $cI+A_N\ge I$, and set
\[\Lambda:=(cI+A_N)^{1/2}.\]
We denote by
\[H_N^s:=D(\Lambda^s),\qquad\|u\|_{H_N^s}:=\|\Lambda^s u\|_{L^2(M)},\qquad s\ge0,\]
the corresponding Neumann Sobolev scale. It is classical that (see \textit{e.g.} \cite{Grisvard,Seeley,Grubb}), for $0\leq s<7/2$, $s\neq3/2$, the Neumann scale is characterized by the usual boundary compatibility conditions:
\[H_N^s=H^s(M),\qquad 0\le s<\frac32,\]
whereas
\[H_N^s=\{u\in H^s(M):\partial_\nu u=0\text{ on }\partial M\},\qquad \frac32<s<\frac72.\]
No higher compatibility condition such as $\partial_\nu\Delta_g u=0$ is needed as long as $s<7/2$.  We now construct the finite-dimensional spaces in \Cref{th:Neumann-approximation}, using boundary and interior profiles.\medskip

\noindent Let $\mathfrak r:\partial M\times[0,3\eps)\longrightarrow U\subset M$, $\mathfrak r(y,r):=\exp_y(-r\nu_y)$, be a smooth normal collar neighbourhood of the boundary, with
$0<\eps<1/2$ sufficiently small. The existence of such a collar follows from \cite[Theorem~5.25 and Example~6.44]{Lee2018}. In particular,
\[\partial_r\mathfrak r(y,0)=-\nu_y,\qquad y\in\partial M.\]
For every $x\in U$, we write $\mathfrak r^{-1}(x)=(\pi(x),\rho(x))$, where $\pi(x)$ and $\rho(x)$ denote the corresponding tangential and normal variables, respectively. Let $\varphi^\partial=(\varphi_1^\partial,\ldots,\varphi_\ell^\partial):\partial M\longrightarrow\R^\ell$ be a smooth embedding of the boundary. We choose two smooth cut-off functions depending only on the collar variable 
\[\chi^\partial,\chi^M\in C^\infty([0,\infty);[0,1])\]
such that
\[\chi^\partial(r)=1\quad \text{for } r\leq\eps, \qquad \chi^\partial(r)=0 \quad \text{for } r\geq2\eps,\]
and
\[\chi^M(r)=0\quad\text{for }r\leq\eps/4, \qquad \chi^M(r)>0\quad\text{for }r>\eps/4, \qquad \chi^M(r)=1\quad\text{for }r\geq\eps/2.\]
We define the boundary profiles
\[\Phi_j(x)=
\begin{cases}
\chi^\partial(\rho(x))
\varphi_j^\partial(\pi(x)),&x\in U,\\
0,&x\notin U,
\end{cases} \qquad 1\leq j\leq\ell,\]
and the normal profile
\[R(x)=
\begin{cases}
\chi^\partial(\rho(x))(1+\rho(x)^2),&x\in U,\\
0,&x\notin U.
\end{cases}\]
Also, let $\mathcal V=(V_1,\ldots,V_m):M\longrightarrow\R^m$ be a smooth embedding of the whole manifold with boundary. We define
\[\Theta(x)=
\begin{cases}
\chi^M(\rho(x)),&x\in U,\\
1,&x\notin U,
\end{cases}
\qquad
\Psi_k(x):=\Theta(x)V_k(x),
\qquad 1\leq k\leq m.\]
Note that $\Phi_j$, $R$, $\Theta$ and $\Psi_k$ are smooth on $M$. Finally, set
\begin{equation}
\label{eq:Neumann-control-family}
E_N^\R :=\operatorname{span}_{\R}\{1,\Phi_1,\ldots,\Phi_\ell,R,\Theta, \Psi_1,\ldots,\Psi_m\}.
\end{equation}
We also write $E_N^\C:=E_N^\R\otimes_{\R}\C$. Every element of $E_N^\R$ is Neumann-compatible. Indeed, for
$y\in\partial M$ and $0\leq r\leq\eps/4$, one has
\[\Phi_j(\mathfrak r(y,r))=\varphi_j^\partial(y), \qquad R(\mathfrak r(y,r))=1+r^2, \qquad \Theta(\mathfrak r(y,r))=0, \qquad \Psi_k(\mathfrak r(y,r))=0.\]
Therefore,
\[\partial_\nu\Phi_j=0, \qquad \partial_\nu R=0, \qquad \partial_\nu\Theta=0, \qquad \partial_\nu\Psi_k=0 \quad\text{on }\partial M.\]
Thus,
\[E_N^\R \subset C^\infty(M;\R)\cap H_N^s(M;\R), \qquad 0\leq s<\frac72.\]
We next prove that this family separates points. Consider the map
\[\mathcal G:M\longrightarrow\R^{\ell+m+2}, \qquad \mathcal G(x)
:= \bigl(\Phi_1(x),\ldots,\Phi_\ell(x), R(x),\Theta(x), \Psi_1(x),\ldots,\Psi_m(x)\bigr).\]
Let $x,y\in M$, $x\neq y$ and assume that $\mathcal G(x)=\mathcal G(y)$.

\begin{itemize}
\item If $\Theta(x)=\Theta(y)>0$, then
\[\Theta(x)V_k(x)=\Psi_k(x)=\Psi_k(y)=\Theta(y)V_k(y), \qquad 1\leq k\leq m.\]
Hence $V_k(x)=V_k(y)$ for every $k$. Since $V$ is injective, we obtain $x=y$.

\item If $\Theta(x)=\Theta(y)=0$, then $x,y\in U$ and $\rho(x),\rho(y)\leq\eps/4$. Consequently, for $z=x,y$,
\[\Phi_j(z)=\varphi_j^\partial(\pi(z)), \qquad R(z)=1+\rho(z)^2.\]
The equality $R(x)=R(y)$ yields $\rho(x)=\rho(y)$, while the equalities $\Phi_j(x)=\Phi_j(y)$, $1\leq j\leq\ell$, yield
\[\varphi^\partial(\pi(x))=\varphi^\partial(\pi(y)).\]
Since $\varphi^\partial$ is an embedding of $\partial M$, $\pi(x)=\pi(y)$. Hence $x=y$. 
\end{itemize}

Thus $\mathcal G$ is injective, or equivalently, its components separate the points of $M$.
\medskip

\noindent Let $\mathcal A_N$ be the real polynomial algebra generated by the functions in \eqref{eq:Neumann-control-family}:

\[\mathcal A_N=\Bigl\{P(\Phi_1,\ldots,\Phi_\ell,R,\Theta,\Psi_1,\ldots,\Psi_m): P\in\R[X_1,\ldots,X_{\ell+m+2}]\Bigr\}.\]

Let $\mathsf E_{N,\infty}^\R$ denote the power saturation of $E_N^\R$ of degree $p\geq2$, as in \eqref{eq:odd-saturation}. Since $1\in E_N^\R$, the
polarization identity yields $\mathcal A_N\subset\mathsf E_{N,\infty}^\R$. The reverse inclusion follows because $\mathcal A_N$ contains the initial family and is closed under powers. Thus $\mathsf E_{N,\infty}^\R=\mathcal A_N$.
\begin{theorem}[Smooth closure of the Neumann algebra]
\label{th:Neumann-smooth-closure}
For the fixed normal collar $\mathfrak r$, set
\[\mathcal Y_N:=\left\{u\in C^\infty(\overline M;\R):
\partial_r^{2k+1}(u\circ\mathfrak r)(y,0)=0
\text{ for all }y\in\partial M,\ k\geq0\right\}.\]
The space $E_N^\R$ in \eqref{eq:Neumann-control-family} is contained in $\mathcal Y_N$, and its generated unital algebra satisfies
\[\overline{\mathcal A_N}^{C^\infty(\overline M;\R)}=\mathcal Y_N.\]
\end{theorem}
\color{black}

\begin{lemma}[Sobolev density of the Neumann algebra]
\label{lem:Neumann-sobolev-density}
Let $0\leq s<7/2$, $s\neq3/2$. Then
\[\overline{\mathcal A_N}^{H_N^s(M;\R)}=H_N^s(M;\R)\]
Consequently,
\[\overline{\mathsf E_{N,\infty}^\R}^{H_N^s(M;\R)}=H_N^s(M;\R),\quad \text{and}\quad\overline{\mathsf E_{N,\infty}^\C}^{H_N^s(M;\C)}=H_N^s(M;\C).\]
\end{lemma}
\begin{proof}[Proof of \Cref{th:Neumann-smooth-closure,lem:Neumann-sobolev-density}]
For $\eps>0$ small enough, define
\[\mathcal V_\eps:=\mathfrak r\bigl(\partial M\times[0,\eps)\bigr),\qquad \mathcal U_\eps:=\mathfrak r\bigl(\partial M\times[0,2\eps)\bigr),\]
so that $\overline{\mathcal V_\eps}\subset\mathcal U_\eps$. Choose a smooth strictly increasing function $h:[0,3\eps)\longrightarrow[0,3\eps)$ such that
\[h(r)=r^2\quad\text{for }0\leq r\leq\eps, \qquad h(r)=r\quad\text{for }2\eps\leq r<3\eps, \qquad h'(r)>0\quad\text{for }r>0.\]
Define $D:M\to M$ by
\[D(x)=
\begin{cases}
\mathfrak r(y,h(r)),&x=\mathfrak r(y,r)\in U,\\
x,&x\notin U.
\end{cases}\]
Then $D$ is a homeomorphism of $M$, $D=\operatorname{Id}\quad\text{on }M\setminus\mathcal U_\eps$ and $D|_{M^\circ}:M^\circ\longrightarrow M^\circ$ is a smooth diffeomorphism. Moreover,
\[D(x)=\mathfrak r\bigl(\pi(x),\rho(x)^2\bigr),
\qquad x\in\mathcal V_\eps.\]
In particular, for every $K\Subset M^\circ$, the restriction $D|_K$ is a smooth diffeomorphism from $K$ onto $D(K)$. \medskip

\noindent We introduce the real algebra $\widetilde{\mathcal A}_N:=\bigl\{f\circ D^{-1}:f\in\mathcal A_N\bigr\}$. Set $\widetilde\Phi_j:=\Phi_j\circ D^{-1}$, $\widetilde R:=R\circ D^{-1}$, $\widetilde\Theta:=\Theta\circ D^{-1}$, $\widetilde\Psi_k:=\Psi_k\circ D^{-1}$ and $\widetilde{\mathcal G}:=\mathcal G\circ D^{-1}$. Since composition with $D^{-1}$ commutes with polynomial operations, one has
\[\widetilde{\mathcal A}_N=\Bigl\{P(\widetilde\Phi_1,\ldots,\widetilde\Phi_\ell, \widetilde R,\widetilde\Theta, \widetilde\Psi_1,\ldots,\widetilde\Psi_m): P\in\R[X_1,\ldots,X_{\ell+m+2}]\Bigr\}.\]
We first verify that $\widetilde{\mathcal A}_N\subset C^\infty(M;\mathbb R)$. The only point to check is the regularity at the boundary, since $D^{-1}$ is a smooth diffeomorphism in $M^\circ$. For $r\geq0$ sufficiently small, $D^{-1}(\mathfrak r(y,r))=\mathfrak r(y,\sqrt r)$. Hence, near the boundary,
\[\widetilde\Phi_j(\mathfrak r(y,r))=\varphi_j^\partial(y),\qquad \widetilde R(\mathfrak r(y,r))=1+r, \qquad \widetilde\Theta(\mathfrak r(y,r))=\widetilde\Psi_k(\mathfrak r(y,r))=0.\]
Thus all the generators of $\widetilde{\mathcal A}_N$ are smooth up to the boundary. Since every element of $\widetilde{\mathcal A}_N$ is a polynomial in these
generators, it follows that $\widetilde{\mathcal A}_N\subset C^\infty(M;\mathbb R)$.\medskip

\noindent We would like to apply \Cref{thm:Nachbin} to the unital algebra $\widetilde{\mathcal A}_N$. For this, we verify the following conditions:

\begin{enumerate}[(i)]
\item for every $x,y\in M$, $x\neq y$, there exists $\widetilde f\in\widetilde{\mathcal A}_N$ such that $\widetilde f(x)\neq \widetilde f(y)$;
\item for every $x\in M$ and every $\xi\in T_xM\setminus\{0\}$, there exists $\widetilde f\in\widetilde{\mathcal A}_N$ such that $d_x\widetilde f(\xi)\neq0$.
\end{enumerate}
For \emph{(i)}, we recall that the map $\mathcal G$ is injective and that $D$ is a homeomorphism. In particular, $D^{-1}$ is injective, therefore $\widetilde{\mathcal G}:=\mathcal G\circ D^{-1}$ is injective.  \medskip

\noindent It remains to verify condition \emph{(ii)}. Let
$x\in M$ and $\xi\in T_xM\setminus\{0\}$, and set
$z:=D^{-1}(x)$. We distinguish two cases.

\medskip

\noindent\textbf{Case 1: $z\in\mathcal V_\eps$.}
Equivalently, $x\in D(\mathcal V_\eps)=\mathfrak r\bigl(\partial M\times[0,\eps^2)\bigr)$.
Write $x=\mathfrak r(y,r)$, with $0\leq r<\eps^2$, and $\xi=d_{(y,r)}\mathfrak r(\xi',\xi_d)$, $\xi'\in T_y\partial M$, $\xi_d\in\R$. Since
\[D^{-1}(\mathfrak r(y,r))=\mathfrak r(y,\sqrt r),\]
we have
\[\widetilde\Phi_j(\mathfrak r(y,r))=\varphi_j^\partial(y), \qquad \widetilde R(\mathfrak r (y,r))=1+r.\]
If $\xi_d\neq0$, then $d_x\widetilde R(\xi)=\xi_d\neq0$. If $\xi_d=0$, then $\xi'\neq0$. Since $\varphi^\partial$ is an immersion, there exists $1\leq j\leq\ell$ such that $d_y\varphi^\partial_{j}(\xi')\neq0$. Consequently, $d_x\widetilde\Phi_{j}(\xi)\neq0$.

\medskip

\noindent\textbf{Case 2: $z\notin\mathcal V_\eps$.}
In this case, $\Theta(z)\equiv1$ near $z$, and therefore
\[\widetilde\Theta(x)=1, \qquad \widetilde\Psi_k(x)=V_k(z)=(V_k\circ D^{-1})(x), \qquad 1\leq k\leq m.\]
Moreover, $x,z\in M^\circ$, and $D^{-1}$ is a smooth diffeomorphism in the interior. Since $\mathcal V$ is an immersion,
$\mathcal V\circ D^{-1}$ is also an immersion. Hence there exists $1\leq k\leq m$ such that $d_x\widetilde\Psi_{k}(\xi)\neq0$.\medskip

\noindent The two cases cover $M$ because $D$ is a homeomorphism.
Therefore condition \emph{(ii)} holds.\par\medskip

\noindent \Cref{thm:Nachbin} therefore yields
\[\overline{\widetilde{\mathcal A}_N}^{C^\infty(M;\mathbb R)}=C^\infty(M;\mathbb R).\]
We now return to the original variables. We set
\[X:=\left\{u\in C^\infty(M;\mathbb R):u\circ D^{-1}\in C^\infty(M;\mathbb R)\right\}.\]
Let $u\in X$ and set $\widetilde u:=u\circ D^{-1}$. By the density of $\widetilde{\mathcal A}_N$ in $C^\infty(M;\mathbb R)$, there exists a sequence $(f_n)_{n\geq1}\subset\mathcal A_N$ such that
\[f_n\circ D^{-1}\longrightarrow \widetilde u \qquad\text{in }C^\infty(M;\mathbb R).\]
We have
\[f_n-u=\bigl(f_n\circ D^{-1}-\widetilde u\bigr)\circ D.\]
Although $D$ is not a diffeomorphism at the boundary, it is a smooth map on $M$. Therefore, its pullback
\[ D^*:C^\infty(M;\mathbb R)\longrightarrow C^\infty(M;\mathbb R),\qquad D^*v:=v\circ D,\]
is continuous. Therefore
\[f_n\longrightarrow u\qquad\text{in }C^\infty(M;\mathbb R).\]
Consequently, $\mathcal A_N$ is dense in $X$ for the topology induced
by $C^\infty(M;\mathbb R)$; equivalently,
\[X\subset\overline{\mathcal A_N}^{C^\infty(M;\mathbb R)}.\]
Let us define 
\[ Y:= \left\{ u\in C^\infty (M;\mathbb R): \partial_\nu^{2k+1}u\big|_{\partial M}=0, \quad \text{for all } k\geq0 \right\}. \] 
Let us show that $X=Y$.\medskip

\noindent We first prove that $X\subset Y$.   Let $u\in X$ and set $\widetilde u:=u\circ D^{-1}\in C^\infty(M;\mathbb R)$. Since $u=\widetilde u\circ D$, in the normal collar we have
\[u\bigl(\mathfrak r(y,r)\bigr)=\widetilde u\bigl(\mathfrak r(y,r^2)\bigr).\]
The right-hand side is a smooth even function of $r$. Therefore,
\[\partial_r^{2k+1}(u\circ\mathfrak r)(y,0)=0\qquad\text{for every }k\geq0.\]
Since $\mathfrak r$ is a normal collar, this yields
\[\partial_\nu^{2k+1}u\big|_{\partial M}=0\qquad\text{for every }k\geq0.\]
Thus $u\in Y$.

\medskip
\noindent We now prove that $Y\subset X$. Let $u\in Y$. Since $D^{-1}$ is smooth in the interior, it is enough to prove that $u\circ D^{-1}$ is smooth near the boundary. Fix $y_0\in\partial M$, and let $V\subset\partial M$ be an open neighbourhood of $y_0$. Let $\varphi:V\longrightarrow W\subset\mathbb R^{d-1}$ be a boundary chart containing $y_0$. Choose an open set $V_0\Subset V$ containing $y_0$, and set $W_0:=\varphi(V_0)$. Let  $\chi_1\in C_c^\infty((-\varepsilon,\varepsilon);[0,1])$, $\chi_2\in C_c^\infty(W;[0,1])$ be such that 
\[\chi_1 \text{ even}, \quad \chi_1=1\quad\text{on }(-\varepsilon_0,\varepsilon_0), \qquad \chi_2=1\quad\text{on }W_0\]
for some $\varepsilon_0\in(0,\varepsilon)$. For $(z,r)\in W\times[0,\varepsilon)$, set
\[u_\varphi(z,r):=u\bigl(\mathfrak r(\varphi^{-1}(z),r)\bigr),\]
and define $U$ on $\mathbb R^{d-1}\times[0,\infty)$ by
\[U(z,r):=
\begin{cases}
\chi_2(z)\chi_1(r)u_\varphi(z,r),& (z,r)\in W\times[0,\varepsilon),\\
0,& \text{otherwise}.
\end{cases}\]
We then define the even extension
\[\widehat U(z,r):=U(z,|r|),\qquad (z,r)\in\mathbb R^{d-1}\times\mathbb R.\]
Since $u\in Y$ and $\mathfrak r$ is a normal collar, all the odd
derivatives of $U$ with respect to $r$ vanish at $r=0$. Therefore, $\widehat U\in C_c^\infty(\mathbb R^{d-1}\times\mathbb R)$. Moreover, $\widehat U$ is invariant under the action
\[(z,r)\longmapsto(z,-r).\]
The polynomial invariants of this action are
\[z_1,\ldots,z_{d-1},r^2.\]
Hence, by Whitney--Schwarz's theorem \cite[Theorem~1]{Schwarz1975}, applied to the action of $\mathbb Z_2$ given by $(y,r)\mapsto(y,-r)$, there exists $H\in C^\infty(\mathbb R^{d-1}\times\mathbb R)$ such that $\widehat U(z,r)=H(z,r^2)$. For $z\in W_0$ and $0\leq r<\varepsilon_0$, we have
$\chi_1(r)=\chi_2(z)=1$, and therefore $u_\varphi(z,r)=H(z,r^2)$. Consequently, for $z\in W_0$ and $0\leq\rho<\varepsilon_0^2$,
\[\begin{aligned}
(u\circ D^{-1})\bigl(\mathfrak r(\varphi^{-1}(z),\rho)\bigr)&=u\bigl(\mathfrak r(\varphi^{-1}(z),\sqrt{\rho})\bigr)\\
&=u_\varphi(z,\sqrt{\rho})\\
&=H(z,\rho).
\end{aligned}\]
Thus $u\circ D^{-1}$ is smooth near $y_0$. Since $y_0\in\partial M$
was arbitrary, $u\circ D^{-1}\in C^\infty(M;\mathbb R)$.
Therefore $u\in X$, and hence $Y\subset X$.
Therefore, $X=Y$.
This space $Y$ is precisely $\mathcal Y_N$ in \Cref{th:Neumann-smooth-closure}, with the normal derivatives understood in the fixed collar.

\medskip
\noindent Every element of $\mathcal A_N$ belongs to $Y$. Indeed, near the boundary, every $f\in\mathcal A_N$ has the form
\[f\bigl(\mathfrak r(y,r)\bigr)=P\bigl(\varphi_1^\partial(y),\ldots,\varphi_\ell^\partial(y),1+r^2,0,\ldots,0\bigr),\]
and is therefore even in $r$. Since $Y$ is closed in $C^\infty(M;\mathbb R)$, we conclude that $\overline{\mathcal A_N}^{C^\infty(M;\mathbb R)}=Y$.
\medskip

Finally, the density of $Y$ in $H_N^s(M;\R)$ follows from \cite[Remark~A.2]{MullerSchweigerSuli2020}, applied with $K=\{1,3,5,\ldots\}$. Indeed, we use a finite smooth partition of unity subordinate to interior charts and boundary normal charts, choosing the cut-off functions associated with the boundary charts to be independent of the normal variable near $\partial M$. Multiplication by these cut-off functions therefore preserves the vanishing of all odd normal derivatives. \cite[Remark~A.2]{MullerSchweigerSuli2020} applies in each boundary chart, while the usual density of smooth functions applies in the interior charts. Summing the resulting local approximations gives
\[\overline{Y}^{H^s(M;\R)}=H_N^s(M;\R).\]
Since $\mathcal A_N$ is dense in $X=Y$ for the $C^\infty$ topology, it is also dense in $Y$ for the $H^s$ topology. Consequently,
\[H_N^s(M;\R)=\overline{Y}^{H^s(M;\R)} \subset\overline{\mathcal A_N}^{H_N^s(M;\R)} \subset H_N^s(M;\R),\]
and hence
\[\overline{\mathcal A_N}^{H_N^s(M;\R)}=H_N^s(M;\R).\]
Since $\mathcal A_N\subset \mathsf E_{N,\infty}^{\R}$, we finally obtain
\[\overline{\mathsf E_{N,\infty}^{\R}}^{H_N^s(M;\R)}=H_N^s(M;\R).\]
\end{proof}

\begin{remark}
The algebra $\mathcal A_N$ does not satisfy the hypotheses of Nachbin's theorem. Indeed, every $f\in\mathcal A_N$ satisfies $\partial_\nu f\big|_{\partial M}=0$. Hence, for every $x\in\partial M$\footnote{The tangent space $T_xM$ should not be confused with
$T_x\partial M$. In particular $\nu(x)\in T_xM$ but $\nu(x)\notin T_x\partial M$ .} one has
\[d_xf(\nu(x))=0 \qquad\text{for every }f\in\mathcal A_N.\]
Therefore, $\mathcal A_N$ does not separate tangent vectors at the boundary, and condition \emph{(ii)} in Nachbin's theorem fails.
\end{remark}
\begin{corollary}[Density of the Neumann cone saturation]
\label{cor:Neumann-even-saturation}
Let $p\geq2$ be even. Define the additional profile
\[\Psi:=\sqrt{K^2-\sum_{j=1}^{\ell}\Phi_j^2-R^2-\Theta^2-\sum_{k=1}^{m}\Psi_k^2},\]
where
\[K^2>\max_{x\in M}\left(\sum_{j=1}^{\ell}\Phi_j(x)^2+R(x)^2+\Theta(x)^2+\sum_{k=1}^{m}\Psi_k(x)^2\right).\]
Let $\mathsf F_{N,\infty}^\R$ be the cone saturation of $E_N^\R+\operatorname{span}_{\R}\{\Psi\}$ defined by \eqref{eq:even-saturation}, with closures taken in $H_N^s(M;\R)$. Then
\[\overline{\mathsf F_{N,\infty}^\R}^{H_N^s(M;\R)}
=H_N^s(M;\R)\]
for $0\leq s<7/2$, $s\neq3/2$.
\end{corollary}

\begin{proof}
The additional profile $\Psi$ is smooth and satisfies $\partial_\nu\Psi=0$ on $\partial M$. Moreover, the sum of the squares of the enlarged family of generating profiles, including $1$, is the constant $1+K^2$. The same argument as in the even-power case of \Cref{lem:algebra-odd-sat} therefore shows that $\mathsf F_{N,\infty}^\R$ contains the algebra generated by the original control profiles. The conclusion follows from \Cref{lem:Neumann-sobolev-density}.
\end{proof}

\begin{proof}[Proof of \Cref{th:Neumann-approximation}]
Choose $E_N^\R$ as in \eqref{eq:Neumann-control-family} and $\Psi$ as in \Cref{cor:Neumann-even-saturation}. The two density statements follow from \Cref{lem:Neumann-sobolev-density,cor:Neumann-even-saturation}. These profiles do not depend on $p$ or $s$. Each finite stage of the cone saturation is contained in a finite-dimensional space of smooth Neumann-compatible functions. All the norms $H_N^s$ in the stated range induce the same topology on this space, so the closures do not introduce a dependence on $s$. Complexification gives the corresponding complex-valued density statements.
\end{proof}

\subsection{A Dirichlet obstruction}
\label{sec:Dirichlet-obstruction}
The construction of finite families under Dirichlet boundary conditions is less clear. A natural idea is to choose a smooth function $\rho$ which is positive in $M^\circ$ and vanishes on $\partial M$, and to replace the embedding coordinates $1,V_1,\ldots,V_m$ by $\rho,\rho V_1,\ldots,\rho V_m$. This family separates the points of the quotient space $M/\partial M$, and the non-unital algebra
\[\mathcal A_D:=\left\{P(\rho,\rho V_1,\ldots,\rho V_m):P\in\mathbb R[X_0,\ldots,X_m],\ P(0)=0\right\}\]
is dense in the space of continuous functions vanishing on
the boundary. This observation, however, does not give the required density property. We identify two main complications:\par\medskip

\noindent The constant function $1$ is not an admissible Dirichlet profile. The polarization identity of degree $p>2$ gives products of $p$ elements of the power saturation, but it does not give products of two elements. In particular, there is no reason in general for even-degree monomials such as $\rho^2$ to be generated.\par\medskip

\noindent There is also a Sobolev obstruction when $s>3/2$, as stated in the following proposition.
Let $p\geq2$ be an integer and let $E_D^\R$ be a finite-dimensional subspace of
\[\{f\in C^\infty(\overline M;\R):f|_{\partial M}=0\}.\]
Define $\mathsf E_{D,\infty}^\R$ by the recursion \eqref{eq:odd-saturation}, starting from $\mathsf E_{D,0}^\R=E_D^\R$, without requiring the constant function $1$. For even $p$, also define $\mathsf F_{D,\infty}^\R$ by \eqref{eq:even-saturation}, starting from $\mathsf F_{D,0}^\R=E_D^\R$ and taking the closures in $H^s(M;\R)$. We define also the spaces 
\[\mathcal X_D^s:=\{u\in H^s(M;\R):u|_{\partial M}=0\}\]
and 
\[T_D:=\{\partial_\nu f|_{\partial M}:f\in E_D^\R\}.\]
\begin{proposition}[Normal-trace obstruction]
\label{prop:Dirichlet-power-obstruction}
Let $(M,g)$ be a smooth compact manifold with boundary and dimension $d\geq2$ and $s>3/2$. Then, $\mathsf E_{D,\infty}^\R$ is not dense, in the $H^s$ topology, in $\mathcal X_D^s$. More precisely, 
\[\overline{\mathsf E_{D,\infty}^\R}^{H^s(M;\R)}
\subset\{u\in\mathcal X_D^s:\partial_\nu u|_{\partial M}\in T_D\} \subsetneq\mathcal X_D^s.\]
For even $p$, the same inclusions hold with $\mathsf F_{D,\infty}^\R$ in place of $\mathsf E_{D,\infty}^\R$.
\end{proposition}

\begin{proof}
All functions in $\mathsf E_{D,n}^\R$ vanish on the boundary. For each such function,
\[\partial_\nu(f^p)|_{\partial M}=p f^{p-1}\partial_\nu f|_{\partial M}=0.\]
Induction therefore gives $\partial_\nu\mathsf E_{D,n}^\R|_{\partial M}\subset T_D$ for every $n$. For even $p$, induction also gives $\mathsf F_{D,n}^\R\subset\mathsf E_{D,n}^\R$. Indeed, the cone generated from $\mathsf F_{D,n}^\R$ is contained in $\mathsf E_{D,n+1}^\R$, and this last space is finite-dimensional and hence closed in $H^s(M;\R)$. In particular, the cone saturation cannot create any new normal trace either:
\[\partial_\nu\left(\eta-\sum_{j=1}^N\varphi_j^p\right)\Big|_{\partial M}=\partial_\nu\eta\big|_{\partial M}.\]
For $s>3/2$, the boundary trace and the normal trace are continuous on $H^s(M)$, with the latter taking values in $H^{s-3/2}(\partial M)$. Since $T_D$ is finite-dimensional, it is closed in this space. The restrictions to zero boundary trace and to normal trace in $T_D$ therefore persist in the $H^s$ closures of both saturations. This proves the asserted inclusions.
Since $\dim\partial M\geq1$, choose $h\in C^\infty(\partial M;\R)\setminus T_D$. In a normal collar, the function $u(\mathfrak r(y,r))=-r\chi(r)h(y)$, with a smooth cut-off $\chi$ equal to $1$ near $0$ and supported in the collar, extends smoothly by zero to $M$ and satisfies $u|_{\partial M}=0$ and $\partial_\nu u|_{\partial M}=h$.
Thus $u\in\mathcal X_D^s$ lies outside the right-hand set, proving strictness.
\end{proof}

Proposition \ref{prop:Dirichlet-power-obstruction} concerns both saturations (power \eqref{eq:odd-saturation} and cone \eqref{eq:even-saturation}). For $3/2<s<5/2$, $\mathcal X_D^s$ also agrees, with equivalent norms, with the spectral Dirichlet space $D((I-\Delta_{g,D})^{s/2})$; see \cite{Grisvard,Seeley,Grubb}.

\section{Controllability in the Sobolev-algebra regime}
\label{sec:odd_nonlinearity}
In this section, we prove the small-time approximate controllability results for nonlinear Schr\"odinger, Klein--Gordon, and parabolic equations in the Sobolev-algebra regime. Throughout this section, we fix an integer $p\geq2$ and assume that $s>d/2$. The algebraic saturation argument has already been established in the previous section. It remains to prove the short-time dynamical limits for each equation.

\subsection{Schr\"odinger equation}
\label{sec:additive-NLS}
Let $(M,g)$ be a smooth compact connected Riemannian manifold without boundary. Let $W\in C^\infty(M;\R)$, $\kappa\in\R\setminus\{0\}$. We consider the nonlinear Schr\"odinger equation with gauge-invariant nonlinearity $|u|^{p-1}u$, for both odd and even $p$.
\begin{equation}
\label{eq:NLS}
\left\{
\begin{aligned}
&i\partial_t u=(-\Delta_g+W)u+\kappa |u|^{p-1}u+\eta(t,x),\qquad (t,x)\in[0,T]\times M,\\
&u(0)=u_0\in H^s(M;\C),
\end{aligned}
\right.
\end{equation}
where $s>d/2$. For odd $p$, the nonlinearity is polynomial in $u$ and $\overline u$.
For even $p$, it is not smooth at the origin, and we impose $s<p-1$ in order to use the composition estimates. Set $A:=-\Delta_g+W$ and $U(t):=e^{-itA}$. Since $W\in C^\infty(M;\R)$, the operator $A$ is self-adjoint on $L^2(M)$ and $U(t)$ defines a strongly continuous Schrödinger group. Moreover, for every $s\ge 0$ and every $T>0$,
\[\sup_{|t|\le T}\|U(t)\|_{\mathcal L(H^s,H^s)}<\infty .\]
We say that $u\in C([0,T_*];H^s(M;\C))$ is a mild solution of \eqref{eq:NLS} if
\begin{equation}
\label{eq:NLS-mild}
u(t)=\Gamma_{u_0,\eta}(u)(t):=U(t)u_0-i\kappa\int_0^t U(t-\sigma)\bigl(|u(\sigma)|^{p-1}u(\sigma)\bigr)d\sigma
-i\int_0^t U(t-\sigma)\eta(\sigma)d\sigma .
\end{equation}
For such $s$, let $u_0\in H^s(M;\C)$ and $\eta\in L^1(0,T;H^s(M;\C))$.
By the algebra property of $H^s(M)$ when $p$ is odd, and by \cite[Corollary~23(ii)]{Dinh2018} when $p$ is even, the nonlinear map $F(u):=|u|^{p-1}u$ is locally Lipschitz from $H^s(M;\C)$ into itself. More precisely, for every $R>0$, there exists $C_R>0$ such that
\[\|F(u)-F(v)\|_{H^s}\le C_R\|u-v\|_{H^s},\qquad\|u\|_{H^s},\|v\|_{H^s}\le R.\]
By homogeneity, the estimate on the unit ball also gives, for every $u,v\in H^s(M;\C)$,
\begin{equation}
\label{eq:NLS-nonlinear-estimates}
\begin{aligned}
\|F(u)\|_{H^s}&\le C\|u\|_{H^s}^{p},\\
\|F(u)-F(v)\|_{H^s}&\le C\bigl(\|u\|_{H^s}+\|v\|_{H^s}\bigr)^{p-1}\|u-v\|_{H^s}.
\end{aligned}
\end{equation}
Therefore, applying the contraction mapping theorem to \eqref{eq:NLS-mild}, there exists $T_*>0$ such that \eqref{eq:NLS} admits a unique mild solution $u\in C([0,T_*];H^s(M;\C))$. Furthermore, the solution depends continuously on the initial datum $u_0$ and on the control $\eta$. Using standard arguments, we obtain the following result.
\begin{proposition}[Local well-posedness for the Schrödinger equation]
\label{prop:NLS-LWP}
Let $T>0$, $W\in C^\infty(M;\R)$, $\kappa\in\R$ and $s>d/2$, with $s<p-1$ when $p$ is even. Let $u_0\in H^s(M;\C)$, and $\eta\in L^1(0,T;H^s(M;\C))$.
Then there exists $T_*\in(0,T]$ such that \eqref{eq:NLS} admits a unique mild solution $u\in C([0,T_*];H^s(M;\C))$. Moreover, the following properties hold:
\begin{itemize}
\item The following uniform local bound holds. For every $R>0$, there exist $T_R\in(0,T]$ and $C_R>0$ such that, for every $v_0\in H^s(M;\C)$ and every $\zeta\in L^1(0,T_R;H^s(M;\C))$ satisfying $\|v_0\|_{H^s}+\|\zeta\|_{L^1(0,T_R;H^s)}\le R$,
the corresponding solution is defined on $[0,T_R]$ and satisfies
\[\|v\|_{C([0,T_R];H^s)}\le C_R.\]
\item The solution depends continuously on $u_0$ and $\eta$; i.e.
for any $\varkappa\in(0,T_*)$, there exist $\rho=\rho(\varkappa,\Lambda)>0$ and $C=C(\varkappa,\Lambda)>0$, with $\Lambda=\|u\|_{C([0,\varkappa];H^s)}+\|\eta\|_{L^1(0,\varkappa;H^s)}$,
such that the following holds. If $v_0\in H^s(M;\C)$ and $\zeta\in L^1(0,\varkappa;H^s(M;\C))$ satisfy
\[\|v_0-u_0\|_{H^s}+\|\zeta-\eta\|_{L^1(0,\varkappa;H^s)}<\rho,\]
then the solution $v$ associated with $(v_0,\zeta)$ is well defined on $[0,\varkappa]$ and satisfies
\[\|v-u\|_{C([0,\varkappa];H^s)}\le C\left(
\|v_0-u_0\|_{H^s}+\|\zeta-\eta\|_{L^1(0,\varkappa;H^s)}\right).\]
\end{itemize}
\end{proposition}

Whenever the solution of \eqref{eq:NLS} with initial datum $u_0$ and control $\eta$ is defined, we denote by $\mathcal R_t(u_0,\eta)$ its value at time $t$. The uncontrolled flow is given by $\mathcal R_t(u_0,0)$. For $\xi\in H^s(M;\C)$, we define the translation $\tau_\xi(u):=u+\xi$.
\medskip

\noindent The flow satisfies the following concatenation property. If $\eta$ is defined by
\[\eta(t)=
\begin{cases}
\eta_1(t), & t\in[0,t_1],\\
\eta_2(t-t_1), & t\in[t_1,t_1+t_2],
\end{cases}\]
then
\[\mathcal R_{t_1+t_2}(u_0,\eta)=\mathcal R_{t_2}\bigl(\mathcal R_{t_1}(u_0,\eta_1),\eta_2\bigr).\]
Let $V_0=1$ and let $V_1,\dots,V_m\in C^\infty(M;\R)$. We define the finite-dimensional control space $E:=\operatorname{span}_{\C}\{V_0,V_1,\dots,V_m\}
\subset C^\infty(M;\C)$. Thus the additive controls are of the form  $\eta\in L^\infty(0,T;E)$, 
\[\eta(t,x)=\sum_{j=0}^m \alpha_j(t)V_j(x),
\qquad \alpha_j\in L^\infty(0,T;\C).\]
Even when $p$ is even, we use the same control profiles. In particular, we do not need to add the extra profile $V_{m+1}$. Indeed, the map $\alpha\mapsto-i\kappa\alpha|\alpha|^{p-1}$ is onto $\C$ for both parities, so the nonlinear kick recovers all complex multiples of the generated profile, as used in \Cref{lem:NLS-saturated-directions}.

We now prove \Cref{thm:intro-NLS-STAC}. The proof is based on three steps: small-time realizability of direct translations, a short-time asymptotic limit for large translations, and the nonlinear kick mechanism, which turns realizability in the direction $\varphi$ into realizability in the direction ${|\varphi|^{p-1}\varphi}$.
\begin{lemma}[Direct translations]
\label{lem:NLS-direct-translations}
For every $\xi\in E$, the translation $\tau_\xi$ is small-time approximately realizable.
\end{lemma}

\begin{proof}
Fix $u_0\in H^s(M;\C)$ and $\xi\in E$. Let $\varkappa>0$. For $\delta>0$, set the constant control
\[\eta_\delta(t):=\frac{i}{\delta}\xi,
\qquad 0\le t\le\delta.\]
Then $\eta_\delta\in L^\infty(0,\delta;E)$ and, by the mild formulation,
\[\mathcal R_\delta(u_0,\eta_\delta)=U(\delta)u_0+\frac1\delta\int_0^\delta U(\delta-\sigma)\xi d\sigma-i\kappa\int_0^\delta U(\delta-\sigma)\left(|\mathcal R_\sigma(u_0,\eta_\delta)|^{p-1}\mathcal R_\sigma(u_0,\eta_\delta)\right)d\sigma.\]
By strong continuity of $U(t)$ in $H^s$ 
\[U(\delta)u_0\to u_0,
\qquad
\frac1\delta\int_0^\delta U(\delta-\sigma)\xi d\sigma\to \xi
\qquad \text{ in } H^s  \ \text{ as } \delta\to0^+.\]
Moreover,
\begin{equation}
\label{eq:eta_L1Hs}
\int_0^\delta \|\eta_\delta(t)\|_{H^s}dt=\|\xi\|_{H^s},  
\end{equation}
and by the uniform local bound in \Cref{prop:NLS-LWP}, since
\[\|u_0\|_{H^s}+\|\eta_\delta\|_{L^1(0,\delta;H^s)}
=\|u_0\|_{H^s}+\|\xi\|_{H^s},\]
we have, for all sufficiently small $\delta>0$,
\[\sup_{0\le t\le\delta}
\|\mathcal R_t(u_0,\eta_\delta)\|_{H^s}
\le C(u_0,\xi).\]
By \eqref{eq:NLS-nonlinear-estimates},
\[\left\|
\int_0^\delta U(\delta-\sigma)\left(|\mathcal R_\sigma(u_0,\eta_\delta)|^{p-1}\mathcal R_\sigma(u_0,\eta_\delta)\right)d\sigma\right\|_{H^s}\le C\delta\sup_{0\le\sigma\le\delta}
\|\mathcal R_\sigma(u_0,\eta_\delta)\|_{H^s}^{p}
\to0.\]
Hence, as $\delta$ goes to zero,
\[\mathcal R_\delta(u_0,\eta_\delta)\to u_0+\xi \qquad\text{in }H^s(M;\C).\]
Taking $\delta<\varkappa$ sufficiently small gives the result.
\end{proof}

We next prove the short-time asymptotic mechanism that produces new directions through the power nonlinearity. The main idea is to use the control to insert a large translation of size $\delta^{{-1/p}}$, let the uncontrolled equation evolve for a time $\delta$, and then use the control to remove the same translation. This produces a non-trivial commutator, yielding a \textit{nonlinear kick}. The scaling is chosen so that the nonlinear term has a finite contribution, since $\delta (\delta^{{-1/p}})^{p}=1$.
\begin{lemma}[Asymptotic limit]
\label{lem:NLS-asymptotic-limit}
Let $\varphi\in C^\infty(M;\R)$ and $\alpha\in\C$. Then, for every $u_0\in H^s(M;\C)$, there exists $\delta_0>0$ such that, for every $\delta\in(0,\delta_0)$,
\[t\mapsto \mathcal R_t\Big( u_0 +\delta^{-\frac{1}{p}}\alpha\varphi ,0\Big) \quad\text{ is well-defined in } C([0,\delta];H^s(M)).\]
Moreover,
\[\mathcal R_\delta\left(u_0+\delta^{-\frac1{p}}\alpha\varphi
,0\right)-\delta^{-\frac1{p}}\alpha\varphi\longrightarrow
u_0-i\kappa\alpha|\alpha|^{p-1}{|\varphi|^{p-1}\varphi} \quad \text{ in } H^s(M;\C) \quad  \text{ as } \delta\to0^+.\]
\end{lemma}

\begin{proof}
First, notice that \Cref{prop:NLS-LWP} does not directly ensure that $\mathcal R_\delta\left(u_0+\delta^{-\frac1{p}}\alpha\varphi,0\right)$ is well defined, since the corresponding initial datum may become unbounded in $H^s$ as $\delta\to0$. As part of the proof, we justify that this is indeed well-defined, by introducing a renormalized Duhamel map. Define 
\begin{equation}
 \widetilde{\Gamma}_\delta(w):=\Gamma_{u_0+\delta^{-\frac1{p}}\alpha\varphi,0}\left(w+\delta^{-\frac1{p}}\alpha\varphi\right)-\delta^{-\frac1{p}}\alpha\varphi,  
\end{equation}
where $\Gamma_{u_0,\eta}$ is the Duhamel map \eqref{eq:NLS-mild} associated with initial datum $u_0$ and control $\eta$. We therefore apply a fixed-point argument to the equation
\begin{align}
w(t)=\widetilde{\Gamma}_\delta(w){(t)}={}&U(t)u_0+\bigl(U(t)-I\bigr)\delta^{-\frac1{p}}\alpha\varphi\nonumber\\
&-i\kappa\int_0^tU(t-\sigma)\left(\left|\delta^{-\frac1{p}}\alpha\varphi+w(\sigma)\right|^{p-1}\left(\delta^{-\frac1{p}}\alpha\varphi+w(\sigma)\right)\right)d\sigma.
\label{eq:w_delta}
\end{align}
For this purpose, we need two estimates. The first one ensures that the
Duhamel map leaves a fixed ball of $C([0,\delta];H^s)$ invariant, while the second one ensures that it is a contraction. First, since $\varphi\in C^\infty(M)$,
\begin{equation}
\label{eq:U-I_delta}
\sup_{0\leq t\leq\delta}\left\|\bigl(U(t)-I\bigr)\delta^{-\frac1{p}}\alpha\varphi\right\|_{H^s}\leq C_\varphi|\alpha|\delta^{1-\frac1{p}}=C_\varphi|\alpha|\delta^{\frac{p-1}{p}}.
\end{equation}
By \eqref{eq:NLS-nonlinear-estimates}, for $\|w\|_{H^s}\leq R$,
\[\left\|\left|\delta^{-\frac1{p}}\alpha\varphi+w\right|^{p-1}\left(\delta^{-\frac1{p}}\alpha\varphi+w\right)\right\|_{H^s}\leq C(\delta^{-1}+R^{p}).\]
After integration in time over $[0,\delta]$, this estimate provides a bounded contribution and therefore allows one to choose a fixed ball of $C([0,\delta];H^s)$ which is invariant under the Duhamel map, for $\delta>0$ sufficiently small. Moreover, for
$\|w_1\|_{H^s},\|w_2\|_{H^s}\le R$,
\begin{multline*}
\left\|\left|\delta^{-\frac1{p}}\alpha\varphi+w_1\right|^{p-1}\left(\delta^{-\frac1{p}}\alpha\varphi+w_1\right)-\left|\delta^{-\frac1{p}}\alpha\varphi+w_2\right|^{p-1}\left(\delta^{-\frac1{p}}\alpha\varphi+w_2\right)\right\|_{H^s}\\
\leq C\left(\delta^{-\frac{p-1}{p}}+R^{p-1}\right)\|w_1-w_2\|_{H^s}.
\end{multline*}
Since
\[\delta\left(\delta^{-\frac{p-1}{p}}+R^{p-1}\right)=\delta^{\frac1{p}}+\delta R^{p-1}\longrightarrow0,\]
this Lipschitz estimate ensures that the same Duhamel map is a contraction on that ball, for $\delta>0$ sufficiently small. Hence, by the Banach fixed-point theorem, $\widetilde{\Gamma}_\delta$ admits a unique fixed point $w^\delta=\widetilde{\Gamma}_\delta(w^\delta) \in C([0,\delta];H^s)$. Moreover, there exists $C>0$ such that
\[\sup_{0\le t\le\delta}\|w^\delta(t)\|_{H^s}\le C,\qquad\text{for all sufficiently small }\delta.\]
For $0\le t\le\delta$, $v^\delta(t):=w^\delta(t)+\delta^{-\frac1{p}}\alpha\varphi$ satisfies Duhamel's formulation \eqref{eq:NLS-mild} with initial datum $u_0+\delta^{-\frac1{p}}\alpha\varphi$ and no control. Therefore by uniqueness $v^\delta(t){=}\mathcal R_t\left(u_0+\delta^{-\frac1{p}}\alpha\varphi,0\right)$. We now take the limit $\delta\to0$ in \eqref{eq:w_delta}. As before, $U(\delta)u_0\to u_0$ in $H^s$, and by \eqref{eq:U-I_delta}
\[U(\delta)\left(u_0+\delta^{-\frac1{p}}\alpha\varphi\right)-\delta^{-\frac1{p}}\alpha\varphi\longrightarrow u_0 \qquad\text{in }H^s.\]
Moreover{, by \eqref{eq:NLS-nonlinear-estimates}, uniformly for $0\le t\le\delta$,}
\begin{align*}
|v^\delta(t)|^{p-1}v^\delta(t)&=\left|\delta^{-\frac1{p}}\alpha\varphi+w^\delta(t)\right|^{p-1}\left(\delta^{-\frac1{p}}\alpha\varphi+w^\delta(t)\right) =\delta^{-1}\alpha|\alpha|^{p-1}{|\varphi|^{p-1}\varphi}+O_{H^s}\left(\delta^{-\frac{p-1}{p}}\right).
\end{align*}
Thus
\[\int_0^\delta U(\delta-\sigma)\left(|v^\delta(\sigma)|^{p-1}v^\delta(\sigma)\right)d\sigma=\delta^{-1}\alpha|\alpha|^{p-1}\int_0^\delta U(\delta-\sigma){|\varphi|^{p-1}\varphi} d\sigma+O_{H^s}\left(\delta^{\frac1{p}}\right).\]
The profile $|\varphi|^{p-1}\varphi$ belongs to $H^s(M)$ by \eqref{eq:NLS-nonlinear-estimates}. By the strong continuity of $U(t)$ on $H^s$,
\[\frac1\delta\int_0^\delta U(\delta-\sigma){|\varphi|^{p-1}\varphi} d\sigma\longrightarrow{|\varphi|^{p-1}\varphi}\qquad\text{in }H^s.\]
Therefore
\[\int_0^\delta U(\delta-\sigma)\left(|v^\delta(\sigma)|^{p-1}v^\delta(\sigma)\right)d\sigma\longrightarrow\alpha|\alpha|^{p-1}{|\varphi|^{p-1}\varphi}\qquad\text{in }H^s.\]
We conclude that
\[w^\delta(\delta)\longrightarrow u_0-i\kappa\alpha|\alpha|^{p-1}{|\varphi|^{p-1}\varphi},\]
which provides the desired limit.
\end{proof}
We now prove the first nonlinear saturation step which turns realizability of large translations in the direction $\phi$ into realizability of the nonlinear direction ${|\phi|^{p-1}\phi}$.
\begin{lemma}[Nonlinear kick]
\label{lem:NLS-odd-kick}
Let $\varphi\in C^\infty(M;\R)$ and $\alpha\in\C$. Assume that, for every $\delta>0$, the translations $\tau_{\pm\delta^{-\frac1{p}}\alpha\varphi}$ are small-time approximately realizable. Then the translation $\tau_{-i\kappa\alpha|\alpha|^{p-1}{|\varphi|^{p-1}\varphi}}$ is small-time approximately realizable.
\end{lemma}
\begin{proof}
Set $\chi:=-i\kappa\alpha|\alpha|^{p-1}{|\varphi|^{p-1}\varphi}$. Fix $u_0\in H^s(M;\C)$, $\eps>0$ and $\varkappa>0$. By \Cref{lem:NLS-asymptotic-limit}, we can choose $\delta\in(0,\varkappa/3)$ such that
\[\left\|\mathcal R_\delta(u_0+\delta^{-\frac1{p}}\alpha\varphi,0)-\delta^{-\frac1{p}}\alpha\varphi-(u_0+\chi)\right\|_{H^s}<\frac{\eps}{3}.\]
By the same fixed-point argument and continuity of $\mathcal R_\delta(\cdot,0)$ with respect to the initial datum, there exists $\rho>0$ such that, if
\[\|z-(u_0+\delta^{-\frac1{p}}\alpha\varphi)\|_{H^s}<\rho,\]
then the uncontrolled solution with initial datum $z$ is well defined on $[0,\delta]$ and
\[\|\mathcal R_\delta(z,0)-\mathcal R_\delta(u_0+\delta^{-\frac1{p}}\alpha\varphi,0)\|_{H^s}<\frac{\eps}{3}.\]
Since $\tau_{\delta^{-\frac1{p}}\alpha\varphi}$ is small-time approximately realizable, there exist $T_1\in(0,\varkappa/3)$ and a control $\eta_1\in L^\infty(0,T_1;E)$ such that
\[\left\|\mathcal R_{T_1}(u_0,\eta_1)-(u_0+\delta^{-\frac1{p}}\alpha\varphi)\right\|_{H^s}<\rho.\]
Set $z:=\mathcal R_{T_1}(u_0,\eta_1)$. On the interval $[T_1,T_1+\delta]$, we take zero control, so the state becomes $\mathcal R_\delta(z,0)$. Finally, since $\tau_{-\delta^{-\frac1{p}}\alpha\varphi}$ is small-time approximately realizable, there exist $T_2\in(0,\varkappa/3)$ and a control $\eta_2\in L^\infty(0,T_2;E)$ such that
\[\left\|\mathcal R_{T_2}(\mathcal R_\delta(z,0),\eta_2)-\left(\mathcal R_\delta(z,0)-\delta^{-\frac1{p}}\alpha\varphi\right)\right\|_{H^s}<\frac{\eps}{3}.\]
Concatenating the three pieces, with total time $T:=T_1+\delta+T_2<\varkappa$, we obtain
\[\left\|\mathcal R_T(u_0,\eta)-(u_0+\chi)\right\|_{H^s}<\eps.\]
Thus $\tau_\chi$ is small-time approximately realizable.
\end{proof}
We also prove that small-time approximate realizability is stable under sums and limits.
\begin{lemma}[Stability under sums and limits]
\label{lem:NLS-sums}
If $\tau_{\xi_1}$ and $\tau_{\xi_2}$ are small-time approximately realizable, then $\tau_{\xi_1+\xi_2}$ is small-time approximately realizable.
Moreover, if $\xi_j\to\xi$ in $H^s(M;\C)$ and every $\tau_{\xi_j}$ is small-time approximately realizable, then so is $\tau_\xi$.
\end{lemma}
\begin{proof}
Fix $u_0\in H^s(M;\C)$, $\eps>0$ and $\varkappa>0$. First realize $\tau_{\xi_1}$ from $u_0$ in time $T_1<\varkappa/2$ with sufficiently small error. From the obtained point, realize $\tau_{\xi_2}$ in time $T_2<\varkappa/2$ with sufficiently small error. By concatenation and the triangle inequality, the final point is arbitrarily close to $u_0+\xi_1+\xi_2$ in time $T_1+T_2<\varkappa$. For the second assertion, choose $j$ such that $\|\xi_j-\xi\|_{H^s}<\eps/2$, and then realize $\tau_{\xi_j}$ from $u_0$ in time smaller than $\varkappa$, with error smaller than $\eps/2$. The triangle inequality gives the result.
\end{proof}
From now on we assume $\kappa\neq0$ and we recall the definition of the saturation spaces $\mathsf E_n^\R$ and $\mathsf E_n$ given in \eqref{eq:odd-saturation}, which we use for both odd and even $p$.

\begin{lemma}[Realizability of saturated directions]
\label{lem:NLS-saturated-directions} 
Let $\xi\in\mathsf E_\infty$. Then the translation $\tau_\xi$ is small-time approximately realizable.
\end{lemma}

\begin{proof}
We argue by induction on $n$. For $n=0$, the result follows from \Cref{lem:NLS-direct-translations}, since $\mathsf E_0=E$. Assume that every translation $\tau_\xi$, with $\xi\in \mathsf E_n$, is small-time approximately realizable. Let $\varphi\in \mathsf E_n^{\R}$. Then, for every $\delta>0$ and every $\alpha\in\C$, $\pm\delta^{-\frac1{p}}\alpha\varphi\in \mathsf E_n$, and therefore the translations $\tau_{\pm\delta^{-\frac1{p}}\alpha\varphi}$ are small-time approximately realizable. By \Cref{lem:NLS-odd-kick}, the translation $\tau_{-i\kappa\alpha|\alpha|^{p-1}{|\varphi|^{p-1}\varphi}}$ is small-time approximately realizable. Since $\kappa\neq0$ and the map $\alpha\mapsto -i\kappa\alpha|\alpha|^{p-1}$ is onto $\C$, it follows that $\tau_{\beta{|\varphi|^{p-1}\varphi}}$ is small-time approximately realizable for every $\beta\in\C$. If $p$ is odd, since $\varphi$ is real, we have $|\varphi|^{p-1}\varphi=\varphi^p$. Combining this with the induction hypothesis and \Cref{lem:NLS-sums}, we obtain in this case that $\tau_\xi$ is small-time approximately realizable for every $\xi\in \mathsf E_{n+1}$.

If $p$ is even, choose distinct constants $c_0,\ldots,c_p>\|\varphi\|_{L^\infty}$. Since $1\in \mathsf E_n^\R$, we have $c_j+\varphi\in \mathsf E_n^\R$ and $c_j+\varphi>0$ on $M$. Applying the preceding argument to these profiles, we obtain that $\tau_{\beta(c_j+\varphi)^p}$ is small-time approximately realizable for every $\beta\in\C$. By the binomial formula and the invertibility of the corresponding Vandermonde matrix,
\[\varphi^p\in\operatorname{span}_{\R}\{(c_0+\varphi)^p,\ldots,(c_p+\varphi)^p\}.\]
Hence $\tau_{\beta\varphi^p}$ is small-time approximately realizable
for every $\beta\in\C$. Combining this with the induction hypothesis
and \Cref{lem:NLS-sums}, we obtain that $\tau_\xi$ is small-time
approximately realizable for every $\xi\in \mathsf E_{n+1}$.
\end{proof}
\begin{proposition}[Saturation criterion]
\label{prop:NLS-saturation-criterion}
Assume that $\mathsf E_\infty$ is dense in $H^s(M;\C)$. Then \eqref{eq:NLS} is small-time approximately controllable in $H^s(M;\C)$.
\end{proposition}

\begin{proof}
Let $u_0,u_1\in H^s(M;\C)$, $\eps>0$ and $\varkappa>0$. By density of $\mathsf E_\infty$, there exists $\xi\in \mathsf E_\infty$ such that
\[\|u_1-u_0-\xi\|_{H^s}<\frac{\eps}{2}.\]
By \Cref{lem:NLS-saturated-directions}, the translation $\tau_\xi$ is small-time approximately realizable. Hence there exist $T\in(0,\varkappa)$ and $\eta\in L^\infty(0,T;E)$ such that
\[\|\mathcal R_T(u_0,\eta)-(u_0+\xi)\|_{H^s}<\frac{\eps}{2}.\]
Therefore
\[\|\mathcal R_T(u_0,\eta)-u_1\|_{H^s}<\eps.\]
\end{proof}

\textit{Proof of \Cref{thm:intro-NLS-STAC}:}
By \Cref{th:odd-saturation-density}, the power saturation $\mathsf E_\infty$ of $E$ is dense in $H^s(M;\C)$. The conclusion follows directly from \Cref{prop:NLS-saturation-criterion}.
\qed
\begin{remark}
The same argument also applies to the less physical polynomial
nonlinearity $u^p$, which is not gauge-invariant. In this case,
the nonlinear kick generates the direction
$-i\kappa\alpha^p\varphi^p$. Since the map
$\alpha\mapsto-i\kappa\alpha^p$ is onto $\C$ for every integer
$p\ge2$, the same control space
$E=\operatorname{span}_{\C}\{1,V_1,\dots,V_m\}$ suffices for both
parities, without adding the profile $V_{m+1}$.
Moreover, since $H^s(M)$ is an algebra for $s>d/2$,
the map $u\mapsto u^p$ is locally Lipschitz on $H^s(M;\C)$,
so the local well-posedness and controllability arguments apply
for every $s>d/2$.
\end{remark}
\subsection{Klein--Gordon equation}
\label{sec:additive-KG}
Let $(M,g)$ be a smooth compact connected Riemannian manifold without boundary. Let $W\in C^\infty(M;\R)$ and $\kappa\in\R\setminus\{0\}$. We consider the controlled Klein--Gordon equation with polynomial nonlinearity
\begin{equation}
\label{eq:KG}
\begin{cases}
\partial_t^2 y-\Delta_g y+Wy+\kappa y^{p}=\eta(t,x), &(t,x)\in(0,T)\times M,\\
(y(0),\partial_t y(0))=(y_0,y_1).
\end{cases}
\end{equation}
We take $E^\R=\operatorname{span}_{\R}\{1,V_1,\ldots,V_m\}$ when $p$ is odd, and $E^\R=\operatorname{span}_{\R}\{1,V_1,\ldots,V_{m+1}\}$ when $p$ is even, where $V_{m+1}$ is defined in \eqref{eq:V_m+1}. Let $\Lambda:=(-\Delta_g+W)^{1/2}$, and assume, for simplicity, that $-\Delta_g+W$ is positive. For $s>d/2$, we work in the Sobolev energy space
\[\mathcal H^s:=H^{s}(M;\R)\times H^{s-1}(M;\R).\]
We write Klein--Gordon states as column vectors, using $Y=(y,v)^\top$ in inline notation; along solutions, $v=\partial_t y$. We denote the linear Klein--Gordon group by
\[S(t)\binom{y_0}{y_1}:=\binom{\cos(t\Lambda)y_0+\Lambda^{-1}\sin(t\Lambda)y_1}{-\Lambda\sin(t\Lambda)y_0+\cos(t\Lambda)y_1}.\]
A mild solution of \eqref{eq:KG} is a function $Y=(y,\partial_t y)^\top\in C([0,T_*];\mathcal H^s)$ such that
\begin{equation}
\label{eq:mildformulationwaves}
\begin{cases}
y(t)=\cos(t\Lambda)y_0+\Lambda^{-1}\sin(t\Lambda)y_1+\dis\int_0^t \Lambda^{-1}\sin((t-\sigma)\Lambda)
\left(\eta(\sigma)-\kappa y(\sigma)^{p}\right)d\sigma\\
\partial_t y(t)=-\Lambda\sin(t\Lambda)y_0+\cos(t\Lambda)y_1+\dis\int_0^t \cos((t-\sigma)\Lambda)
\left(\eta(\sigma)-\kappa y(\sigma)^{p}\right)d\sigma.\end{cases}
\end{equation}

\begin{proposition}[Local well-posedness for the Klein--Gordon equation]
\label{prop:KG-LWP}
Let $T>0$, $W\in C^\infty(M;\R)$, $\kappa\in\R$ and $s>d/2$. Let $(y_0,y_1)^\top\in\mathcal H^s$, $\eta\in L^1(0,T;H^{s-1}(M;\R))$. Then there exists $T_*\in(0,T]$ such that \eqref{eq:KG} admits a unique mild solution
\[Y=\binom{y}{\partial_t y}\in C([0,T_*];\mathcal H^s).\]
Moreover, the following properties hold:
\begin{itemize}
\item For every $R>0$, there exist $T_R\in(0,T]$ and $C_R>0$ such that, for every $(z_0,z_1)^\top\in\mathcal H^s$ and every $\zeta\in L^1(0,T_R;H^{s-1}(M;\R))$ satisfying
\[\left\|\binom{z_0}{z_1}\right\|_{\mathcal H^s}+\|\zeta\|_{L^1(0,T_R;H^{s-1})}\le R,\]
the corresponding solution is defined on $[0,T_R]$ and satisfies
\[\left\|\binom{z}{\partial_t z}\right\|_{C([0,T_R];\mathcal H^s)}\le C_R.\]
\item The solution depends continuously on the initial datum and on the control. More precisely, if $Y=(y,\partial_t y)^\top$ and $Z=(z,\partial_t z)^\top$ are two solutions on $[0,\varkappa]$, associated with data $((y_0,y_1)^\top,\eta)$ and $((z_0,z_1)^\top,\zeta)$, and if
\[\|Y\|_{C([0,\varkappa];\mathcal H^s)}+\|Z\|_{C([0,\varkappa];\mathcal H^s)}\le R,\]
then there exists $C_R>0$ such that
\[\|Y-Z\|_{C([0,\varkappa];\mathcal H^s)}\le C_R\left(\left\|\binom{y_0}{y_1}-\binom{z_0}{z_1}\right\|_{\mathcal H^s}
+\|\eta-\zeta\|_{L^1(0,\varkappa;H^{s-1})}\right).\]
\end{itemize}
\end{proposition}
\begin{proof}
We use a fixed-point argument in $C([0,T];\mathcal H^s)$. Since $W\in C^\infty(M;\mathbb R)$ and $-\Delta_g+W$ is positive, $S(t)$ is strongly continuous on $\mathcal H^s$ and, for every $T>0$,
\[\sup_{0\le t\le T}\|S(t)\|_{\mathcal L(\mathcal H^s,\mathcal H^s)}<\infty.\]
For $Y=(y,v)^\top\in C([0,T];\mathcal H^s)$, define the Duhamel map by
\[\Gamma(Y)(t){:=\Gamma_{(y_0,y_1)^\top,\eta}(Y)(t)}=S(t)\binom{y_0}{y_1}+\int_0^t
S(t-\sigma)\binom{0}{\eta(\sigma)-\kappa y(\sigma)^{p}}d\sigma.\]
The second component of $\Gamma(Y)$ is the time derivative of its first component, so every fixed point satisfies $v=\partial_t y$.
Since $s>d/2$, the space $H^s(M)$ is an algebra. Hence, for $y,z$ bounded in $H^{s}(M)$,
\[\|y^{p}-z^{p}\|_{H^{s-1}} \leq \|y^{p}-z^{p}\|_{H^s} \leq C_R\|y-z\|_{H^s}.\]
Moreover,
\[\|y^{p}\|_{H^{s-1}} \leq C\|y\|_{H^s}^{p}.\]
Let $R>0$ be fixed and assume
\[\left\|\binom{y_0}{y_1}\right\|_{\mathcal H^s}+\|\eta\|_{L^1(0,T;H^{s-1})}\le R.\]
For $Y\in C([0,T];\mathcal H^s)$ with $\|Y\|_{C([0,T];\mathcal H^s)}\le M$, the previous estimates yield
\[\|\Gamma(Y)\|_{C([0,T];\mathcal H^s)}\leq C_T R+C_T T M^{p}.\]
Choosing $M$ large enough and then $T>0$ small enough, $\Gamma$ maps the closed ball of radius $M$ in $C([0,T];\mathcal H^s)$ into itself. Similarly, if $Y,Z$ belong to this ball, then
\[\|\Gamma(Y)-\Gamma(Z)\|_{C([0,T];\mathcal H^s)} \leq C_T T M^{p-1}\|Y-Z\|_{C([0,T];\mathcal H^s)}.\]
Taking $T>0$ smaller if necessary, $\Gamma$ is a contraction. This gives a unique mild solution $Y=(y,\partial_t y)^\top\in C([0,T];\mathcal H^s)$. This proves local existence and uniqueness. The uniform local bound follows from the same construction, with $T_R$ and $C_R$ depending only on $R$.
\medskip

\noindent It remains to prove the stability estimate. Let $Y=(y,\partial_t y)^\top$ and $Z=(z,\partial_t z)^\top$ be two solutions on $[0,\varkappa]$, associated with data $((y_0,y_1)^\top,\eta)$ and $((z_0,z_1)^\top,\zeta)$, and assume
\[\|Y\|_{C([0,\varkappa];\mathcal H^s)}+\|Z\|_{C([0,\varkappa];\mathcal H^s)}\le R.\]
By Duhamel's formula,
\[Y(t)-Z(t)=S(t)\left(\binom{y_0}{y_1}-\binom{z_0}{z_1}\right)+\int_0^t S(t-\sigma)
\binom{0}{\eta(\sigma)-\zeta(\sigma)-\kappa(y(\sigma)^{p}-z(\sigma)^{p})}d\sigma.\]
Using the boundedness of $S(t)$ on $\mathcal H^s$ and the algebra estimate,
\[\|Y(t)-Z(t)\|_{\mathcal H^s}\le C_R\left\|\binom{y_0}{y_1}-\binom{z_0}{z_1}\right\|_{\mathcal H^s}+C_R\|\eta-\zeta\|_{L^1(0,t;H^{s-1})}+C_R\int_0^t\|Y(\sigma)-Z(\sigma)\|_{\mathcal H^s}d\sigma.\]
Gronwall's inequality yields
\[\|Y-Z\|_{C([0,\varkappa];\mathcal H^s)}\le C_R
\left(\left\|\binom{y_0}{y_1}-\binom{z_0}{z_1}\right\|_{\mathcal H^s}+\|\eta-\zeta\|_{L^1(0,\varkappa;H^{s-1})}\right).\]
This completes the proof.
\end{proof}
For $Y_0=(y_0,y_1)^\top\in\mathcal H^s$, we denote by $\mathcal R_t(Y_0,\eta)$ the value at time $t$ of the solution of \eqref{eq:KG} with initial datum $Y_0$ and control $\eta$. The uncontrolled Klein--Gordon flow is given by
$\mathcal R_t(Y_0,0)$. Finally, for $\Xi=(\xi_0,\xi_1)^\top\in\mathcal H^s$, we define the translation
\[\tau_\Xi(Y):=Y+\Xi.\]
As in \Cref{sec:additive-NLS}, the flow satisfies the following concatenation property. If $\eta$ is defined by
\[\eta(t)=
\begin{cases}
\eta_1(t), & t\in[0,t_1],\\
\eta_2(t-t_1), & t\in[t_1,t_1+t_2],
\end{cases}\]
then
\[\mathcal R_{t_1+t_2}(Y_0,\eta)
=\mathcal R_{t_2}\bigl(\mathcal R_{t_1}(Y_0,\eta_1),\eta_2\bigr).\]
In order to prove \Cref{thm:intro-KG-STAC}, we first show that the control directly generates translations in the velocity component. This is the analogue of the direct translations for the Schrödinger equation, but here the forcing enters the second-order equation, hence it acts at leading order on $\partial_t y$. More precisely, a control of size $\delta^{-1}$ acting during a time $\delta$ produces a finite displacement in the velocity variable, while its effect on the position variable vanishes as $\delta\to0^+$.
\begin{lemma}[Direct velocity translations]
\label{lem:KG-direct-velocity-translations}
For every $\xi\in E^\R$, the translation $\tau_{(0,\xi)^\top}$ is small-time approximately realizable for \eqref{eq:KG}.
\end{lemma}

\begin{proof}
Fix $Y_0=(y_0,y_1)^\top\in\mathcal H^s$ and $\xi\in E^\R$. Let $\varkappa>0$. For $\delta>0$, set $\eta_\delta(t):=\frac1\delta\xi$,  $0\le t\le\delta$. Then $\eta_\delta\in L^\infty(0,\delta;E^\R)$ and $\|\eta_\delta\|_{L^1(0,\delta;H^{s-1})}=\|\xi\|_{H^{s-1}}$. By the uniform local bound in \Cref{prop:KG-LWP}, the corresponding solution satisfies
\[\sup_{0\le t\le\delta}\left\|\binom{y(t)}{\partial_t y(t)}\right\|_{\mathcal H^s}
\le C(Y_0,\xi)\]
for all sufficiently small $\delta>0$. Using the mild formulation \eqref{eq:mildformulationwaves}, the contribution of the control to the position is, as $\delta$ goes to zero,
\[\frac1\delta
\int_0^\delta \Lambda^{-1}\sin((\delta-\sigma)\Lambda)\xi d\sigma
\longrightarrow 0
\qquad\text{in }H^{s},\]
whereas the contribution of the control to the velocity is
\[\frac1\delta\int_0^\delta \cos((\delta-\sigma)\Lambda)\xi d\sigma\longrightarrow \xi\qquad\text{in }H^{s-1}.\]
The linear free part converges to $(y_0,y_1)^\top$ in $\mathcal H^s$, and the nonlinear contribution converges to zero in $\mathcal H^s$, since
\[\int_0^\delta \|y(\sigma)^{p}\|_{H^{s-1}}d\sigma\le C\delta
\sup_{0\le\sigma\le\delta}\|y(\sigma)\|_{H^{s}}^{p}
\to0.\]
Therefore, we proved that
\[\mathcal R_\delta(Y_0,\eta_\delta)\longrightarrow Y_0+\binom{0}{\xi}
\qquad\text{in }\mathcal H^s.\]
Taking $\delta<\varkappa$ sufficiently small yields the result.
\end{proof}

We now recover translations in the position component. This is specific to the second-order nature of the Klein--Gordon equation. The idea is to first create a large velocity translation of size $\delta^{-1}\xi$, let the uncontrolled dynamics evolve during a time $\delta$, and then cancel the large velocity. At leading order, the position gains the finite displacement $\xi$, while the velocity returns to its original value.

\begin{lemma}[Position translations]
\label{lem:KG-position-translations}
For every $\xi\in E^\R$, the translation $\tau_{(\xi,0)^\top}$ is small-time approximately realizable for \eqref{eq:KG}.
\end{lemma}
\begin{proof}
Fix $Y_0=(y_0,y_1)^\top\in\mathcal H^s$, $\xi\in E^\R$, $\eps>0$ and $\varkappa>0$. Let $\delta>0$ be small. By \Cref{lem:KG-direct-velocity-translations}, the translations $\tau_{(0,\delta^{-1}\xi)^\top}$ and $\tau_{(0,-\delta^{-1}\xi)^\top}$ are small-time approximately realizable. We first prove that, for all sufficiently small $\delta>0$, the uncontrolled solution with initial datum $(y_0,y_1+\delta^{-1}\xi)^\top$ is well-defined on $[0,\delta]$ and 
\begin{equation}
\label{eq:limit_1_wave}
\mathcal R_\delta\left(\binom{y_0}{y_1+\delta^{-1}\xi},0\right)+\binom{0}{-\delta^{-1}\xi}
\longrightarrow
\binom{y_0+\xi}{y_1}
\qquad\text{in }\mathcal H^s.
\end{equation}
First, notice that local well-posedness does not directly ensure existence on $[0,\delta]$, since the initial datum may become unbounded as $\delta\to 0$. We therefore apply a fixed-point argument to the Duhamel map for the position component (as in \Cref{lem:NLS-asymptotic-limit}). Define, for $y\in C([0,\delta];H^{s})$
\begin{align*}
\Gamma_\delta(y)(t):=\cos(t\Lambda)y_0+\Lambda^{-1}\sin(t\Lambda)y_1+\delta^{-1}\Lambda^{-1}\sin(t\Lambda)\xi-\kappa\int_0^t\Lambda^{-1}\sin((t-\sigma)\Lambda)y(\sigma)^{p}d\sigma.
\end{align*}
Since $\xi\in E^\R\subset C^\infty(M)$, we have $\delta^{-1}\Lambda^{-1}\sin(\delta\Lambda)\xi\to \xi$ in $H^{s}$. Also, $\cos(\delta\Lambda)y_0\to y_0$ in $H^{s}$ and $\Lambda^{-1}\sin(\delta\Lambda)y_1\to0$ in $H^{s}$. We now show that the Duhamel map admits a fixed point in a fixed ball of $C([0,\delta];H^{s})$. Indeed, the only large linear term is $\delta^{-1}\Lambda^{-1}\sin(t\Lambda)\xi$, and for $0\le t\le\delta$,
\[\|\delta^{-1}\Lambda^{-1}\sin(t\Lambda)\xi\|_{H^{s}}\le C\delta^{-1}t\|\xi\|_{H^{s}}
\le C\|\xi\|_{H^{s}}.\]
Although the full linear state has velocity of size $\delta^{-1}$, the position component is uniformly bounded on $[0,\delta]$. Indeed, the nonlinear term only depends on $y$, and the linear position satisfies
\[\sup_{0\le t\le\delta}
\left\|\cos(t\Lambda)y_0+\Lambda^{-1}\sin(t\Lambda)y_1+\delta^{-1}\Lambda^{-1}\sin(t\Lambda)\xi\right\|_{H^{s}}
\le C(Y_0,\xi).\]
Since $H^s(M)$ is an algebra, for $\|y\|_{C([0,\delta];H^{s})}\leq R$
\[\|\Gamma_\delta(y)\|_{C([0,\delta];H^{s})}\leq C(Y_0,\xi)+C\delta R^{p}.\]
Moreover, for $\|y\|_{C([0,\delta];H^{s})},\|z\|_{C([0,\delta];H^{s})}\leq R$,
\[\|\Gamma_\delta(y)-\Gamma_\delta(z)\|_{C([0,\delta];H^{s})}\leq C\delta R^{p-1}\|y-z\|_{C([0,\delta];H^{s})}.\]
Choosing $R$ sufficiently large and then $\delta>0$ sufficiently small, the Duhamel map leaves the closed ball of radius $R$ invariant and it is a contraction there. Hence, by the Banach fixed-point theorem, $\Gamma_\delta$ admits a unique fixed point on that ball and 
\[\sup_{0\le t\le\delta}\|y(t)\|_{H^{s}} \le C(Y_0,\xi).\]
The integral formulation also gives $y\in C^{1}([0,\delta];H^{s-1})$. Therefore, $(y,\partial_t y)^\top \in C([0,\delta];\mc H^{s})$ and satisfies the mild formulation with initial datum $(y_0,y_1+\delta^{-1}\xi)^\top$ and no control. By uniqueness, 
\[\binom{y(t)}{\partial_t y(t)}=\mathcal R_t\left(\binom{y_0}{y_1+\delta^{-1}\xi},0\right), \qquad 0\leq t\leq \delta.\]
The uniform bound on $y$ gives 
\[\left\|\int_0^\delta\Lambda^{-1}\sin((\delta-\sigma)\Lambda)y(\sigma)^{p}d\sigma\right\|_{H^{s}}\le C\delta\sup_{0\le\sigma\le\delta}\|y(\sigma)\|_{H^{s}}^{p}\to0.\]
Therefore $y(\delta)\to y_0+\xi$ in $H^{s}$. For the velocity component, again by the mild formulation,
\[\partial_t y(\delta)=-\Lambda\sin(\delta\Lambda)y_0+\cos(\delta\Lambda)y_1+\delta^{-1}\cos(\delta\Lambda)\xi-\kappa\int_0^\delta\cos((\delta-\sigma)\Lambda)y(\sigma)^{p}d\sigma.
\]
Hence
\[\partial_t y(\delta)-\delta^{-1}\xi=-\Lambda\sin(\delta\Lambda)y_0+\cos(\delta\Lambda)y_1+\delta^{-1}(\cos(\delta\Lambda)-I)\xi-\kappa\int_0^\delta\cos((\delta-\sigma)\Lambda)y(\sigma)^{p}d\sigma.\]
The first term tends to $0$ in $H^{s-1}$, the second one tends to $y_1$ in $H^{s-1}$, and the nonlinear term tends to $0$ in $H^{s-1}$. Finally, since $\xi$ is smooth,
\[\delta^{-1}(\cos(\delta\Lambda)-I)\xi\to0
\qquad\text{in }H^{s-1}.\]
Thus $\partial_t y(\delta)-\delta^{-1}\xi\to y_1$ in $H^{s-1}$, and \eqref{eq:limit_1_wave} follows. Choose $\delta\in(0,\varkappa/3)$ small enough so that
\[\left\|\mathcal R_\delta\left(\binom{y_0}{y_1+\delta^{-1}\xi},0\right)+\binom{0}{-\delta^{-1}\xi}
-\binom{y_0+\xi}{y_1}\right\|_{\mathcal H^s}<\frac{\eps}{3}.\]
By the same fixed-point argument and continuous dependence on the initial datum, there exists $\rho>0$ such that, if $\|Z-(y_0,y_1+\delta^{-1}\xi)^\top\|_{\mathcal H^s}<\rho$, then the uncontrolled solution with initial datum $Z$ is well defined on $[0,\delta]$ and
\[\left\|\mathcal R_\delta(Z,0)-\mathcal R_\delta\left(\binom{y_0}{y_1+\delta^{-1}\xi},0\right)\right\|_{\mathcal H^s}<\frac{\eps}{3}.\]
Since $\tau_{(0,\delta^{-1}\xi)^\top}$ is small-time approximately realizable, there exist $T_1\in(0,\varkappa/3)$ and $\eta_1\in L^\infty(0,T_1;\allowbreak E^\R)$ such that
\[\left\|\mathcal R_{T_1}(Y_0,\eta_1)
-\binom{y_0}{y_1+\delta^{-1}\xi}
\right\|_{\mathcal H^s}<\rho.\]
Set $Z:=\mathcal R_{T_1}(Y_0,\eta_1)$. On $[T_1,T_1+\delta]$, we take zero control, so the state becomes $\mathcal R_\delta(Z,0)$. Finally, since $\tau_{(0,-\delta^{-1}\xi)^\top}$ is small-time approximately realizable, there exist $T_2\in(0,\varkappa/3)$ and $\eta_2\in L^\infty(0,T_2;E^\R)$ such that
\[\left\|\mathcal R_{T_2}(\mathcal R_\delta(Z,0),\eta_2)-\left(\mathcal R_\delta(Z,0)+\binom{0}{-\delta^{-1}\xi}\right)\right\|_{\mathcal H^s}<\frac{\eps}{3}.\]
Concatenating the three pieces, with $T:=T_1+\delta+T_2<\varkappa$, we obtain
\[\left\|\mathcal R_T(Y_0,\eta)-\left(Y_0+\binom{\xi}{0}\right)\right\|_{\mathcal H^s}<\eps.\]
Therefore $\tau_{(\xi,0)^\top}$ is small-time approximately realizable.
\end{proof}
We now prove the nonlinear generation step. Since the Klein--Gordon equation is real-valued, the parameter $\alpha$ is real. The idea is to insert a large position translation of size $\delta^{{-1/p}}\alpha\varphi$, let the uncontrolled equation evolve during time $\delta$, and then remove the same position translation. The position returns to its initial value, while the velocity gains a finite contribution generated by the polynomial term.

\begin{lemma}[Nonlinear kick]
\label{lem:KG-odd-kick}
Let $\varphi\in C^\infty(M;\R)$ and $\alpha\in\R$. Assume that, for every $\delta>0$, the translations $\tau_{(\pm\delta^{-\frac1{p}}\alpha\varphi,0)^\top}$ are small-time approximately realizable for \eqref{eq:KG}. Then the translation $\tau_{(0,-\kappa\alpha^{p}\varphi^{p})^\top}$ is small-time approximately realizable for \eqref{eq:KG}.
\end{lemma}
\begin{proof}
Set $\chi:=(0,-\kappa\alpha^{p}\varphi^{p})^\top$ and fix $Y_0=(y_0,y_1)^\top\in\mathcal H^s$, $\eps>0$ and $\varkappa>0$. We first prove that, for all sufficiently small $\delta>0$, the uncontrolled solution with initial datum $(y_0+\delta^{-\frac1{p}}\alpha\varphi,y_1)^\top$ is well defined on $[0,\delta]$ and
\begin{equation}
\label{eq:KG-odd-kick-limit}
\mathcal R_\delta\left(\binom{y_0+\delta^{-\frac1{p}}\alpha\varphi}{y_1},0\right)-\binom{\delta^{-\frac1{p}}\alpha\varphi}{0}\longrightarrow Y_0+\chi\qquad\text{in }\mathcal H^s.
\end{equation}
As in the proof of \Cref{lem:NLS-asymptotic-limit}, we first construct a solution of the renormalized equation. Define on $C([0,\delta];\mathcal H^s)$
\[\widetilde{\Gamma}_\delta(Z):=\Gamma_{(y_0+\delta^{-\frac1{p}}\alpha\varphi,y_1)^\top,0}\left(Z+\binom{\delta^{-\frac1{p}}\alpha\varphi}{0}\right)-\binom{\delta^{-\frac1{p}}\alpha\varphi}{0},\]
where $\Gamma_{Y_0,\eta}$ is the Duhamel map associated with initial datum $Y_0$ and control $\eta$. Since $\varphi$ is smooth, the linear contribution of the translation satisfies
\[\sup_{0\le t\le\delta}\left\|\binom{(\cos(t\Lambda)-I)\delta^{-\frac1{p}}\alpha\varphi}{-\Lambda\sin(t\Lambda)\delta^{-\frac1{p}}\alpha\varphi}\right\|_{\mathcal H^s}\le C\delta^{\frac{p-1}{p}}.\]
Moreover, since $H^s(M)$ is an algebra, for $\|z\|_{H^s}\le R$,
\[\left\|(\delta^{-\frac1{p}}\alpha\varphi+z)^{p}\right\|_{H^{s-1}}\le C(\delta^{-1}+R^{p}),\]
while its Lipschitz constant on such a ball is bounded by $C(\delta^{-\frac{p-1}{p}}+R^{p-1})$. After integration on $[0,\delta]$, these estimates give
\[\|\widetilde{\Gamma}_\delta(Z)\|_{C([0,\delta];\mathcal H^s)}\le C(Y_0,\alpha,\varphi)+C\delta R^{p}, \quad \text{ for }\|Z\|_{C([0,\delta];\mathcal H^s)}\le R\] 
and
\[\|\widetilde{\Gamma}_\delta(Z_1)-\widetilde{\Gamma}_\delta(Z_2)\|_{C([0,\delta];\mathcal H^s)}\le C\left(\delta^{\frac1{p}}+\delta R^{p-1}\right)\|Z_1-Z_2\|_{C([0,\delta];\mathcal H^s)}\]
for $\|Z_1\|_{C([0,\delta];\mathcal H^s)},\|Z_2\|_{C([0,\delta];\mathcal H^s)}\le R$. Choosing $R$ sufficiently large and then $\delta>0$ sufficiently small, the map leaves the closed ball of radius $R$ invariant and is a contraction there. Hence it admits a unique fixed point in that ball. By the integral formulation, its first component $z^\delta$ belongs to $C^1([0,\delta];H^{s-1})$, and its second component is $\partial_t z^\delta$.
\noindent  Set $y^\delta(t):=z^\delta(t)+\delta^{-\frac1{p}}\alpha\varphi$. Then $(y^{\delta},\partial_t y^{\delta})^\top$ satisfies the mild formulation with initial datum $(y_0+\delta^{-\frac1{p}}\alpha\varphi,y_1)^\top$ and no control. By uniqueness
\[\binom{y^\delta(t)}{\partial_t y^\delta(t)}=\mathcal R_t\left(\binom{y_0+\delta^{-\frac1{p}}\alpha\varphi}{y_1},0\right), \quad 0\leq t\leq \delta.\] 
Moreover
\[\sup_{0\le t\le\delta}\left\|\binom{z^\delta(t)}{\partial_t y^\delta(t)}\right\|_{\mathcal H^s}\le C\]
for all sufficiently small $\delta$. Therefore, uniformly for
$0\le t\le\delta$,
\[(y^\delta(t))^{p}=\delta^{-1}\alpha^{p}\varphi^{p}+O_{H^s}\left(\delta^{-\frac{p-1}{p}}\right).\]
For the position component, the linear contribution satisfies
\[\cos(\delta\Lambda)(y_0+\delta^{-\frac1{p}}\alpha\varphi)+\Lambda^{-1}\sin(\delta\Lambda)y_1-\delta^{-\frac1{p}}\alpha\varphi\longrightarrow y_0\qquad\text{in }H^{s},\]
and the nonlinear contribution tends to zero in $H^{s}$, since
\[\int_0^\delta\left\|\Lambda^{-1}\sin((\delta-\sigma)\Lambda)(y^\delta(\sigma))^{p}\right\|_{H^{s}}d\sigma\le C\int_0^\delta(\delta-\sigma)\|(y^\delta(\sigma))^{p}\|_{H^s}d\sigma\to0.\]
Thus the first component converges to $y_0$. For the velocity component,
\[\partial_t y^\delta(\delta)=-\Lambda\sin(\delta\Lambda)(y_0+\delta^{-\frac1{p}}\alpha\varphi)+\cos(\delta\Lambda)y_1-\kappa\int_0^\delta\cos((\delta-\sigma)\Lambda)(y^\delta(\sigma))^{p}d\sigma.\]
The linear terms converge to $y_1$ in $H^{s-1}$. Moreover,
\[\int_0^\delta\cos((\delta-\sigma)\Lambda)(y^\delta(\sigma))^{p}d\sigma=\delta^{-1}\alpha^{p}\int_0^\delta\cos((\delta-\sigma)\Lambda)\varphi^{p}d\sigma+O_{H^{s-1}}\left(\delta^{\frac1{p}}\right).\]
By strong continuity of $\cos(t\Lambda)$ on $H^{s-1}$,
\[\frac1\delta\int_0^\delta\cos((\delta-\sigma)\Lambda)\varphi^{p}d\sigma\longrightarrow\varphi^{p}\qquad\text{in }H^{s-1}.\]
Hence
\[\partial_t y^\delta(\delta)\longrightarrow y_1-\kappa\alpha^{p}\varphi^{p}\qquad\text{in }H^{s-1}.\]
This proves \eqref{eq:KG-odd-kick-limit}.
\medskip

\noindent Choose $\delta\in(0,\varkappa/3)$ small enough so that the error in \eqref{eq:KG-odd-kick-limit} is smaller than $\eps/3$. By the same fixed-point argument and continuous dependence on the initial datum, there exists $\rho>0$ such that if
\[\left\|Z-\binom{y_0+\delta^{-\frac1{p}}\alpha\varphi}{y_1}\right\|_{\mathcal H^s}<\rho,\]
then the uncontrolled solution with initial datum $Z$ is well defined on $[0,\delta]$ and
\[\left\|\mathcal R_\delta(Z,0)-\mathcal R_\delta\left(\binom{y_0+\delta^{-\frac1{p}}\alpha\varphi}{y_1},0\right)\right\|_{\mathcal H^s}<\frac{\eps}{3}.\]
Since $\tau_{(\delta^{-\frac1{p}}\alpha\varphi,0)^\top}$ is small-time approximately realizable, there exist $T_1\in(0,\varkappa/3)$ and $\eta_1\in L^\infty(0,T_1;\allowbreak E^\R)$ such that
\[\left\|\mathcal R_{T_1}(Y_0,\eta_1)-\binom{y_0+\delta^{-\frac1{p}}\alpha\varphi}{y_1}\right\|_{\mathcal H^s}<\rho.\]
Set $Z:=\mathcal R_{T_1}(Y_0,\eta_1)$. On $[T_1,T_1+\delta]$, we take zero control, so the state becomes $\mathcal R_\delta(Z,0)$. Finally, since
$\tau_{(-\delta^{-\frac1{p}}\alpha\varphi,0)^\top}$ is small-time approximately realizable, there exist $T_2\in(0,\varkappa/3)$ and $\eta_2\in L^\infty(0,T_2;E^\R)$ such that
\[\left\|\mathcal R_{T_2}(\mathcal R_\delta(Z,0),\eta_2)-\left(\mathcal R_\delta(Z,0)-\binom{\delta^{-\frac1{p}}\alpha\varphi}{0}\right) \right\|_{\mathcal H^s}<\frac{\eps}{3}.\]
Concatenating the three pieces, with
$T:=T_1+\delta+T_2<\varkappa$, we obtain
\[\|\mathcal R_T(Y_0,\eta)-(Y_0+\chi)\|_{\mathcal H^s}<\eps.\]
Therefore $\tau_\chi$ is small-time approximately realizable.
\end{proof}
\begin{lemma}[Velocity-to-position  translations]
\label{lem:KG-velocity-position-translations}
Let $\xi\in C^\infty(M;\R)$. Assume that, for every $\delta>0$, the translations $\tau_{(0,\pm\delta^{-1}\xi)^\top}$ are small-time approximately realizable for \eqref{eq:KG}. Then the translation $\tau_{(\xi,0)^\top}$ is small-time approximately realizable for \eqref{eq:KG}.
\end{lemma}

\begin{proof}
The proof is the same as the proof of \Cref{lem:KG-position-translations}, replacing the use of \Cref{lem:KG-direct-velocity-translations} by the present assumption.
\end{proof}

\begin{lemma}[Stability under sums and limits]
\label{lem:KG-sums}
If $\tau_{\Xi_1}$ and $\tau_{\Xi_2}$ are small-time approximately realizable for \eqref{eq:KG}, then $\tau_{\Xi_1+\Xi_2}$ is small-time approximately realizable. Moreover, if $\Xi_j\to\Xi$ in $\mathcal H^s$ and every $\tau_{\Xi_j}$ is small-time approximately realizable, then $\tau_\Xi$ is small-time approximately realizable.
\end{lemma}

\begin{proof}
Fix $Y_0\in\mathcal H^s$, $\eps>0$ and $\varkappa>0$. First realize $\tau_{\Xi_1}$ from $Y_0$ in time $T_1<\varkappa/2$ with error small enough. From the obtained point, realize $\tau_{\Xi_2}$ in time $T_2<\varkappa/2$ with error small enough. By concatenation and the triangle inequality, the final point is arbitrarily close to $Y_0+\Xi_1+\Xi_2$ in time $T_1+T_2<\varkappa$. For the second assertion, choose $j$ such that $\|\Xi_j-\Xi\|_{\mathcal H^s}<\eps/2$, and realize $\tau_{\Xi_j}$ from $Y_0$ in time smaller than $\varkappa$, with error smaller than $\eps/2$. The triangle inequality gives the result.
\end{proof}

From now on we assume $\kappa\neq0$ and use the saturation spaces $\mathsf E_n^\R$ defined in \eqref{eq:odd-saturation} when $p$ is odd, and $\mathsf F_n^\R$ defined in \eqref{eq:even-saturation}, with $r=s$, when $p$ is even.

\begin{lemma}[Realizability of saturated Klein--Gordon directions]
\label{lem:KG-saturated-directions}
Let $\xi\in\mathsf E_\infty^\R$ if $p$ is odd, and
$\xi\in\mathsf F_\infty^\R$ if $p$ is even. Then both translations $\tau_{(\xi,0)^\top}$ and $\tau_{(0,\xi)^\top}$ are small-time approximately realizable for \eqref{eq:KG}.
\end{lemma}

\begin{proof}
We argue by induction on $n$. For $n=0$, the translation $\tau_{(0,\xi)^\top}$ follows from \Cref{lem:KG-direct-velocity-translations}, while $\tau_{(\xi,0)^\top}$ follows from \Cref{lem:KG-position-translations}. Assume that the result holds for some $n\ge0$.
Let $\varphi\in\mathsf E_n^\R$ if $p$ is odd, and $\varphi\in\mathsf F_n^\R$ if $p$ is even.
Then, for every $\delta>0$ and every $\alpha\in\R$, the induction hypothesis gives the realizability of the translations $\tau_{(\pm\delta^{-1/p}\alpha\varphi,0)^\top}$. By \Cref{lem:KG-odd-kick}, the translation $\tau_{(0,-\kappa\alpha^{p}\varphi^{p})^\top}$ is small-time approximately realizable. If $p$ is odd, since $\kappa\neq0$ and the map $\alpha\mapsto-\kappa\alpha^{p}$ is onto $\R$, it follows that $\tau_{(0,\beta\varphi^{p})^\top}$ is small-time approximately realizable for every $\beta\in\R$. Applying \Cref{lem:KG-velocity-position-translations}, we also get that $\tau_{(\beta\varphi^{p},0)^\top}$ is small-time approximately realizable for every $\beta\in\R$. Combining this with the induction hypothesis and \Cref{lem:KG-sums}, we obtain that both $\tau_{(\xi,0)^\top}$ and $\tau_{(0,\xi)^\top}$ are small-time approximately realizable for every $\xi\in \mathsf E_{n+1}^\R$.

If $p$ is even, taking $\alpha=|\kappa|^{-1/p}$ gives the realizability of $\tau_{(0,-\operatorname{sgn}(\kappa)\varphi^p)^\top}$. Since $\mathsf F_n^\R$ is a vector space, the induction hypothesis and \Cref{lem:KG-sums} give the realizability of $\tau_{(0,\operatorname{sgn}(\kappa)\psi)^\top}$ for every $\psi\in\mathsf C_n$. Stability under limits gives the same conclusion for $\psi\in\overline{\mathsf C_n}^{H^s(M;\R)}$. For $\xi\in \mathsf F_{n+1}^\R$, both $\xi$ and $-\xi$ belong to this closure, so $\tau_{(0,\xi)^\top}$ is small-time approximately realizable. Since $\mathsf F_{n+1}^\R$ is a vector space of smooth functions, this also applies to $\pm\delta^{-1}\xi$ for every $\delta>0$. Applying \Cref{lem:KG-velocity-position-translations}, we obtain that $\tau_{(\xi,0)^\top}$ is small-time approximately realizable. This completes the induction in the even case.
\end{proof}

\begin{proposition}[Klein--Gordon saturation criterion]
\label{prop:KG-saturation-criterion}

Assume that the corresponding saturation, $\mathsf E_\infty^\R$ when $p$ is odd or $\mathsf F_\infty^\R$ when $p$ is even, is dense in $H^s(M;\R)$. Then \eqref{eq:KG} is small-time approximately controllable in $\mathcal H^s$.
\end{proposition}

\begin{proof}
Let $Y_0=(y_0,y_1)^\top$, $Y_1=(z_0,z_1)^\top\in\mathcal H^s$,
$\eps>0$ and $\varkappa>0$. By density, there exist $\xi_0,\xi_1\in\mathsf E_\infty^\R$ if $p$ is odd, or $\xi_0,\xi_1\in\mathsf F_\infty^\R$ if $p$ is even, such that 
\[\|z_0-y_0-\xi_0\|_{H^{s}}+\|z_1-y_1-\xi_1\|_{H^{s-1}}<\frac{\eps}{2}.\]
By \Cref{lem:KG-saturated-directions}, the translations $\tau_{(\xi_0,0)^\top}$ and $\tau_{(0,\xi_1)^\top}$ are small-time approximately realizable. By \Cref{lem:KG-sums}, the translation $\tau_{(\xi_0,\xi_1)^\top}$ is small-time approximately realizable. Hence there exist $T\in(0,\varkappa)$ and $\eta\in L^\infty(0,T;E^\R)$ such that
\[\left\|\mathcal R_T(Y_0,\eta)-\left(Y_0+\binom{\xi_0}{\xi_1}\right)\right\|_{\mathcal H^s}<\frac{\eps}{2}.\]
Therefore
\[\|\mathcal R_T(Y_0,\eta)-Y_1\|_{\mathcal H^s}<\eps.\]
\end{proof}

\textit{Proof of \Cref{thm:intro-KG-STAC}:}
By \Cref{th:odd-saturation-density}, the corresponding saturation of $E^\R$, namely $\mathsf E_\infty^\R$ when $p$ is odd and $\mathsf F_\infty^\R$ when $p$ is even, is dense in $H^s(M;\R)$ and in $H^{s-1}(M;\R)$. The conclusion follows directly from \Cref{prop:KG-saturation-criterion}.
\qed
\subsection{Parabolic equation}
\label{sec:additive-Heat}
Let $(M,g)$ be a smooth compact connected Riemannian manifold without boundary, and let $W\in C^\infty(M;\R)$. Let $P\in\R[X]$ be a real polynomial such that
\[P(\lambda)\to+\infty
\qquad\text{as }\lambda\to+\infty,\]
and set
\[A:=P(-\Delta_g)+W.\]
Then $A$ is self-adjoint and bounded from below on $L^2(M)$, and $-A$
generates a strongly continuous analytic semigroup $S(t):=e^{-tA}$. In particular, for every $s\ge0$ and every $T>0$,
\[\sup_{0\le t\le T}\|S(t)\|_{\mathcal L(H^s,H^s)}<\infty.\]
As in \Cref{sec:additive-KG}, we take $E^\R=\operatorname{span}_{\R}\{1,V_1,\ldots,V_m\}$ when $p$ is odd, and $E^\R=\operatorname{span}_{\R}\{1,V_1,\ldots,V_{m+1}\}$ when $p$ is even, where $V_{m+1}$ is defined in \eqref{eq:V_m+1}. We consider the controlled polynomial parabolic equation
\begin{equation}
\label{eq:heat}
\left\{
\begin{aligned}
&\partial_t y+Ay+\kappa y^{p}=\eta(t,x),\qquad (t,x)\in(0,T)\times M,\\
&y(0)=y_0,
\end{aligned}
\right.
\end{equation}
where $\kappa\in\R\setminus\{0\}$. Equivalently,
\[\partial_t y+P(-\Delta_g)y+Wy+\kappa y^{p}=\eta(t,x).\]
\begin{proposition}[Local well-posedness for the parabolic equation]
\label{prop:heat-LWP}
Let $s>d/2$, $T>0$, $y_0\in H^s(M;\R)$ and
$\eta\in L^1(0,T;H^s(M;\R))$. Then there exists $T_*\in(0,T]$ such that
\eqref{eq:heat} admits a unique mild solution $y\in C([0,T_*];H^s(M;\R))$, given by
\[y(t)=S(t)y_0+\int_0^t S(t-\sigma)\eta(\sigma)d\sigma-\kappa\int_0^t S(t-\sigma)y(\sigma)^{p}d\sigma.\]
Moreover, for every $R>0$, there exist ${T_R\in(0,T]}$ and $C_R>0$ such that, if
\[\|y_0\|_{H^s}+\|\eta\|_{L^1(0,T_R;H^s)}\le R,\]
then the corresponding solution is defined on $[0,T_R]$ and satisfies
\[\|y\|_{C([0,T_R];H^s)}\le C_R.\]
Finally, the solution depends continuously on the initial datum and on the
control in the following sense: if $y,z$ are two solutions on $[0,\varkappa]$ with
\[\|y\|_{C([0,\varkappa];H^s)}+\|z\|_{C([0,\varkappa];H^s)}\le R,\]
associated with data $(y_0,\eta)$ and $(z_0,\zeta)$, then
\[\|y-z\|_{C([0,\varkappa];H^s)}\le C_R\left(\|y_0-z_0\|_{H^s}+\|\eta-\zeta\|_{L^1(0,\varkappa;H^s)}\right).\]
\end{proposition}

\begin{proof}
Since $s>d/2$, $H^s(M)$ is an algebra and the map $y\mapsto y^{p}$ is locally
Lipschitz from $H^s(M)$ into itself. The result follows from the standard fixed
point argument applied to the mild formulation, using the boundedness of
$S(t)$ on $H^s(M)$ on bounded time intervals. The uniform bound and continuous
dependence estimate follow from the same contraction argument and Gronwall's
inequality.
\end{proof}

Whenever the solution of \eqref{eq:heat} with initial datum $y_0$ and control $\eta$ is defined, we denote by $\mathcal R_t(y_0,\eta)$ its value at time $t$. The uncontrolled flow is given by $\mathcal R_t(y_0,0)$.
\begin{lemma}[Direct translations for the parabolic equation]
\label{lem:heat-direct-translations}
For every $\xi\in E^\R$, the translation $\tau_\xi$ is
small-time approximately realizable for \eqref{eq:heat}.
\end{lemma}

\begin{proof}
Fix $y_0\in H^s(M;\R)$ and $\xi\in E^\R$. For $\delta>0$, set $\eta_\delta(t):=\delta^{-1}\xi$, $0\le t\le\delta$. Then $\eta_\delta\in L^\infty(0,\delta;E^\R)$ and
$\|\eta_\delta\|_{L^1(0,\delta;H^s)}=\|\xi\|_{H^s}$. By \Cref{prop:heat-LWP}, the corresponding solution is uniformly bounded in $C([0,\delta];H^s)$ for $\delta>0$ small. Hence, using the mild formulation,
\[y(\delta)=S(\delta)y_0+\frac1\delta\int_0^\delta S(\delta-\sigma)\xi d\sigma-\kappa\int_0^\delta S(\delta-\sigma)y(\sigma)^{p}d\sigma.\]
By the strong continuity of $S(t)$ on $H^s$,
$S(\delta)y_0\to y_0$, and $\dis\frac1\delta\int_0^\delta S(\delta-\sigma)\xi d\sigma\to\xi$
in $H^s$. Moreover, since $H^s(M)$ is an algebra,
\[\left\|\int_0^\delta S(\delta-\sigma)y(\sigma)^{p}d\sigma\right\|_{H^s}\le C\delta\sup_{0\le\sigma\le\delta}\|y(\sigma)\|_{H^s}^{p}\to0.\]
Therefore $y(\delta)\to y_0+\xi$ in $H^s(M;\R)$.
\end{proof}

\begin{lemma}[Parabolic nonlinear kick]
\label{lem:heat-odd-kick}
Let $\varphi\in C^\infty(M;\R)$ and $\alpha\in\R$. Assume that, for every $\delta>0$, the translations $\tau_{\pm\delta^{-\frac1{p}}\alpha\varphi}$ are small-time approximately realizable for \eqref{eq:heat}. Then the translation $\tau_{-\kappa\alpha^{p}\varphi^{p}}$ is small-time approximately realizable.
\end{lemma}
\begin{proof}
We first prove that, for every $y_0\in H^s(M;\R)$ and all sufficiently small $\delta>0$, the uncontrolled solution with initial datum $y_0+\delta^{-\frac1{p}}\alpha\varphi$ is well defined on $[0,\delta]$ and
\[\mathcal R_\delta\left(y_0+\delta^{-\frac1{p}}\alpha\varphi,0\right)-\delta^{-\frac1{p}}\alpha\varphi\longrightarrow y_0-\kappa\alpha^{p}\varphi^{p}\quad\text{in }H^s(M;\R).\]
Define on $C([0,\delta];H^s)$
\[\widetilde{\Gamma}_\delta(w):=\Gamma_{y_0+\delta^{-\frac1{p}}\alpha\varphi,0}\left(w+\delta^{-\frac1{p}}\alpha\varphi\right)-\delta^{-\frac1{p}}\alpha\varphi,\]
where $\Gamma_{y_0,\eta}$ is the Duhamel map for \eqref{eq:heat} associated with initial datum $y_0$ and control $\eta$. Since $\varphi$ is smooth,
\[\sup_{0\le t\le\delta}\left\|(S(t)-I)\delta^{-\frac1{p}}\alpha\varphi\right\|_{H^s}\le C\delta^{\frac{p-1}{p}}.\]
Using the boundedness of $S(t)$ on $H^s$ for $0\le t\le1$, the same fixed-point argument as in \Cref{lem:NLS-asymptotic-limit}, using the algebra property of $H^s(M)$ for the polynomial nonlinearity $y^p$, gives a fixed point $w^\delta\in C([0,\delta];H^s)$ satisfying
\[\sup_{0\le t\le\delta}\|w^\delta(t)\|_{H^s}\le C\]
for all sufficiently small $\delta>0$, with $C$ independent of $\delta$.
Set $y^\delta(t):=w^\delta(t)+\delta^{-\frac1{p}}\alpha\varphi$. By construction, $y^\delta$ satisfies the mild formulation with initial datum $y_0+\delta^{-\frac1{p}}\alpha\varphi$ and no control. By uniqueness,
\[y^\delta(t)=\mathcal R_t\left(y_0+\delta^{-\frac1{p}}\alpha\varphi,0\right),\qquad 0\le t\le\delta.\]

Therefore, uniformly for $0\le t\le\delta$,
\[(y^\delta(t))^{p}=\left(\delta^{-\frac1{p}}\alpha\varphi+w^\delta(t)\right)^{p}=
\delta^{-1}\alpha^{p}\varphi^{p}+O_{H^s}\left(\delta^{-\frac{p-1}{p}}\right).\]
It follows that
\[\int_0^\delta S(\delta-\sigma)(y^\delta(\sigma))^{p}d\sigma=\delta^{-1}\alpha^{p}\int_0^\delta S(\delta-\sigma)\varphi^{p}d\sigma+O_{H^s}\left(\delta^{\frac1{p}}\right).\]
By the strong continuity of $S(t)$ on $H^s$,
\[\frac1\delta\int_0^\delta S(\delta-\sigma)\varphi^{p}d\sigma\longrightarrow\varphi^{p}\quad\text{in }H^s.\]
Thus the desired asymptotic limit follows. Now fix
$y_0\in H^s(M;\R)$, $\varepsilon>0$ and $\varkappa>0$. Choose
$\delta\in(0,\varkappa/3)$ such that
\[\left\|\mathcal R_\delta\left(y_0+\delta^{-\frac1{p}}\alpha\varphi,0\right)-\delta^{-\frac1{p}}\alpha\varphi-\left(y_0-\kappa\alpha^{p}\varphi^{p}\right)\right\|_{H^s}<\frac{\varepsilon}{3}.\]
By the same fixed-point argument and continuous dependence of the flow, choose $\rho>0$ such that if
\[\left\|z-\left(y_0+\delta^{-\frac1{p}}\alpha\varphi\right)\right\|_{H^s}<\rho,\]
then the uncontrolled solution with initial datum $z$ is well defined on $[0,\delta]$ and
\[\left\|\mathcal R_\delta(z,0)-\mathcal R_\delta\left(y_0+\delta^{-\frac1{p}}\alpha\varphi,0\right)\right\|_{H^s}<\frac{\varepsilon}{3}.\]
Using the realizability of
$\tau_{\delta^{-\frac1{p}}\alpha\varphi}$, we reach such a point $z$ from $y_0$ in time $T_1<\varkappa/3$. Then we let the system evolve freely during time $\delta$. Finally, using the realizability of $\tau_{-\delta^{-\frac1{p}}\alpha\varphi}$, we remove the large
translation in time $T_2<\varkappa/3$. By concatenation, the final state is within $\varepsilon$ of $y_0-\kappa\alpha^{p}\varphi^{p}$. Hence $\tau_{-\kappa\alpha^{p}\varphi^{p}}$ is small-time approximately realizable.
\end{proof}
\textit{Proof of \Cref{thm:heat-STAC}.}
We use the saturation spaces $\mathsf E_n^\R$ defined in \eqref{eq:odd-saturation} for odd $p$, and $\mathsf F_n^\R$ defined in \eqref{eq:even-saturation}, with $r=s$, for even $p$. The direct translations in the directions of $E^\R$ are realizable by \Cref{lem:heat-direct-translations}.
We argue by induction on the saturation spaces. Assume that every direction in the corresponding space at step $n$ is small-time approximately realizable, and let $\varphi$ belong to this space. Then, for every $\delta>0$ and every
$\alpha\in\R$, the translations $\tau_{\pm\delta^{-\frac1{p}}\alpha\varphi}$ are small-time approximately realizable. Hence, by
\Cref{lem:heat-odd-kick}, the direction $-\kappa\alpha^{p}\varphi^{p}$ is realizable. If $p$ is odd, since $\kappa\ne0$ and the map
$\alpha\mapsto-\kappa\alpha^{p}$ is onto $\R$, all real multiples of $\varphi^{p}$ are realizable. By stability under sums as in \Cref{lem:NLS-sums}, every direction in $\mathsf E_{n+1}^\R$ is small-time approximately realizable. If $p$ is even, taking $\alpha=|\kappa|^{-1/p}$ gives the realizability of the direction $-\operatorname{sgn}(\kappa)\varphi^p$. Since $\mathsf F_n^\R$ is a vector space, the induction hypothesis and stability under sums give realizability in every direction of $\operatorname{sgn}(\kappa)\mathsf C_n$. By stability under limits in $H^s$, as in \Cref{lem:NLS-sums}, this also holds for every direction of $\operatorname{sgn}(\kappa)\overline{\mathsf C_n}^{H^s(M;\R)}$. Since $\mathsf F_{n+1}^\R\subset\operatorname{sgn}(\kappa)\overline{\mathsf C_n}^{H^s(M;\R)}$, every direction in $\mathsf F_{n+1}^\R$ is small-time approximately realizable.\par\medskip

\noindent Therefore, by induction, every direction in $\mathsf E_\infty^\R$ when $p$ is odd, or in $\mathsf F_\infty^\R$ when $p$ is even, is small-time approximately realizable. By \Cref{th:odd-saturation-density}, this space is dense in $H^s(M;\R)$. Hence, for any $y_0,y_1\in H^s(M;\R)$ and any $\varepsilon>0$, we choose $\xi$ in this space such that
\[\|y_1-y_0-\xi\|_{H^s}<\frac{\varepsilon}{2}.\]
Realizing the translation $\tau_\xi$ with error $\varepsilon/2$ gives a final
state within $\varepsilon$ of $y_1$, in arbitrarily small time.
\qed

\section{Controllability under Neumann boundary conditions}
\label{sec:manifold_boundary}
We now explain how the previous saturation argument can be adapted to compact manifolds with boundary under homogeneous Neumann boundary conditions. The literature on finite-dimensional saturation methods for PDEs on manifolds with boundary is considerably scarcer than in the boundaryless setting. For the Navier--Stokes system, explicit saturating families have been constructed for some particular geometries and boundary conditions; see, for instance, \cite{Rodrigues2006,PhanRodrigues2019,Phan2021} for rectangles and cylinders under Lions boundary conditions. For the viscous Burgers equation on bounded intervals with Dirichlet or Neumann boundary conditions, a recent construction based on bilinear controls was obtained in \cite{DucaTakahashi2025}. We also mention the paper on the two-dimensional Navier-Stokes equation \cite{Rodrigues2007}, where sufficient conditions for saturation are derived on compact connected Riemannian surfaces, with or without smooth boundary, although explicit saturating sets are only given for the round sphere $\mathbb{S}^2$, the flat torus $\mathbb{T}^2$ and rectangles.\medskip
\addtocontents{toc}{\protect\setcounter{tocdepth}{1}}

\subsection{Main result}
We use the Neumann operator $A_N$, the positive operator $\Lambda$ and the scale $H_N^s$ introduced in \Cref{sec:admissible_control_boundary}. By spectral calculus, $\Lambda$ commutes with the Schr\"odinger group $e^{-itA_N}$.  For an integer $p\geq2$, we consider
\begin{equation}
\label{eq:NLS-boundary-Neumann}
\begin{cases}
i\partial_t u=(-\Delta_g+W)u+\kappa |u|^{p-1}u+\eta(t,x),
& (t,x)\in[0,T]\times M,\\
\partial_\nu u=0,
& (t,x)\in[0,T]\times \partial M,\\
u(0)=u_0\in H_N^s(M;\C).
\end{cases}
\end{equation}
The linear part is handled as in the boundaryless case. Indeed, $A_N$ is
self-adjoint on $L^2(M)$ and the group $e^{-itA_N}$ is strongly continuous on
$H_N^s$. We work in the range
\begin{equation}
\label{eq:admissible-s}
\frac d2<s<
\begin{cases}
\frac72-\frac{2}{p}, & p \text{ odd},\\
\min\left\{p-1,\frac72-\frac{2}{p}\right\},
& p \text{ even},
\end{cases}
\qquad s\neq\frac32.
\end{equation}
For instance, in the cubic case, the admissible range is nonempty for $2\leq d\leq5$. The characterization of $H_N^s$ was recalled in \Cref{sec:admissible_control_boundary}. Since $s>d/2$, the space $H^s(M)$ is a Banach algebra. Recall that, when $p$ is even, the restriction $s<p-1$ ensures that the nonlinearity is locally Lipschitz on $H^s(M;\C)$, as explained in \Cref{sec:additive-NLS}.
Hence the map $u\mapsto |u|^{p-1}u$ is locally Lipschitz from $H^s(M)$ into itself. Moreover, if
$\partial_\nu u=0$, then
\[\partial_\nu(|u|^{p-1}u)=(p-1)|u|^{p-3}\operatorname{Re}(\partial_\nu u\overline u)u+|u|^{p-1}\partial_\nu u=0\quad\text{on }\partial M.\]
Therefore, for every $s$ satisfying \eqref{eq:admissible-s},
\[u\in H_N^s\quad\Longrightarrow\quad|u|^{p-1}u\in H_N^s.\]
Assuming \eqref{eq:admissible-s}, the nonlinear map is consequently locally Lipschitz from $H_N^s$ into itself. Thus, for $u_0\in H_N^s$ and $\eta\in L^1(0,T;H_N^s)$, the local well-posedness theory follows by the same fixed-point argument as in the closed-manifold case \Cref{prop:NLS-LWP}.\medskip

The approximation results of \Cref{sec:admissible_control_boundary} give the additional regularity needed for the dynamical limit.
\begin{corollary}[Regular profiles for the Neumann dynamics]
\label{th:Neumann-saturating-family}
Let $s$ satisfy \eqref{eq:admissible-s}. Then there exist a finite-dimensional real space
$E^\R\subset C^\infty(\overline M;\R)$ containing $1$ and $\gamma>\frac{2}{p}$ such that, setting
$E:=E^\R\otimes_\R\C$, its power saturation $\mathsf E_\infty$ of degree $p$ satisfies
\[\mathsf E_\infty\subset H_N^{s+\gamma}(M;\C),
\qquad \overline{\mathsf E_\infty}^{H_N^s(M;\C)}=H_N^s(M;\C).\]
\end{corollary}
\begin{proof}
Choose $2/p<\gamma<7/2-s$ with $s+\gamma\neq3/2$. This is possible by \eqref{eq:admissible-s}. The family in \Cref{th:Neumann-approximation} has power saturation $\mathcal A_N\subset H_N^{s+\gamma}$, and \Cref{lem:Neumann-sobolev-density} gives density in $H_N^s$.
\end{proof}
The corresponding controllability result is the following.
\begin{theorem}
\label{th:bdry-manifold}
Assume $\kappa\neq0$. Let $W\in C^\infty(\overline M;\R)$, and let $s>\tfrac d2$ satisfy \eqref{eq:admissible-s}, and let $E$ be a control space provided by \Cref{th:Neumann-saturating-family}. Then \eqref{eq:NLS-boundary-Neumann} is small-time approximately controllable in $H_N^s(M;\C)$ with controls $\eta\in L^\infty(0,T;E)$.
\end{theorem}

The proof follows the same saturation strategy as in the closed-manifold case \Cref{thm:intro-NLS-STAC}. There are two main differences:\par\medskip

\noindent The first one is to construct a finite-dimensional family of profiles which is both Neumann-compatible and saturating in $H_N^s$. This is provided by \Cref{sec:admissible_control_boundary}. The main additional restriction appears in the realization of translations. If $\xi$ is a control profile, then the translated state $u+\xi$ must still belong
to the phase space in which the equation is solved. Thus, under Neumann boundary conditions, one has to impose $\xi\in H_N^s$. Consequently, the saturation procedure must be formulated with profiles belonging to $H_N^s$, and the corresponding density property has to be proved inside $H_N^s$. This replaces, in the boundary case, the use of arbitrary smooth embedding coordinates on manifolds with boundary, since such coordinates can not satisfy the Neumann compatibility condition on $\partial M$. \par\medskip

\noindent The second difference concerns the large-amplitude limit, analogous to the one used in the proof of \Cref{lem:NLS-asymptotic-limit}, in particular estimate \eqref{eq:U-I_delta}. The Neumann Schr\"odinger group
$U_N(t):=e^{-itA_N}$ must satisfy
\[\sup_{0\leq t\leq\delta}\|(U_N(t)-I)\delta^{-1/p}\alpha\varphi\|_{H_N^s}\longrightarrow 0\quad \text{as }\delta\to0^+.\]
The profiles constructed in \Cref{sec:admissible_control_boundary} belong to $H_N^r$ for every $r<7/2$. However, working only at the phase-space regularity $H_N^s$ does not provide a sufficient rate of convergence for the previous limit. The price to pay is therefore to restrict slightly the regularity range of the phase space. More precisely, for $\theta\in[0,1]$ it is easy to show that 
\[\sup_{0\leq t\leq\delta}\|(U_N(t)-I)\delta^{-1/p}\alpha\varphi\|_{H_N^s}\leq C|\alpha|\delta^{\theta-1/p}\|\varphi\|_{H_N^{s+2\theta}}.\]
We thus need $\theta>1/p$. Since the profiles constructed in \Cref{sec:admissible_control_boundary} belong to $H_N^r$ for every $r<7/2$, we may choose
\[\frac{1}{p}<\theta<\min\left\{1,\frac{7/2-s}{2}\right\},\]
which is possible precisely when $s<\frac72-\frac{2}{p}$, as stated in \eqref{eq:admissible-s}.\medskip

\begin{remark}[Simpler Neumann control families]
When $0\leq s<3/2$, one has $H_N^s(M)=H^s(M)$, and the power saturation of the embedding coordinates $1,V_1,\ldots,V_m$ is dense in $H_N^s(M)$. However, for arbitrary smooth profiles, the preceding estimate gives the large-amplitude limit under the stronger restriction $s<\tfrac32-\tfrac{2}{p}$, since such profiles belong to $H_N^r$ for every $0\leq r<3/2$. For $s>3/2$, however, every translation direction must satisfy the Neumann compatibility condition $\partial_\nu\xi=0$ on $\partial M$. An arbitrary embedding does not have this property. Moreover, one cannot require all the coordinates of an ordinary smooth embedding to satisfy the Neumann condition: if $\partial_\nu V_j=0$ on $\partial M$ for every $j$, then
\[ d_x\mathcal V(\nu_x)=0,\qquad x\in\partial M,\]
and hence $\mathcal V$ is not an immersion at the boundary. For the same reason, projecting an arbitrary embedding onto finitely many Neumann eigenfunctions does not immediately solve the problem. Every such spectral projection has vanishing normal derivative, so it cannot converge in $C^1(\overline M)$ to an embedding having a nonzero normal derivative. This is why boundary and interior profiles are introduced separately.
\end{remark}
\color{black}
\section{Controllability for less regular data}
\label{sec:less_regular}
The condition $s>d/2$ used above is mainly imposed in order to work in a Sobolev-algebra and to keep the nonlinear estimates completely elementary. This assumption is not expected to be optimal. In fact, using Strichartz estimates one can lower the regularity threshold and treat rougher initial data, at the price of working in mixed space-time norms and modifying the large-amplitude estimates accordingly. On some specific geometries, such as flat tori, sharper dispersive or multilinear estimates may lead to even better ranges of admissible Sobolev exponents. We do not pursue these refinements here, since our purpose is to present the saturation mechanism in a robust functional framework. In this section, we restrict our attention to the cubic Schr\"odinger and Klein--Gordon equations. 

\subsection{Schr\"odinger equation}
\label{ssec:Stri-Sch}
Let $(M,g)$ be a compact Riemannian manifold without boundary of dimension $d$, let $W\in C^\infty(M;\mathbb R)$, and let $\kappa\in\mathbb R\setminus\{0\}$. We consider the controlled cubic Schr\"odinger equation
\begin{equation}
\label{eq:NLS-2}
\begin{cases}
i\partial_t u=(-\Delta_g+W)u+\kappa |u|^2u+\eta(t,x),& (t,x)\in[0,T]\times M,\\
u(0)=u_0\in H^s(M;\C).
\end{cases}\end{equation}
Let $s\geq 0$ and set $\Lambda=(1-\Delta_g)^{\frac 12}$. We define the space $H^s (M;\C)$ by 
\[H^s (M;\C):=\{u \in L^2(M;\C) : \|u\|_{H^s}=\|\Lambda^s u\|_{L^2}<\infty\}.\]
The regularity estimates and the well-posedness argument in this section are mainly based on the works \cite{burq2004strichartz,dehman2006stabilization}. The following Strichartz estimate is presented in \cite[Theorem 1]{burq2004strichartz}. 
\begin{theorem}
 Let $2\leq b, a < \infty$ satisfy the Schr\"odinger admissibility condition
\begin{equation}
\label{eq:p-q}
\frac{2}{b}+\frac{d}{a}=\frac{d}{2}, \qquad (b,a)\neq (2,\infty).
\end{equation}
Then, for every finite time interval $I$, there exists $C_I>0$ such that
\begin{equation}
\label{eq:Strichartz}
\left\|e^{it\Delta_g}v_0\right\|_{L^b(I;L^a(M))}
\leq C_I\|v_0\|_{H^{1/b}(M)}
\end{equation}
for every $v_0\in H^{1/b}(M)$.
\end{theorem}

Finally, we also define the space $Y_T^s=L^{\infty}(0,T;H^{s}(M))\cap L^{b}(0,T;W^{s-\frac{1}{b},a}(M))$, endowed with the norm
\[\|u\|_{Y_T^s}=\|u\|_{L^{\infty}(0,T;H^{s}(M))}+\|\Lambda^{s-\frac{1}{b}}u\|_{L^{b}(0,T;L^{a}(M))}.\]

\noindent As we are interested in small-time controllability, we fix $0<T<1$, so that the Strichartz constants can be chosen uniformly by taking $I=(0,1)$. Similarly, since we work in short-time intervals, the sign of the nonlinearity is not relevant.

\begin{proposition}
\label{prop:NLS-LWP-2}
Let $b,a$ satisfy \eqref{eq:p-q}, $b>2$, $0<T<1$, $W\in C^\infty(M;\R)$, $\kappa\in\R$. Let $s>d/2 - 1/b$, $u_0\in H^s(M;\C)$, and $\eta\in L^1(0,T;H^s(M;\C))$.
Then there exists $T_*\in(0,T]$ such that \eqref{eq:NLS-2} admits a unique mild solution $u\in C([0,T_*];H^s(M))\cap Y_{T_*}^s$. Moreover, the solution depends locally Lipschitz continuously on the initial datum and on the source term, on a common interval of existence.
\end{proposition}
\begin{proof}
As already mentioned, the proof is mainly based on \cite{burq2004strichartz,dehman2006stabilization}. We explain here the main ideas and differences with respect to  \Cref{prop:NLS-LWP}. We recall that $e^{it\Delta_g }$ defines a strongly continuous Schrödinger group. Using Duhamel's formula, define $\Gamma(u)$ by
\begin{equation}
\label{eq:NLS-Duhamel}
\Gamma(u)=e^{it\Delta_g }u_0-i\int_0^t e^{i(t-\sigma)\Delta_g }\bigl(Wu+\kappa|u|^2 u+\eta\bigr)d\sigma.
\end{equation}
We estimate the $Y_T^s$ norm. Using \eqref{eq:NLS-Duhamel} and the unitarity of the Schrödinger group, we obtain 
\begin{equation}
\label{eq:Linfty-Hs}
\|\Gamma(u)\|_{L^{\infty}(0,T;H^{s}(M))}\leq \|u_0\|_{H^{s}(M)}+\|Wu+\kappa|u|^2 u +\eta\|_{L^{1}(0,T;H^{s}(M))}.
\end{equation}
Applying $\Lambda^{s-\frac{1}{b}}$ to \eqref{eq:NLS-Duhamel}, we get
\[\begin{aligned}
 \|\Lambda^{s-\frac{1}{b}}\Gamma(u)\|_{L^{b}(0,T;L^{a}(M))}\leq& \|\Lambda^{s-\frac{1}{b}}e^{it\Delta_g }u_0\|_{L^{b}(0,T;L^{a}(M))}\\
 &+\left\|\Lambda^{s-\frac{1}{b}}\int_0^t e^{i(t-\sigma)\Delta_g }(Wu+\kappa|u|^2u+\eta)d\sigma \right\|_{L^{b}(0,T;L^{a}(M))}.   
\end{aligned}\]
The left-hand side corresponds to the $L^{b}(0,T;W^{s-\frac{1}{b},a}(M))-$norm of $\Gamma(u)$. For the first term on the right-hand side, since $\Lambda^{s-\frac{1}{b}}$ commutes with $e^{it\Delta_g }$ and using Strichartz estimate \eqref{eq:Strichartz} we deduce
\[\begin{aligned}
\|\Lambda^{s-\frac{1}{b}}e^{it\Delta_g }u_0\|_{L^b (0,T;L^{a}(M))}=\|e^{it\Delta_g }\Lambda^{s-\frac{1}{b}}u_0\|_{L^{b}(0,T;L^{a}(M))}&\leq C\|\Lambda^{s-\frac{1}{b}}u_0\|_{H^{\frac{1}{b}}(M)}=C\|u_0\|_{H^{s}(M)}.
\end{aligned}\]
For the second term, Minkowski's inequality and the Strichartz estimate yield
\[\begin{aligned}
&\left\|\Lambda^{s-\frac{1}{b}}\int_0^t e^{i(t-\sigma)\Delta_g }(Wu+\kappa|u|^2u+\eta)d\sigma \right\|_{L^{b}(0,T;L^{a}(M))}\\ 
&\leq \int_0^T \left\|e^{i(t-\sigma)\Delta_g  }\Lambda^{s-\frac{1}{b}}(Wu+\kappa|u|^2u+\eta)\right\|_{L^{b}(0,T;L^{a}(M))} d\sigma \\
&\leq C\|Wu+\kappa|u|^2u+\eta\|_{L^1 (0,T;H^{s}(M))}.
\end{aligned}\]
Gathering the previous estimates with \eqref{eq:Linfty-Hs}, we obtain
\begin{equation}
\label{eq:Y_Ts}
\|\Gamma(u)\|_{Y_T^s}\leq C\left(\|u_0\|_{H^{s}(M)}+\|Wu+\kappa|u|^2 u+\eta\|_{L^{1}(0,T;H^{s}(M))}\right).
\end{equation}
Now, we need to estimate the $L^{1}(0,T;H^{s}(M))$ term on the right-hand side of the above estimate. First, as $W \in C^\infty (M)$,
\[\begin{aligned}
\|Wu\|_{L^{1}(0,T;H^{s}(M))}&\leq C_W T\|u\|_{L^{\infty}(0,T;H^{s}(M))}  \leq C_W T\|u\|_{Y_T^s}.
\end{aligned}\]
Second, using Moser's estimate and H\"older's inequality, we obtain
\[\begin{aligned}
\||u|^2u\|_{L^{1}(0,T;H^{s}(M))}&\leq C\int_0^T \|u\|_{L^{\infty}}^2 \|u\|_{H^s(M)}d\sigma \\
& \leq C T^{1-\frac{2}{b}}\|u\|_{L^{b}(0,T;L^{\infty}(M))}^2 \|u\|_{L^{\infty}(0,T;H^{s}(M))}.
\end{aligned}\]
Moreover, $W^{s-\frac{1}{b},a}(M)\hookrightarrow L^\infty(M)$ provided $s-\frac{1}{b}>\frac{d}{a}$. Combining this condition with \eqref{eq:p-q}, we obtain $s>\frac{d}{2}-\frac{1}{b}$. Thus, 
\[\|u\|_{L^{b}(0,T;L^{\infty}(M))}\leq C \|u\|_{L^{b}(0,T;W^{s-\frac{1}{b},a}(M))}.\]
Finally,
\[\|\Gamma(u)\|_{Y_T^s}\leq C\left(\|u_0\|_{H^{s}(M)}+\|\eta\|_{L^{1}(0,T;H^{s}(M))}+T\|u\|_{Y_T^s}+T^{1-\frac{2}{b}}\|u\|_{Y_T^s}^3\right).\]
Thus $\Gamma:Y_T^s \rightarrow Y_T^s$ is well-defined. Set $R=2C(\|u_0\|_{H^{s}(M)}+\|\eta\|_{L^{1}(0,1;H^{s}(M))})$. If $\|u\|_{Y_T^s}\leq R$ then
\begin{equation}\label{eq:uniformY_T}
\begin{aligned}
\|\Gamma(u)\|_{Y_T^s}&\leq \tfrac{R}{2}+CTR+CT^{1-\frac{2}{b}}R^3  \\
& \leq R, \qquad \text{ for } T \text{ small}. 
\end{aligned}\end{equation}
Similarly, for $u, v \in Y_T^s$ such that $\|u\|_{Y_T^s},\|v\|_{Y_T^s}\leq R$, we obtain
\[\|\Gamma(u)-\Gamma(v)\|_{Y_T^s}\leq C(T+T^{1-\frac{2}{b}}R^2)\|u-v\|_{Y_T^s}.\]
Thus, for $T>0$ sufficiently small, $\Gamma$ is a contraction. Therefore, there exists a unique fixed point $u\in Y_T^s$, which is a solution of \eqref{eq:NLS-2}. It remains to recover the continuity in $H^s(M)$. From \eqref{eq:NLS-Duhamel}, using an estimate analogous to \eqref{eq:Linfty-Hs} but in the norm $C([0,T];H^s(M))$, we get
\[\|u\|_{C([0,T];H^s(M))}\leq \|u_0\|_{H^{s}(M)}+C\|Wu+\kappa|u|^2 u +\eta\|_{L^{1}(0,T;H^{s}(M))}.\]
The right-hand side is finite, since it is controlled by the $Y_T^s$ norm and by $\|\eta\|_{L^1(0,T;H^s(M))}$. Hence $ u\in C([0,T];H^s(M))\cap Y_T^s.$ Finally, for sufficiently close initial data and source terms, the fixed-point argument gives solutions on a common interval $[0,T]$. Let $u$ and $v$ be two mild solutions with respective data $(u_0,\eta)$ and $(v_0,\zeta)$. Applying the previous difference estimates on a sufficiently short time interval, we obtain
\[\|u-v\|_{Y_T^s}\le C\left(\|u_0-v_0\|_{H^s}+\|\eta-\zeta\|_{L^1(0,T;H^s)}\right),\]
where $C$ depends on a bound for their $Y_T^s$ norms. This proves uniqueness and local Lipschitz dependence in $Y_T^s$ and in $C([0,T];H^s)$.
\end{proof}

In this section, $\mathsf E_n$ and $\mathsf E_\infty$ denote the cubic power saturation spaces, defined by
\eqref{eq:odd-saturation}. The following theorem is the main
result of this section.
\begin{theorem}
\label{thm:NLS-STAC2}
Let $(b,a)$ satisfy \eqref{eq:p-q}, $b>2$. Let $W\in C^\infty(M;\R)$, $\kappa\in\R\setminus\{0\}$ and  $s>d/2-1/b$. Assume that $\mathcal V:M\to\R^m$, $\mathcal V(x):=(V_1(x),\dots,V_m(x))$, is a smooth embedding, and let $E=\operatorname{span}_{\C}\{1,V_1,\dots,V_m\}$. Then \eqref{eq:NLS-2} is small-time approximately controllable in $H^s(M;\C)$ by controls
$\eta\in L^\infty(0,T;E).$
\end{theorem}
We only prove the direct translation property and the large-amplitude asymptotic limit in the Strichartz framework. The remaining steps, namely the cubic kick, the stability under sums, the saturation criterion and the final controllability argument, are unchanged and follow the proof of \Cref{sec:additive-NLS}.

\begin{lemma}[Direct translations]
\label{lem:NLS-direct-translations2}
For every $\xi\in E$, the translation $\tau_\xi$ is small-time approximately realizable.
\end{lemma}
\begin{proof}
Fix $u_0\in H^s(M;\C)$ and $\xi\in E$. For $\delta>0$, set $\eta_\delta(t):=\frac{i}{\delta}\xi$, $0\le t\le\delta$. Clearly, $\eta_\delta\in L^\infty(0,\delta;E)$. Moreover, by \eqref{eq:eta_L1Hs}, $\|\eta_\delta\|_{L^1(0,\delta;H^s)}$ is uniformly bounded with respect to $\delta$. Let $u$ be the solution of \eqref{eq:NLS-2} with initial datum $u_0$ and control $\eta_\delta$. By Duhamel's formula,
\[u(\delta)=e^{i\delta \Delta_g }u_0
+\frac1\delta\int_0^\delta e^{i (\delta-\sigma)\Delta_g}\xi d\sigma-i\int_0^\delta
e^{i(\delta-\sigma)\Delta_g }\left(Wu(\sigma)+\kappa|u(\sigma)|^2u(\sigma)\right)d\sigma.\]
By strong continuity of $e^{it\Delta_g}$ in $H^s$,
\[e^{i\delta\Delta_g }u_0\to u_0, \qquad
\frac1\delta\int_0^\delta e^{i(\delta-\sigma)\Delta_g }\xi d\sigma\to \xi \qquad \text{ in } H^s  \ \text{ as } \delta\to0^+.\]
Moreover, by the uniform local bound obtained in \eqref{eq:uniformY_T}, $\|u\|_{Y_\delta^s}\le C(u_0,\xi)$ for all sufficiently small $\delta>0$. For the last term, using the same estimates as those obtained in the well-posedness argument, we have
\[\begin{aligned}
 \left\|\int_0^\delta
e^{i(\delta-\sigma)\Delta_g }\left(Wu(\sigma)+\kappa|u(\sigma)|^2u(\sigma)\right)d\sigma\right\|_{H^s}&\le \|Wu+\kappa|u|^2u\|_{L^{1}(0,\delta;H^{s}(M))}\\
&\le C\left(\delta \|u\|_{Y_\delta^s}+\delta^{1-\frac{2}{b}}\|u\|_{Y_\delta^s}^3\right).   
\end{aligned}\]
Using the uniform bound in $Y_\delta^s$, the right-hand side tends to zero. Hence $u(\delta) \to u_0+\xi$ in $H^s(M)$.
\end{proof}

Whenever the solution of \eqref{eq:NLS-2} with initial datum $u_0$ and control $\eta$ is defined, we denote by $\mathcal R_t(u_0,\eta)$ its value at time $t$. The uncontrolled flow is given by $\mathcal R_t(u_0,0)$.

\begin{lemma}[Asymptotic limit]
\label{lem:NLS-asymptotic-limit2}
Let $\varphi\in C^\infty(M;\R)$ and $\alpha\in\C$.
Then, for every $\rho>0$, there exists $\delta_0>0$ such that, for every $\delta\in(0,\delta_0)$ and every $u_0\in H^s(M;\C)$ with $\|u_0\|_{H^s}\leq \rho$,
\[t\mapsto\mathcal R_t(u_0+\delta^{-1/3}\alpha\varphi,0) \quad \text{is well defined in } C([0,\delta];H^s(M;\C)).\]
 Moreover,
\[\mathcal R_\delta(u_0+\delta^{-1/3}\alpha\varphi,0)-\delta^{-1/3}\alpha\varphi \longrightarrow u_0-i\kappa\alpha|\alpha|^2\varphi^3
\qquad\text{in }H^s(M;\C) \qquad \text{ as } \delta\to0^+.\]
\end{lemma}

\begin{proof}
As in the proof of \Cref{lem:NLS-asymptotic-limit}, we first construct a solution of the renormalized equation. Define on $Y_\delta^s$
\[\widetilde{\Gamma}_\delta(w):=\Gamma_{u_0+\delta^{-1/3}\alpha\varphi,0}\left(w+\delta^{-1/3}\alpha\varphi\right)-\delta^{-1/3}\alpha\varphi,\]
where $\Gamma_{u_0,\eta}$ is the Duhamel map associated with \eqref{eq:NLS-2}. By Duhamel's formula,
\begin{equation}
\label{eq:w-delta-duhamel}
\begin{aligned}
\widetilde{\Gamma}_\delta(w)(t)={}&e^{it\Delta_g}u_0+\bigl(e^{it\Delta_g}-I\bigr)\delta^{-1/3}\alpha\varphi-i\int_0^t e^{i(t-\sigma)\Delta_g}W(\delta^{-1/3}\alpha\varphi+w(\sigma))d\sigma\\
&-i\kappa\int_0^t e^{i(t-\sigma)\Delta_g}\left(
|\delta^{-1/3}\alpha\varphi+w(\sigma)|^2
(\delta^{-1/3}\alpha\varphi+w(\sigma))
\right)d\sigma.
\end{aligned}
\end{equation}
We apply the same strategy as in \Cref{lem:NLS-asymptotic-limit}, but in the Strichartz norm. First, by the Strichartz estimates,
\begin{equation}
\label{eq:w-delta-Ys}
\begin{aligned}
\|\widetilde{\Gamma}_\delta(w)\|_{Y_\delta^s}\le C\Big(
&\|u_0\|_{H^s}+\|\bigl(e^{it\Delta_g}-I\bigr)\delta^{-1/3}\alpha\varphi\|_{Y_\delta^s}+\|W(\delta^{-1/3}\alpha\varphi+w)\|_{L^1(0,\delta;H^s)}\\
&+\||\delta^{-1/3}\alpha\varphi+w|^2
(\delta^{-1/3}\alpha\varphi+w)\|_{L^1(0,\delta;H^s)}
\Big).
\end{aligned}
\end{equation}
We now collect estimates on the right-hand side of the above inequality. We have
\[\sup_{0\le t\le \delta}\|\bigl(e^{it\Delta_g }-I\bigr)\delta^{-1/3}\alpha\varphi\|_{H^s}\le C\delta^\frac{2}{3},\]
\[\begin{aligned}
\|\bigl(e^{it\Delta_g }-I\bigr)\delta^{-1/3}\alpha\varphi\|_{L^b(0,\delta,W^{s-\frac{1}{b},a})}&\le C\delta^{-1/3}\left(\int_0^\delta t^b dt\right)^{\frac{1}{b}}
&\lesssim \delta^{\frac{2}{3}+\frac{1}{b}},
\end{aligned}\]
and
\[\begin{aligned}
\|W(\delta^{-1/3}\alpha\varphi+w)\|_{L^1 (0,\delta;H^s)}&\le C\int_0^\delta \|\delta^{-1/3}\alpha\varphi+w(t)\|_{H^s}dt\\
&\le C\left(\delta \delta^{\frac{-1}{3}}+\delta \|w\|_{L^\infty (0,\delta;H^s)}\right)\\
&\le C\delta^{\frac{2}{3}}+C\delta\|w\|_{Y_\delta^s}.
\end{aligned}\]
Observe that 
\[\begin{aligned}
\left|\delta^{-1/3}\alpha\varphi+w\right|^2\left(\delta^{-1/3}\alpha\varphi+w\right)=&\delta^{-1}\alpha|\alpha|^2\varphi^3
+\delta^{-2/3}\varphi^2\left(2|\alpha|^2w+\alpha^2\overline{w}\right)\\
&+\delta^{-1/3}\varphi
\left(2\alpha|w|^2+\overline\alpha(w)^2\right)
+|w|^2w.
\end{aligned}\]
By Moser's and Hölder's inequalities,
\[\begin{aligned}
\left\||\delta^{-1/3}\alpha\varphi+w|^2(\delta^{-1/3}\alpha\varphi+w) \right\|_{L^1 (0,\delta;H^s)}
&\le C\left(1+\delta^{-2/3}\int_0^\delta\|w\|_{H^s}+ \delta^{-1/3}\int_0^\delta\|w\|_{L^\infty}\|w\|_{H^s}\right.\\
& \quad\left.+\int_0^\delta\|w\|^2_{L^\infty}\|w\|_{H^s}\right)\\
&\le C\left(1+\delta^{1/3}\|w\|_{Y_\delta^s}+ \delta^{2/3-1/b}\|w\|^2_{Y_\delta^s}+ \delta^{1-2/b}\|w\|^3_{Y_\delta^s}\right).\end{aligned}\]
Let $X_\delta=\|w\|_{Y_\delta^s}$. The previous estimates show that 
\[\|\widetilde{\Gamma}_\delta(w)\|_{Y_\delta^s}\leq C\left(1+\delta^{2/3}+\delta^{2/3+1/b}+(\delta+\delta^{1/3})X_\delta+\delta^{2/3-1/b}X_\delta^2 +\delta^{1-2/b}X_\delta^3\right)\]
Since $b>2$, all powers of $\delta$ appearing in front of $X_\delta$, $X_\delta^2$ and $X_\delta^3$ are positive.  Hence, choosing $R$ sufficiently large and then $\delta>0$ sufficiently small, the map leaves the closed ball of radius $R$ in $Y_\delta^s$ invariant. Moreover, the same estimates applied to a difference give, for $\|w_1\|_{Y_\delta^s},\|w_2\|_{Y_\delta^s}\le R$,
\[\|\widetilde{\Gamma}_\delta(w_1)-\widetilde{\Gamma}_\delta(w_2)\|_{Y_\delta^s}\le C\left(\delta+\delta^{1/3}+\delta^{2/3-1/b}R+\delta^{1-2/b}R^2\right)\|w_1-w_2\|_{Y_\delta^s}.\]
Taking $\delta>0$ smaller if necessary, the map is a contraction on that ball. Hence it admits a unique fixed point $w^\delta$ in that ball, and $\|w^\delta\|_{Y_\delta^s}\le R$ for all sufficiently small $\delta>0$. The integral formulation also gives $w^\delta\in C([0,\delta];H^s)$, since the potential and nonlinear terms belong to $L^1(0,\delta;H^s)$. Set
\[v^\delta(t):=w^\delta(t)+\delta^{-1/3}\alpha\varphi.\]
By construction, $v^\delta$ satisfies the mild formulation with initial datum $u_0+\delta^{-1/3}\alpha\varphi$ and no control. By uniqueness,
\[v^\delta(t)=\mathcal R_t(u_0+\delta^{-1/3}\alpha\varphi,0),\qquad 0\le t\le\delta.\]
Clearly, $e^{i\delta\Delta_g}u_0\to u_0$ in $H^s$. The estimates above also give
\[\|\bigl(e^{i\delta\Delta_g}-I\bigr)\delta^{-1/3}\alpha\varphi\|_{H^s}\to0\]
and
\[\left\|\int_0^\delta e^{i(\delta-\sigma)\Delta_g}W(\delta^{-1/3}\alpha\varphi+w^\delta(\sigma))d\sigma\right\|_{H^s}\le C(\delta^{2/3}+\delta R)\to0.\]
It remains to identify the limit of the nonlinear term. We write
\[|\delta^{-1/3}\alpha\varphi+w^\delta|^2(\delta^{-1/3}\alpha\varphi+w^\delta)
=\delta^{-1}\alpha|\alpha|^2\varphi^3+\mathcal E_\delta .\]
Using the uniform bound of $w^\delta$ in $Y_\delta^s$ and the same estimates
as above, one obtains
\[\left\|\int_0^\delta e^{i(\delta-\sigma)\Delta_g}\mathcal E_\delta(\sigma)d\sigma
\right\|_{H^s}\to0.\]
On the other hand, since $\varphi^3\in C^\infty(M)\subset H^s(M)$ and
$e^{it\Delta_g}$ is strongly continuous in $H^s(M)$,
\[\delta^{-1}\int_0^\delta e^{i(\delta-\sigma)\Delta_g}\varphi^3d\sigma\to\varphi^3\quad\text{in }H^s(M).\]
Therefore, $w^\delta(\delta) \to u_0-i\kappa\alpha|\alpha|^2\varphi^3$ in $H^s(M;\C)$, which proves the claim.
\end{proof}
The analogues of \Cref{lem:NLS-odd-kick,lem:NLS-sums,lem:NLS-saturated-directions} and \Cref{prop:NLS-saturation-criterion} follow directly.
\begin{lemma}
\label{lem:NLS2}
The following assertions hold:
\begin{itemize}
\item Let $\varphi\in  C^\infty(M;\R)$ and $\alpha\in\C$. Assume that, for every $\delta>0$, the translations $\tau_{\pm\delta^{-1/3}\alpha\varphi}$ are small-time approximately realizable. Then the translation $\tau_{-i\kappa\alpha|\alpha|^2\varphi^3}$ is small-time approximately realizable.
\item If $\tau_{\xi_1}$ and $\tau_{\xi_2}$ are small-time approximately realizable, then $\tau_{\xi_1+\xi_2}$ is small-time approximately realizable.
\item Let $\xi\in \mathsf E_\infty$. Then the translation $\tau_\xi$ is small-time approximately realizable.
\end{itemize}
\end{lemma}

\textit{Proof of \Cref{thm:NLS-STAC2}:}
By \Cref{th:odd-saturation-density}, the power saturation $\mathsf E_\infty$ of $E$ is dense in $H^s(M;\C)$. The conclusion follows directly from \Cref{lem:NLS2}.
\qed
\subsection{Klein--Gordon equation}
\label{ssec:Stri-KG}
Let $(M,g)$ be a compact Riemannian manifold without boundary of dimension $d$, let $W\in C^\infty(M;\mathbb R)$, and let $\kappa\in\mathbb R\setminus\{0\}$. We consider the controlled cubic Klein--Gordon equation
\begin{equation}
\label{eq:KG-2}
\begin{cases}
\partial_t^2 y-\Delta_g y+Wy+\kappa y^3=\eta, &(t,x)\in[0,T]\times M,\\
y(0)=y_0, \quad \partial_t y(0)=y_1.
\end{cases}\end{equation}
Let $s\geq 0$ and set $\Lambda=(-\Delta_g+W)^{\frac 12}$. We assume, for simplicity, that $-\Delta_g+W$ is positive.  The linear Klein--Gordon group is denoted by
\[S(t)\binom{y_0}{y_1}:=
\binom{\cos(t\Lambda)y_0+\Lambda^{-1}\sin(t\Lambda)y_1}{-\Lambda\sin(t\Lambda)y_0+\cos(t\Lambda)y_1}.\]
A mild solution of \eqref{eq:KG-2} is a function $Y=(y,\partial_t y)^\top\in C([0,T_*];\mathcal H^s)$ such that
\[y(t)=\cos(t\Lambda)y_0+\Lambda^{-1}\sin(t\Lambda)y_1+\int_0^t \Lambda^{-1}\sin((t-\sigma)\Lambda)
\left(\eta(\sigma)-\kappa y(\sigma)^3\right)d\sigma,\]
and
\[\partial_t y(t)=-\Lambda\sin(t\Lambda)y_0+\cos(t\Lambda)y_1+\int_0^t \cos((t-\sigma)\Lambda)
\left(\eta(\sigma)-\kappa y(\sigma)^3\right)d\sigma.\]
We recall the following local Strichartz estimate for the wave equation on compact manifolds, due to Kapitanskii \cite{Kap89}. Let $2\le b,a<\infty$ be a non-endpoint wave-admissible pair, namely
\begin{equation}
\label{eq:wave-admissible}
\frac{2}{b}+\frac{d-1}{a}\le \frac{d-1}{2},
\end{equation}
and let
\begin{equation}
\label{eq:wave-compatibility}
\gamma=\frac d2-\frac1b-\frac d a.
\end{equation}
Then, for every $T>0$, the solution of
\[\begin{cases}
\partial_t^2 z-\Delta_g z+Wz=F,\\
z(0)=z_0,\quad \partial_t z(0)=z_1,
\end{cases}\]
satisfies
\begin{equation}
\label{eq:wave-Strichartz}
\|z\|_{L^b(0,T;L^a)} \le C_T\left(\|z_0\|_{H^{\gamma}}+\|z_1\|_{H^{\gamma-1}}+\|F\|_{L^1(0,T;H^{\gamma-1})}
\right).
\end{equation}
We shall also assume
\begin{equation}
\label{eq:KG-Strichartz-restriction}
b>2,\qquad a>2d.
\end{equation}
We define
\[\mathcal X_T^s:=C([0,T];H^{s}(M))\cap C^1([0,T];H^{s-1}(M))\cap L^b(0,T;L^a(M)),\]
with the natural norm
\[\left\|\binom{z}{\partial_t z}\right\|_{\mathcal X_T^s}:=\|z\|_{C([0,T];H^{s})}+\|\partial_tz\|_{C([0,T];H^{s-1})}+\|z\|_{L^b(0,T;L^a)}.\]
We identify each trajectory $z\in\mathcal X_T^s$ with its state vector $(z,\partial_tz)^\top$.
We first record the local well-posedness statement in this space.

\begin{proposition}[Strichartz local well-posedness for the Klein--Gordon equation]
\label{prop:KG-2-LWP}
Let $\gamma,b,a$ satisfy \eqref{eq:wave-admissible}, \eqref{eq:wave-compatibility} and \eqref{eq:KG-Strichartz-restriction}. Let $s>\gamma$ and $(y_0,y_1)^\top\in \mathcal{H}^s$, $\eta\in L^1(0,T;H^{s-1}(M;\R))$. Then there exists $T_*\in(0,T]$ such that the cubic Klein--Gordon equation admits a unique solution $y\in \mathcal X_{T_*}^s$. Moreover, for every $R>0$, there exist $T_R\in(0,T]$ and $C_R>0$ such that, if
\[\left\|\binom{y_0}{y_1}\right\|_{H^{s}\times H^{s-1}}+\|\eta\|_{L^1(0,T_R;H^{s-1})}\le R,\]
then the corresponding solution is defined on $[0,T_R]$ and satisfies $\|y\|_{\mathcal X_{T_R}^s}\le C_R$. The flow also depends continuously on the initial datum and on the source term
in the corresponding norms.
\end{proposition}
\begin{proof}
We argue by a fixed-point argument in $\mathcal X_T^s$. First, $S(t)$ is
strongly continuous on $\mathcal H^s$ and, for every $T>0$,
\[\sup_{0\le t\le T}\|S(t)\|_{\mathcal L(\mathcal H^s,\mathcal H^s)}<\infty.\]
For $Y=(y,\partial_t y)^\top$, define the Duhamel map by
\[\Gamma(Y)(t)=S(t)\binom{y_0}{y_1}+\int_0^t
S(t-\sigma)\binom{0}{\eta(\sigma)-\kappa y(\sigma)^3}d\sigma.\]
Equivalently, the first component of $\Gamma(Y)$ is
\[\Gamma_1(y)(t)=\cos(t\Lambda)y_0+\Lambda^{-1}\sin(t\Lambda)y_1+\int_0^t\Lambda^{-1}\sin((t-\sigma)\Lambda)\left(\eta(\sigma)-\kappa y(\sigma)^3\right)d\sigma .\]
By the energy estimate and the Strichartz estimate
\eqref{eq:wave-Strichartz}, there exists $C_T>0$ such that
\[\|\Gamma(Y)\|_{\mathcal X_T^s}
\le C_T\left(\left\|\binom{y_0}{y_1}\right\|_{\mathcal H^s}+\|\eta\|_{L^1(0,T;H^{s-1})}+\|y^3\|_{L^1(0,T;H^{s-1})}\right).\]
We now estimate the cubic term. For $s>1$, since $a>2d$ by \eqref{eq:KG-Strichartz-restriction}, the Sobolev embedding $H^{s}(M)\hookrightarrow W^{s-1,r}(M)$ holds for $\frac1r=\frac12-\frac2a$. Combining this embedding with the fractional Leibniz rule (see \cite[Chapter~4]{RunstSickel1996}) gives the estimate below. For $0<s\leq1$, the same estimate follows from Sobolev embeddings and H\"older's inequality, using duality when $s<1$ (see \cite[Appendix~A.1]{Taylor1991}):
\[\|fgh\|_{H^{s-1}}\le C\left(\|f\|_{H^{s}}\|g\|_{L^a}\|h\|_{L^a}
+\|g\|_{H^{s}}\|f\|_{L^a}\|h\|_{L^a}+\|h\|_{H^{s}}\|f\|_{L^a}\|g\|_{L^a}\right).\]
Thus $\|y^3\|_{H^{s-1}}\le C\|y\|_{H^{s}}\|y\|_{L^a}^2$. Therefore, by H\"older's inequality in time,
\[\|y^3\|_{L^1(0,T;H^{s-1})}\le CT^{1-\frac2b}\|y\|_{C([0,T];H^{s})}\|y\|_{L^b(0,T;L^a)}^2.\]
Hence
\begin{equation}
\label{eq:KG-Strichartz-cubic-bound}
\|y^3\|_{L^1(0,T;H^{s-1})}\le CT^{1-\frac2b}\|Y\|_{\mathcal X_T^s}^3.
\end{equation}
Similarly, for two functions $y,z\in \mathcal X_T^s$, we get
\[\|y^3-z^3\|_{H^{s-1}} \le C\left(\|y\|_{L^a}^2+\|z\|_{L^a}^2\right)\|y-z\|_{H^{s}}+C\left(\|y\|_{H^{s}}+\|z\|_{H^{s}}\right)\left(\|y\|_{L^a}+\|z\|_{L^a}\right)\|y-z\|_{L^a}.\]
Consequently, if $\|Y\|_{\mathcal X_T^s}+\|Z\|_{\mathcal X_T^s}\le R$, then
\begin{equation}
\label{eq:KG-Strichartz-cubic-Lipschitz}
\|y^3-z^3\|_{L^1(0,T;H^{s-1})}\le C_RT^{1-\frac2b}\|Y-Z\|_{\mathcal X_T^s}.
\end{equation}
Since $b>2$, the factor $T^{1-\frac2b}$ tends to zero as $T\to0$. Let $R>0$ and assume
\[\left\|\binom{y_0}{y_1}\right\|_{\mathcal H^s}+\|\eta\|_{L^1(0,T;H^{s-1})}\le R.\]
Choose $M>0$ large enough, depending only on $R$. By
\eqref{eq:KG-Strichartz-cubic-bound},
\[\|\Gamma(Y)\|_{\mathcal X_T^s}\le C_T R + C_T T^{1-\frac2b}M^3\]
for every $Y\in \mathcal X_T^s$ with $\|Y\|_{\mathcal X_T^s}\le M$. Taking $T>0$ small enough, depending only on $R$, $\Gamma$ maps the closed ball of radius $M$ in $\mathcal X_T^s$ into itself. Moreover, if $Y,Z$ belong to this ball, then by
\eqref{eq:KG-Strichartz-cubic-Lipschitz},
\[\|\Gamma(Y)-\Gamma(Z)\|_{\mathcal X_T^s}\le C_M T^{1-\frac2b}\|Y-Z\|_{\mathcal X_T^s}.\]
Taking $T>0$ smaller if necessary, $\Gamma$ is a contraction. Hence \eqref{eq:KG-2} admits a unique mild solution $Y=(y,\partial_t y)^\top\in\mathcal X_T^s$. The uniform local bound follows from the same construction, with an existence time depending only on $R$. Finally, let $Y=(y,\partial_t y)^\top$ and $Z=(z,\partial_t z)^\top$ be two solutions on $[0,\varkappa]$ associated with data $((y_0,y_1)^\top,\eta)$ and $((z_0,z_1)^\top,\zeta)$, and assume that both are bounded by $R$ in $\mathcal X_\varkappa^s$. Applying the energy and Strichartz estimates to the difference equation and using \eqref{eq:KG-Strichartz-cubic-Lipschitz} on small subintervals, then iterating finitely many times, gives
\[\|Y-Z\|_{\mathcal X_\varkappa^s}\le C_{R,\varkappa}\left(\left\|\binom{y_0}{y_1}-\binom{z_0}{z_1}\right\|_{\mathcal H^s}+\|\eta-\zeta\|_{L^1(0,\varkappa;H^{s-1})}\right).\]
This proves continuous dependence and completes the proof.
\end{proof}
The following theorem is the main result of this part.
\begin{theorem}
\label{thm:KG-STAC2}
Let $\gamma,b,a$ satisfy \eqref{eq:wave-admissible}, \eqref{eq:wave-compatibility} and \eqref{eq:KG-Strichartz-restriction}. Let $W\in C^\infty(M;\R)$, $\kappa\in\R\setminus\{0\}$ and  $s>d/2-1/b-d/a$. Assume that $\mathcal V:M\to\R^m$, $\mathcal V(x):=(V_1(x),\dots,V_m(x))$, is a smooth embedding, and let $E^\R=\operatorname{span}_{\R}\{1,V_1,\dots,V_m\}$. Then \eqref{eq:KG-2} is small-time approximately controllable in $\mathcal{H}^s$ by controls
$\eta\in L^\infty(0,T;E^\R).$
\end{theorem}
For $Y_0=(y_0,y_1)^\top\in\mathcal H^s$, we denote by $\mathcal R_t(Y_0,\eta)$ the value at time $t$ of the solution of \eqref{eq:KG-2} with initial datum $Y_0$ and control $\eta$. The uncontrolled Klein--Gordon flow is given by $\mathcal R_t(Y_0,0)$.
\begin{lemma}[Direct velocity translations]
\label{lem:KG-2-direct-velocity-translations}
For every $\xi\in E^\R$, the translation $\tau_{(0,\xi)^\top}$ is small-time
approximately realizable for \eqref{eq:KG-2}.
\end{lemma}
\begin{proof}
Fix $Y_0=(y_0,y_1)^\top\in\mathcal H^s$ and $\xi\in E^\R$. For $\delta>0$, set $\eta_\delta(t):=\frac1\delta\xi$, $0\le t\le\delta$. Then $\eta_\delta\in L^\infty(0,\delta;E^\R)$ and $\|\eta_\delta\|_{L^1(0,\delta;H^{s-1})}=\|\xi\|_{H^{s-1}}$. By the uniform local bound in \Cref{prop:KG-2-LWP}, the corresponding
solution satisfies
\[\left\|\binom{y}{\partial_t y}\right\|_{\mathcal X_\delta^s}
\le C(Y_0,\xi)\]
for all sufficiently small $\delta>0$. Using the mild formulation, the contributions of the control to the position and velocity components satisfy, respectively,
\[\frac1\delta\int_0^\delta\Lambda^{-1}\sin((\delta-\sigma)\Lambda)\xi d\sigma
\longrightarrow 0
\quad\text{in }H^{s}\qquad \text{ and } \qquad 
\frac1\delta\int_0^\delta\cos((\delta-\sigma)\Lambda)\xi d\sigma\longrightarrow \xi\quad\text{in }H^{s-1}.\]
The linear free part converges to $(y_0,y_1)^\top$ in $\mathcal H^s$. It remains to estimate the nonlinear contribution. By
\eqref{eq:KG-Strichartz-cubic-bound},
\[\|y^3\|_{L^1(0,\delta;H^{s-1})}\le C\delta^{1-\frac2b}\left\|\binom{y}{\partial_t y}\right\|_{\mathcal X_\delta^s}^3\le C(Y_0,\xi)\delta^{1-\frac2b}\longrightarrow0.\]
Therefore, the two nonlinear Duhamel  terms,
\[\int_0^\delta\Lambda^{-1}\sin((\delta-\sigma)\Lambda)y(\sigma)^3d\sigma\quad\text{and}\quad\int_0^\delta\cos((\delta-\sigma)\Lambda)y(\sigma)^3d\sigma,\]
converge to $0$ in $H^{s}$ and in $H^{s-1}$, respectively. Hence $\mathcal R_\delta(Y_0,\eta_\delta) \longrightarrow Y_0+(0,\xi)^\top$ in $\mathcal H^s$. Taking $\delta<\varkappa$ small enough gives the result.
\end{proof}
\begin{lemma}[Position translations] 
\label{lem:KG-2-position-translations} 
For every $\xi\in E^\R$, the translation $\tau_{(\xi,0)^\top}$ is small-time approximately realizable for \eqref{eq:KG-2}. 
\end{lemma} 
\begin{proof} 
Fix $Y_0=(y_0,y_1)^\top\in\mathcal H^s$, $\xi\in E^\R$, $\eps>0$ and $\varkappa>0$. Let $\delta>0$ be small. By \Cref{lem:KG-2-direct-velocity-translations}, the translations $\tau_{(0,\delta^{-1}\xi)^\top}$ and $\tau_{(0,-\delta^{-1}\xi)^\top}$ are small-time approximately realizable. We first prove that, for all sufficiently small $\delta>0$, the uncontrolled solution with initial datum $(y_0,y_1+\delta^{-1}\xi)^\top$ is well defined on $[0,\delta]$ and
\begin{equation} 
\label{eq:KG-2-position-limit} 
\mathcal R_\delta\left(\binom{y_0}{y_1+\delta^{-1}\xi},0\right)+\binom{0}{-\delta^{-1}\xi} \longrightarrow \binom{y_0+\xi}{y_1} \qquad\text{in }\mathcal H^s . \end{equation} 
As in the proof of \Cref{lem:KG-position-translations}, we first construct the position component. Define on $C([0,\delta];H^s)\cap L^b(0,\delta;L^a)$
\[\Gamma_\delta(y)(t):=\cos(t\Lambda)y_0+\Lambda^{-1}\sin(t\Lambda)y_1+\delta^{-1}\Lambda^{-1}\sin(t\Lambda)\xi-\kappa\int_0^t\Lambda^{-1}\sin((t-\sigma)\Lambda)y(\sigma)^3d\sigma.\] 
Since $\xi\in E^\R\subset C^\infty(M)$, $ \delta^{-1}\Lambda^{-1}\sin(\delta\Lambda)\xi \longrightarrow \xi$ in $H^{s}$. Moreover,
\[\cos(\delta\Lambda)y_0\to y_0 \quad\text{in }H^{s}, \qquad \Lambda^{-1}\sin(\delta\Lambda)y_1\to0 \quad\text{in }H^{s}.\] 
The only point which is different from the case $s>d/2$ is the estimate of the nonlinear term. We cannot use a uniform bound for the full state in $\mathcal H^s$, because the initial velocity contains the term $\delta^{-1}\xi$. However, the position component remains uniformly bounded in the Strichartz space. Indeed, the large linear position term satisfies, for $0\le t\le\delta$, 
\[\|\delta^{-1}\Lambda^{-1}\sin(t\Lambda)\xi\|_{H^{s}} \le C\delta^{-1}t\|\xi\|_{H^{s}} \le C\|\xi\|_{H^{s}}, \] 
and also 
\[\|\delta^{-1}\Lambda^{-1}\sin(t\Lambda)\xi\|_{L^b(0,\delta;L^a)} \le C\delta^{-1}\left(\int_0^\delta t^bdt\right)^{1/b}
\le C_\xi\delta^{1/b}.\]
The other linear terms are uniformly bounded in $C([0,\delta];H^{s})\cap L^b(0,\delta;L^a)$. Using the same estimates as in \Cref{prop:KG-2-LWP}, but only for the position component, the nonlinear part of $\Gamma_\delta$ is bounded by $C\delta^{1-\frac2b}R^3$ on a ball of radius $R$ in $C([0,\delta];H^s)\cap L^b(0,\delta;L^a)$, and its Lipschitz constant on that ball is bounded by $C\delta^{1-\frac2b}R^2$. Since $b>2$, choosing $R$ sufficiently large, independently of $\delta$, and then $\delta>0$ sufficiently small, the map leaves that closed ball invariant and is a contraction there. Hence it admits a unique fixed point $y$ in that ball, and
\[\|y\|_{C([0,\delta];H^s)}+\|y\|_{L^b(0,\delta;L^a)}\le C(Y_0,\xi)\]
for all sufficiently small $\delta>0$. By the integral formulation, $y\in C^1([0,\delta];H^{s-1})$, and $(y,\partial_t y)^\top$ satisfies the mild formulation with initial datum $(y_0,y_1+\delta^{-1}\xi)^\top$ and no control. By uniqueness,
\[\binom{y(t)}{\partial_t y(t)}=\mathcal R_t\left(\binom{y_0}{y_1+\delta^{-1}\xi},0\right),\qquad 0\le t\le\delta.\]
Using the cubic estimate \eqref{eq:KG-Strichartz-cubic-bound}, we get 
\[\|y^3\|_{L^1(0,\delta;H^{s-1})} \le C\delta^{1-\frac2b} \|y\|_{C([0,\delta];H^{s})} \|y\|_{L^b(0,\delta;L^a)}^2 \longrightarrow0. \] 
Consequently, 
\[\left\|\int_0^\delta \Lambda^{-1}\sin((\delta-\sigma)\Lambda)y(\sigma)^3d\sigma \right\|_{H^{s}} \longrightarrow0. \] 
Hence $y(\delta)\to y_0+\xi$ in $H^{s}$. For the velocity component, again by the mild formulation, 
\[\partial_t y(\delta) = -\Lambda\sin(\delta\Lambda)y_0 +\cos(\delta\Lambda)y_1 +\delta^{-1}\cos(\delta\Lambda)\xi -\kappa\int_0^\delta \cos((\delta-\sigma)\Lambda)y(\sigma)^3d\sigma.\] 
Therefore 
\[\partial_t y(\delta)-\delta^{-1}\xi = -\Lambda\sin(\delta\Lambda)y_0 +\cos(\delta\Lambda)y_1 +\delta^{-1}(\cos(\delta\Lambda)-I)\xi-\kappa\int_0^\delta \cos((\delta-\sigma)\Lambda)y(\sigma)^3d\sigma . \]
The first term tends to $0$ in $H^{s-1}$, the second one tends to $y_1$ in $H^{s-1}$, and, since $\xi$ is smooth, 
\[\delta^{-1}(\cos(\delta\Lambda)-I)\xi\to0 \qquad\text{in }H^{s-1}. \] 
Finally, 
\[\left\| \int_0^\delta \cos((\delta-\sigma)\Lambda)y(\sigma)^3d\sigma \right\|_{H^{s-1}} \le C\|y^3\|_{L^1(0,\delta;H^{s-1})} \longrightarrow0. \]
Thus $\partial_t y(\delta)-\delta^{-1}\xi\to y_1$ in $H^{s-1}$ and \eqref{eq:KG-2-position-limit} follows. Choose $\delta\in(0,\varkappa/3)$ small enough so that 
\[\left\| \mathcal R_\delta\left(\binom{y_0}{y_1+\delta^{-1}\xi},0\right) +\binom{0}{-\delta^{-1}\xi} -\binom{y_0+\xi}{y_1} \right\|_{\mathcal H^s} < \frac{\eps}{3}. \] 
By the same fixed-point argument and continuous dependence on the initial datum, there exists $\rho>0$ such that, if
\[\left\|Z-\binom{y_0}{y_1+\delta^{-1}\xi}\right\|_{\mathcal H^s}<\rho, \]
then the uncontrolled solution with initial datum $Z$ is well defined on $[0,\delta]$ and
\[\left\|\mathcal R_\delta(Z,0)-\mathcal R_\delta\left(\binom{y_0}{y_1+\delta^{-1}\xi},0\right)\right\|_{\mathcal H^s} < \frac{\eps}{3}. \]
Since $\tau_{(0,\delta^{-1}\xi)^\top}$ is small-time approximately realizable, there exist $T_1\in(0,\varkappa/3)$ and $\eta_1\in L^\infty(0,T_1;\allowbreak E^\R)$ such that 
\[\left\| \mathcal R_{T_1}(Y_0,\eta_1) -\binom{y_0}{y_1+\delta^{-1}\xi} \right\|_{\mathcal H^s} < \rho. \]
Set $Z:=\mathcal R_{T_1}(Y_0,\eta_1)$. On $[T_1,T_1+\delta]$, we take zero control, so the state becomes $\mathcal R_\delta(Z,0)$. Finally, since $\tau_{(0,-\delta^{-1}\xi)^\top}$ is small-time approximately realizable, there exist $T_2\in(0,\varkappa/3)$ and $\eta_2\in L^\infty(0,T_2;E^\R)$ such that 
\[\left\| \mathcal R_{T_2}(\mathcal R_\delta(Z,0),\eta_2) - \left(\mathcal R_\delta(Z,0)+\binom{0}{-\delta^{-1}\xi}\right) \right\|_{\mathcal H^s} < \frac{\eps}{3}. \] 
Concatenating the three pieces, with $T:=T_1+\delta+T_2<\varkappa$, we obtain 
\[\left\|\mathcal R_T(Y_0,\eta)-\left(Y_0+\binom{\xi}{0}\right)\right\|_{\mathcal H^s} < \eps. \]
Therefore $\tau_{(\xi,0)^\top}$ is small-time approximately realizable. 
\end{proof}
\begin{lemma}[Cubic kick]
\label{lem:KG-2-cubic-kick}
Let $\varphi\in C^\infty(M;\R)$ and $\alpha\in\R$. Assume that, for every
$\delta>0$, the translations $\tau_{(\pm\delta^{-1/3}\alpha\varphi,0)^\top}$ are small-time approximately realizable for \eqref{eq:KG-2}. Then the translation $\tau_{(0,-\kappa\alpha^3\varphi^3)^\top}$ is small-time approximately realizable for \eqref{eq:KG-2}.
\end{lemma}
\begin{proof}
Set $\chi:=(0,-\kappa\alpha^3\varphi^3)^\top$, and fix $Y_0=(y_0,y_1)^\top\in\mathcal H^s$, $\eps>0$ and $\varkappa>0$. We first prove that for all sufficiently small $\delta>0$, the uncontrolled solution with initial datum $(y_0+\delta^{-1/3}\alpha\varphi,y_1)^\top$ is well defined on $[0,\delta]$ and
\begin{equation}
\label{eq:KG-Strichartz-cubic-kick-limit}
\mathcal R_\delta\left(\binom{y_0+\delta^{-1/3}\alpha\varphi}{y_1},0\right)-\binom{\delta^{-1/3}\alpha\varphi}{0}\longrightarrow Y_0+\chi \qquad\text{in }\mathcal H^s.
\end{equation}
As in the proof of \Cref{lem:KG-odd-kick}, we first construct a solution of the renormalized equation. For $Z=(z,\partial_t z)^\top\in\mathcal X_\delta^s$, define
\[\widetilde{\Gamma}_\delta(Z):=\Gamma_{(y_0+\delta^{-1/3}\alpha\varphi,y_1)^\top,0}\left(Z+\binom{\delta^{-1/3}\alpha\varphi}{0}\right)-\binom{\delta^{-1/3}\alpha\varphi}{0},\]
where $\Gamma_{Y_0,\eta}$ is the Duhamel map associated with initial datum $Y_0$ and control $\eta$. The point which differs from the case $s>d/2$ is the uniform bound for the renormalized solution. We prove it in the Strichartz space. By Duhamel's formula, the first component is
\[\begin{aligned}
(\widetilde{\Gamma}_\delta(Z))_1(t)=&\cos(t\Lambda)y_0+\Lambda^{-1}\sin(t\Lambda)y_1+(\cos(t\Lambda)-I)\delta^{-1/3}\alpha\varphi\\
&-\kappa\int_0^t\Lambda^{-1}\sin((t-\sigma)\Lambda)(\delta^{-1/3}\alpha\varphi+z(\sigma))^3d\sigma.    
\end{aligned}\]
The second component is
\[\begin{aligned}(\widetilde{\Gamma}_\delta(Z))_2(t)=&-\Lambda\sin(t\Lambda)y_0+\cos(t\Lambda)y_1-\Lambda\sin(t\Lambda)\delta^{-1/3}\alpha\varphi\\
&-\kappa\int_0^t\cos((t-\sigma)\Lambda)(\delta^{-1/3}\alpha\varphi+z(\sigma))^3d\sigma.\end{aligned}\]
The linear terms are uniformly bounded. For instance, since $\varphi$ is smooth,
\[\|(\cos(t\Lambda)-I)\delta^{-1/3}\alpha\varphi\|_{H^s}+\|\Lambda\sin(t\Lambda)\delta^{-1/3}\alpha\varphi\|_{H^{s-1}}\le C\delta^{2/3}\]
for $0\le t\le\delta$. Moreover,
\[\|(\cos(t\Lambda)-I)\delta^{-1/3}\alpha\varphi\|_{L^b(0,\delta;L^a)}\le C\delta^{2/3+1/b}.\]
We now estimate the nonlinear terms. Using the product estimate from \Cref{prop:KG-2-LWP}, for $z$ bounded in
$C([0,\delta];H^{s})\cap L^b(0,\delta;L^a)$, we have
\[\begin{aligned}
\|(\delta^{-1/3}\alpha\varphi+z)^3\|_{L^1(0,\delta;H^{s-1})}\le& C\left(1+\delta^{1/3}\|z\|_{C([0,\delta];H^{s})}+\delta^{\frac23-\frac1b}\|z\|_{C([0,\delta];H^{s})}\|z\|_{L^b(0,\delta;L^a)}\right.\\
&\left.+\delta^{1-\frac2b}\|z\|_{C([0,\delta];H^{s})}\|z\|_{L^b(0,\delta;L^a)}^2 \right).
\end{aligned}\]
Similarly, for $Z=(z,\partial_t z)^\top$ and $\widetilde Z=(\widetilde z,\partial_t\widetilde z)^\top$ in a fixed ball of radius $R$ in $\mathcal X_\delta^s$,
\[\|(\delta^{-1/3}\alpha\varphi+z)^3-(\delta^{-1/3}\alpha\varphi+\widetilde z)^3\|_{L^1(0,\delta;H^{s-1})}\le C_R\left(\delta^{1/3}+\delta^{\frac23-\frac1b}+\delta^{1-\frac2b}\right)\|Z-\widetilde Z\|_{\mathcal X_\delta^s}.\]
The previous estimates give, for $\|Z\|_{\mathcal X_\delta^s}\le R$,
\[\|\widetilde{\Gamma}_\delta(Z)\|_{\mathcal X_\delta^s}\le C_0+C\left(\delta^{1/3}R+\delta^{\frac23-\frac1b}R^2+\delta^{1-\frac2b}R^3\right),\]
where the constants are independent of $\delta$. Since $b>2$, choosing $R$ sufficiently large and then $\delta>0$ sufficiently small, the map leaves that closed ball invariant. The difference estimate above, together with the energy and Strichartz estimates, shows that it is a contraction there. Hence it admits a unique fixed point $Z^\delta=(z^\delta,\partial_t z^\delta)^\top$ in that ball. Set $y^\delta(t):=z^\delta(t)+\delta^{-1/3}\alpha\varphi$. By construction, $(y^\delta,\partial_t y^\delta)^\top$ satisfies the mild formulation with initial datum $(y_0+\delta^{-1/3}\alpha\varphi,y_1)^\top$ and no control. By uniqueness,
\[\binom{y^\delta(t)}{\partial_t y^\delta(t)}=\mathcal R_t\left(\binom{y_0+\delta^{-1/3}\alpha\varphi}{y_1},0\right),\qquad 0\le t\le\delta.\]
Moreover,
\begin{equation}
\label{eq:KG-Strichartz-z-bound}
\|z^\delta\|_{C([0,\delta];H^s)}+\|\partial_t y^\delta\|_{C([0,\delta];H^{s-1})}+\|z^\delta\|_{L^b(0,\delta;L^a)}\le C
\end{equation}
for all sufficiently small $\delta>0$. We now compute the leading nonlinear contribution. Expanding
\[(y^\delta)^3=(\delta^{-1/3}\alpha\varphi+z^\delta)^3=\delta^{-1}\alpha^3\varphi^3+3\delta^{-2/3}\alpha^2\varphi^2z^\delta
+3\delta^{-1/3}\alpha\varphi(z^\delta)^2+(z^\delta)^3,\]
and using \eqref{eq:KG-Strichartz-z-bound}, we obtain
\[\left\|(y^\delta)^3-\delta^{-1}\alpha^3\varphi^3\right\|_{L^1(0,\delta;H^{s-1})} \le C\left(\delta^{1/3}+\delta^{\frac23-\frac1b}+\delta^{1-\frac2b}\right)\longrightarrow0 .\]
For the position component, the linear part satisfies
\[\cos(\delta\Lambda)(y_0+\delta^{-1/3}\alpha\varphi)+\Lambda^{-1}\sin(\delta\Lambda)y_1 -\delta^{-1/3}\alpha\varphi\longrightarrow y_0\qquad\text{in }H^{s}.\]
Moreover,
\[\left\|\delta^{-1}\int_0^\delta\Lambda^{-1}\sin((\delta-\sigma)\Lambda)\alpha^3\varphi^3d\sigma\right\|_{H^{s}}\le C\delta\|\varphi^3\|_{H^{s}}\longrightarrow0,\]
and the remaining nonlinear part tends to zero because
\[\left\|(y^\delta)^3-\delta^{-1}\alpha^3\varphi^3\right\|_{L^1(0,\delta;H^{s-1})}\to0.\]
Thus the first component converges to $y_0$ in $H^{s}$. For the velocity component, we have
\[\partial_t y^\delta(\delta)=-\Lambda\sin(\delta\Lambda)(y_0+\delta^{-1/3}\alpha\varphi)+\cos(\delta\Lambda)y_1
-\kappa\int_0^\delta\cos((\delta-\sigma)\Lambda)(y^\delta(\sigma))^3d\sigma .\]
The linear terms converge to $y_1$ in $H^{s-1}$. For the nonlinear term,
\[\int_0^\delta\cos((\delta-\sigma)\Lambda)(y^\delta(\sigma))^3d\sigma=\delta^{-1}\alpha^3\int_0^\delta\cos((\delta-\sigma)\Lambda)\varphi^3d\sigma+o_{H^{s-1}}(1).\]
By the strong continuity of $\cos(t\Lambda)$ on $H^{s-1}$,
\[\frac1\delta\int_0^\delta\cos((\delta-\sigma)\Lambda)\varphi^3d\sigma\longrightarrow\varphi^3\qquad\text{in }H^{s-1}.\]
Hence
\[\partial_t y^\delta(\delta)\longrightarrow
y_1-\kappa\alpha^3\varphi^3\qquad\text{in }H^{s-1}.\]
This proves \eqref{eq:KG-Strichartz-cubic-kick-limit}. The rest of the argument is the same concatenation procedure as before. Choose $\delta\in(0,\varkappa/3)$ small enough so that the error in \eqref{eq:KG-Strichartz-cubic-kick-limit} is smaller than $\eps/3$. By the same fixed-point argument and continuous dependence on the initial datum, there exists $\rho>0$ such that, if
\[\left\|Z-\binom{y_0+\delta^{-1/3}\alpha\varphi}{y_1}\right\|_{\mathcal H^s}<\rho,\]
then the uncontrolled solution with initial datum $Z$ is well defined on $[0,\delta]$ and
\[\left\|\mathcal R_\delta(Z,0)-\mathcal R_\delta\left(\binom{y_0+\delta^{-1/3}\alpha\varphi}{y_1},0\right)\right\|_{\mathcal H^s}<\frac{\eps}{3}.\]
Since $\tau_{(\delta^{-1/3}\alpha\varphi,0)^\top}$ is small-time approximately realizable, there exist $T_1\in(0,\varkappa/3)$ and $\eta_1\in L^\infty(0,T_1;E^\R)$ such that
\[\left\|\mathcal R_{T_1}(Y_0,\eta_1)-\binom{y_0+\delta^{-1/3}\alpha\varphi}{y_1}\right\|_{\mathcal H^s}<\rho.\]
Set $Z:=\mathcal R_{T_1}(Y_0,\eta_1)$. On $[T_1,T_1+\delta]$, we take zero control, so the state becomes $\mathcal R_\delta(Z,0)$. Finally, since $\tau_{(-\delta^{-1/3}\alpha\varphi,0)^\top}$ is small-time approximately
realizable, there exist $T_2\in(0,\varkappa/3)$ and $\eta_2\in L^\infty(0,T_2;E^\R)$ such that
\[\left\|\mathcal R_{T_2}(\mathcal R_\delta(Z,0),\eta_2)-\left(\mathcal R_\delta(Z,0)-\binom{\delta^{-1/3}\alpha\varphi}{0}\right)
\right\|_{\mathcal H^s}<\frac{\eps}{3}.\]
Concatenating the three pieces, with $T:=T_1+\delta+T_2<\varkappa$, we obtain
\[\|\mathcal R_T(Y_0,\eta)-(Y_0+\chi)\|_{\mathcal H^s}<\eps.\]
Therefore $\tau_\chi$ is small-time approximately realizable.
\end{proof}
The analogues of \Cref{lem:KG-velocity-position-translations,lem:KG-sums,lem:KG-saturated-directions} and \Cref{prop:KG-saturation-criterion} follow directly. Here, $\mathsf E_n^\R$ and $\mathsf E_\infty^\R$ denote the cubic power saturation spaces.
\begin{lemma}
\label{lem:KG-2-saturation}
The following properties hold:
\begin{itemize}
\item Let $\xi\in C^\infty(M;\R)$. Assume that, for every $\delta>0$, the
translations $\tau_{(0,\pm\delta^{-1}\xi)^\top}$ are small-time approximately realizable for \eqref{eq:KG-2}. Then the
translation $\tau_{(\xi,0)^\top}$ is small-time approximately realizable for
\eqref{eq:KG-2}.
\item The set of small-time approximately realizable translations is stable
under sums.
\item Let $\xi\in \mathsf E_\infty^\R$. Then both translations $\tau_{(\xi,0)^\top}$ and $\tau_{(0,\xi)^\top}$ are small-time approximately realizable for \eqref{eq:KG-2}.
\item If $\mathsf E_\infty^\R$ is dense in $H^{s}(M;\R)$ and in $H^{s-1}(M;\R)$, then
\eqref{eq:KG-2} is small-time approximately controllable in $\mathcal H^s$.
\end{itemize}
\end{lemma}

\textit{Proof of \Cref{thm:KG-STAC2}.}
By \Cref{th:odd-saturation-density}, the power saturation $\mathsf E_\infty^\R$ of $E^\R$ is dense in $H^{s}(M;\R)$ and in $H^{s-1}(M;\R)$. The direct velocity translations are given by \Cref{lem:KG-2-direct-velocity-translations}, and the cubic generation step is given by \Cref{lem:KG-2-cubic-kick}. Hence \Cref{lem:KG-2-saturation} applies and yields the desired small-time approximate controllability in $\mathcal H^s$.
\qed
\section{Stochastic consequences}
\label{sec:stochastic-consequences}
In this section we derive two consequences of the small-time approximate controllability results proved above. We first establish topological irreducibility for the Markov transition probabilities associated with a randomly forced equation. We then prove that the trajectories are almost surely unbounded in higher Sobolev norms. Throughout this section we work in a defocusing regime in which the original equation is globally well-posed. 

The argument is written in a unified way for the cubic Schr\"odinger and Klein--Gordon equations on a smooth compact connected Riemannian manifold $(M,g)$ without boundary. Let $W\in C^\infty(M;\R)$ and $\kappa>0$. We set
\[X^s_{\rm Sch}:=H^s(M;\C),\qquad X^s_{\rm KG}:=\mathcal H^s:=H^s(M;\R)\times H^{s-1}(M;\R).\]
In what follows, $X^s$ denotes the space corresponding to the equation under consideration. Let $E\subset C^\infty(M;\mathbb K)$ be the finite-dimensional control space used in the saturation argument, where $\mathbb K=\C$ for the Schr\"odinger equation and $\mathbb K=\R$ for the Klein--Gordon equation. \medskip

\noindent As already mentioned, we consider a setting in which global-in-time solutions exist. 
\begin{assumption}[Global well-posedness]
\label{ass:global-wellposedness}
For the equation under consideration, for every $z_0\in X^s$ and every $\eta\in L^2_{\mathrm{loc}}([0,\infty);E)$, there exists a unique global solution in $C([0,\infty);X^s)$ of \eqref{eq:NLS-2} or \eqref{eq:KG-2}. Moreover, for every $T>0$, the resolving map is continuous from $X^s\times L^2(0,T;E)$ into $C([0,T];X^s)$.
\end{assumption}

\Cref{ass:global-wellposedness} can be verified in several situations. For the uncontrolled defocusing cubic Schr\"odinger equation, global well-posedness on closed surfaces is proved in \cite[Theorem~1.1]{Hani2012} for $r>2/3$. The energy-space global theory on compact manifolds of dimension at most three and the propagation of higher regularity follow from the Strichartz estimates and the
arguments in \cite{burq2004strichartz}. For the uncontrolled defocusing cubic wave equation on a three-dimensional compact manifold, global well-posedness is stated in \cite[Theorem~1]{BurqTzvetkov2014} for $r\geq3/4$. In particular, we can prove the following. 

\begin{proposition}
\label{prop:NLS-global-wellposedness}
Assume either that $d=2$ and $s\geq1$, or that $d=3$, $W=0$ and $s>1$. Then \Cref{ass:global-wellposedness} holds for \eqref{eq:NLS-2}.
\end{proposition}

\begin{proposition}
\label{prop:KG-global-wellposedness}
Assume that $d\leq3$ and $s\geq1$. Then \Cref{ass:global-wellposedness} holds for equation \eqref{eq:KG-2}.
\end{proposition}
For the sake of completeness, the proofs are given in \Cref{app:NLS-GWP,app:KG-GWP}.\medskip

\noindent For the remainder of this section, the Sobolev indices under consideration are
such that \Cref{ass:global-wellposedness} and the hypotheses of \Cref{thm:NLS-STAC2} or \Cref{thm:KG-STAC2} hold. We write
\[\mathcal R_T(z_0,\eta)\in X^s\]
for the value at time $T$ of the global solution starting from $z_0\in X^s$ and driven by $\eta\in L^2(0,T;E)$. \medskip

\noindent We now introduce the random forcing. Let $(\Omega,\mathcal F,\mathbb P)$ be a probability space and let $\mathcal E:=L^2(0,1;E)$. Let $\ell$ be a Borel probability measure on the separable Hilbert space $\mathcal E$, and let $(\eta_k)_{k\geq0}$ be a sequence of independent and identically distributed $\mathcal E$-valued random variables with common law $\ell$. We define the random forcing by
\begin{equation}
\label{eq:block-random-forcing}
\eta(t)=\eta_k(t-k),
\qquad t\in[k,k+1),\quad k\geq0.
\end{equation}
For every finite $T>0$, one has $\eta\in L^2(0,T;E)$ almost surely. Hence the global solution evaluated at integer times defines the time-homogeneous Markov chain
\begin{equation}
\label{eq:integer-Markov-chain}
Z_{k+1}=\mathcal R_1(Z_k,\eta_k), \qquad Z_0=z_0.
\end{equation}
Indeed, $\eta_k$ is independent of the past and its law does not depend on $k$. The phase space is $X^s$, endowed with its Borel $\sigma$-algebra $\mathcal B(X^s)$. Its one-step transition kernel is
\begin{equation}
\label{eq:one-step-transition-kernel}
P_1(z,\Gamma) := \mathbb P\bigl(\mathcal R_1(z,\eta_0)\in\Gamma\bigr), \qquad z\in X^s,\quad \Gamma\in\mathcal B(X^s).
\end{equation}
Thus, $P_1(z,\Gamma)$ is the probability that the solution starting from $z$ belongs to the Borel set $\Gamma$ after one forcing block. We refer to \cite[Section~1.3]{KuksinShirikyan2012} for the standard Markov framework for randomly forced equations. The continuity of the resolving map implies that $z\mapsto P_1(z,\Gamma)$ is Borel measurable for every $\Gamma\in\mathcal B(X^s)$, so $P_1$ is indeed a Markov kernel. We impose the following non-degeneracy assumption.

\begin{assumption}[Full support of the law]
\label{ass:full-support}
The law $\ell$ has full support in $\mathcal E$: $\operatorname{supp}\ell=\mathcal E$. Equivalently, for every $\zeta\in\mathcal E$ and every $\rho>0$, 
\[0<\ell\bigl(B_{\mathcal E}(\zeta,\rho)\bigr)=\mathbb P(\|\eta_k-\zeta\|_{\mathcal{E}}<\rho).\]
\end{assumption}

A concrete class of noises satisfying \Cref{ass:full-support} can be constructed as follows. Let $(e_j)_{j\geq1}$ be an orthonormal basis of $\mathcal E$, let $b_j\neq0$ with $\sum_{j\geq1}|b_j|^2<\infty$, and set
\[\eta_k=\sum_{j\geq1}b_j\xi_{j,k}e_j.\]
Assume that the $\mathbb K$-valued random variables $\xi_{j,k}$ are independent and, for each fixed $j$, identically distributed as $k$ varies. Assume also that $\mathbb E|\xi_{j,k}|^2=1$ and that their densities with respect to Lebesgue measure on
$\mathbb K$ are positive everywhere. Then the series converges in $L^2(\Omega;\mathcal E)$ and its law has full support in $\mathcal E$. 

\subsection{Topological irreducibility}
We first show that global well-posedness upgrades small-time approximate controllability to approximate controllability at any prescribed positive time. This point is needed because the transition kernel \eqref{eq:one-step-transition-kernel} is defined at the fixed time $1$.
\begin{lemma}
\label{lem:STAC-fixed-time}
For every $T_0>0$, every $z_0,z_1\in X^s$, and every $\varepsilon>0$, there exists $\zeta\in L^\infty(0,T_0;E)$ such that
\[\|\mathcal R_{T_0}(z_0,\zeta)-z_1\|_{X^s}<\varepsilon.\]
\end{lemma}
\begin{proof}
By the continuity of $t\mapsto\mathcal R_t(z_1,0)$, we may choose $\delta\in(0,T_0)$ such that $\|\mathcal R_\theta(z_1,0)-z_1\|_{X^s}<\varepsilon/2$ for $0\leq\theta\leq\delta$. Continuous dependence on the initial datum, uniformly on $[0,\delta]$, then gives $\varepsilon_0>0$ such that
\[\|w-z_1\|_{X^s}<\varepsilon_0,\qquad 0\leq\theta\leq\delta \quad\Longrightarrow\quad \|\mathcal R_\theta(w,0)-z_1\|_{X^s}<\varepsilon.\]
Set $\widetilde z_0:=\mathcal R_{T_0-\delta}(z_0,0)$. By \Cref{thm:NLS-STAC2} or \Cref{thm:KG-STAC2}, there exist $T\in(0,\delta)$ and $\zeta_0\in L^\infty(0,T;E)$ such that
\[\|\mathcal R_T(\widetilde z_0,\zeta_0)-z_1\|_{X^s}<\varepsilon_0.\]
Define $\zeta\in L^\infty(0,T_0;E)$ by
\[\zeta(t)=
\begin{cases}
0, &0\leq t<T_0-\delta,\\
\zeta_0(t-T_0+\delta), &T_0-\delta\leq t<T_0-\delta+T,\\
0, &T_0-\delta+T\leq t\leq T_0.
\end{cases}\]
The concatenation property gives $\mathcal R_{T_0}(z_0,\zeta)=\mathcal R_{\delta-T}\bigl(\mathcal R_T(\widetilde z_0,\zeta_0),0\bigr)$. Since $0\leq\delta-T\leq\delta$, the choice of $\delta$ and $\varepsilon_0$ proves the result.
\end{proof}

\begin{proposition}[Topological irreducibility]
\label{prop:irreducibility}
Under \Cref{ass:full-support}, for every $z_0,z_1\in X^s$ and every
$\varepsilon>0$,
\[P_1\bigl(z_0,B_{X^s}(z_1,\varepsilon)\bigr)>0.\]
In particular, the Markov chain \eqref{eq:integer-Markov-chain} is topologically irreducible in $X^s$.
\end{proposition}

\begin{proof}
By \Cref{lem:STAC-fixed-time}, there exists $\widehat\zeta\in L^\infty(0,1;E)\subset\mathcal E$ such that
\[\|\mathcal R_1(z_0,\widehat\zeta)-z_1\|_{X^s}<\frac{\varepsilon}{2}.\]
By the continuity of the resolving map with respect to the forcing, there exists $\rho>0$ such that
\[\|\eta-\widehat\zeta\|_{\mathcal E}<\rho\quad\Longrightarrow\quad\|\mathcal R_1(z_0,\eta)-z_1\|_{X^s}<\varepsilon.\]
Therefore, \Cref{ass:full-support} gives
\[P_1\bigl(z_0,B_{X^s}(z_1,\varepsilon)\bigr)\geq\ell\bigl(B_{\mathcal E}(\widehat\zeta,\rho)\bigr)>0.\]
\end{proof}

\begin{remark}
This is the standard support mechanism by which approximate controllability
and continuity of the resolving map imply topological irreducibility; see, for
instance, \cite[Lemma~4.7]{GHHM18}.
\end{remark}

We shall use the following uniform version on compact sets.
\begin{lemma}
\label{lem:compact-uniform-support}
Let $K\subset X^s$ be compact. For every $z_\ast\in X^s$ and every $\varepsilon>0$, there exists $p=p(K,z_\ast,\varepsilon)>0$ such that
\[\inf_{z\in K}P_1\bigl(z,B_{X^s}(z_\ast,\varepsilon)\bigr)\geq p.\]
\end{lemma}

\begin{proof}
Fix $z\in K$. The proof of \Cref{prop:irreducibility} provides a control
$\widehat\zeta_z\in\mathcal E$. By the joint continuity of the resolving map,
there exist $\delta_z>0$ and $\rho_z>0$ such that
\[\|w-z\|_{X^s}<\delta_z,\qquad\|\eta-\widehat\zeta_z\|_{\mathcal E}<\rho_z\]
imply $\mathcal R_1(w,\eta)\in B_{X^s}(z_\ast,\varepsilon)$. By \Cref{ass:full-support}, $p_z:=\ell\bigl(B_{\mathcal E}(\widehat\zeta_z,\rho_z)\bigr)>0$. The balls $B_{X^s}(z,\delta_z)$ cover $K$. Choose a finite subcover corresponding to $z_1,\ldots,z_N$ and set $p:=\min_{1\leq j\leq N}p_{z_j}>0$.
Then the required estimate follows.
\end{proof}

\subsection{Unbounded Sobolev trajectories}
We now combine irreducibility at a weaker Sobolev level with compactness of bounded sets at a stronger  Sobolev level. The following compactness and stopping-time argument is adapted from \cite[Section~5]{DN24}. Since we work under \Cref{ass:global-wellposedness}, the original flow is globally defined and the truncation introduced in \cite{DN24} is not needed here. Let $s_0<s$ be two indices satisfying the standing global well-posedness and small-time controllability assumptions. In addition, both $s_0$ and $s$ satisfy the Strichartz hypotheses of \Cref{thm:NLS-STAC2} or \Cref{thm:KG-STAC2}, as appropriate. By uniqueness and persistence of regularity, the resolving map on $X^s$ is the restriction of the resolving map on $X^{s_0}$; we therefore use the same notation for both.
\begin{theorem}[Almost sure unboundedness]
\label{thm:unbounded-Sobolev-trajectories}
Suppose that \Cref{ass:full-support} and the preceding hypotheses on $s_0<s$ hold. Then, for every $z_0\in X^s$, the solution driven by
\eqref{eq:block-random-forcing} satisfies
\[\mathbb P\left(\limsup_{k\to\infty}\|Z_k\|_{X^s}=+\infty\right)=1.\]
Consequently,
\[\mathbb P\left(\limsup_{t\to\infty}\|\mathcal R_t(z_0,\eta)\|_{X^s}=+\infty\right)=1.\]
\end{theorem}

\begin{proof}
Fix $M>0$ and let
\[\mathcal B_M^s:=\{z\in X^s:\|z\|_{X^s}\leq M\}.\]
Since the embedding $X^s\hookrightarrow X^{s_0}$ is compact, the closure $K_M:=\overline{\mathcal B_M^s}^{X^{s_0}}$ is a compact subset of $X^{s_0}$. Let $C_{s,s_0}\geq 1$ be such that
\[\|z\|_{X^{s_0}}\leq C_{s,s_0}\|z\|_{X^s},\qquad \forall z\in X^s.\]
Choose $z_M\in X^s$ satisfying
\[\|z_M\|_{X^{s_0}}>C_{s,s_0}M\]
by multiplying any non-zero smooth state by a sufficiently large constant. Set
\[\varepsilon_M:=\|z_M\|_{X^{s_0}}-C_{s,s_0}M>0.\]
If $w\in X^s$ and
$\|w-z_M\|_{X^{s_0}}<\varepsilon_M$, then
\[\|w\|_{X^{s_0}}>C_{s,s_0}M,\]
and consequently $\|w\|_{X^s}>M$. Apply \Cref{lem:compact-uniform-support} at the Sobolev level $s_0$ to the compact set $K_M$, the target $z_M$, and the radius $\varepsilon_M$. There exists $p_M>0$ such that
\[\inf_{z\in K_M}\mathbb P\left(\mathcal R_1(z,\eta_0)\in B_{X^{s_0}}(z_M,\varepsilon_M)\right)\geq p_M.\]
Since a solution starting from $X^s$ remains in $X^s$, it follows that
\[\{\mathcal R_1(z,\eta_0)\in B_{X^{s_0}}(z_M,\varepsilon_M)\}\subset \{\|\mathcal R_1(z,\eta_0)\|_{X^s}>M\}.\]
Therefore
\begin{equation}
\label{eq:uniform-escape-probability}
\inf_{\|z\|_{X^s}\leq M}P_1\bigl(z,\{w\in X^s:\|w\|_{X^s}>M\}\bigr)\geq p_M.
\end{equation}
Let $\varkappa_M:=\inf\{k\geq0:\|Z_k\|_{X^s}>M\}$ be the first exit time from the ball of radius $M$, with the convention $\inf\varnothing=+\infty$, and let
$\mathcal F_k:=\sigma(\eta_0,\ldots,\eta_{k-1})$. The random variable $Z_k$ and the event $\{\varkappa_M>k\}$ are $\mathcal F_k$-measurable. On this event, one has $\|Z_k\|_{X^s}\leq M$. Hence, by the Markov property and
\eqref{eq:uniform-escape-probability},
\[\begin{aligned}
\mathbb P\left(\|Z_{k+1}\|_{X^s}>M\mid\mathcal F_k\right)&=P_1\left(Z_k,\{w\in X^s:\|w\|_{X^s}>M\}\right)\\
&\geq p_M
\end{aligned}\]
on $\{\varkappa_M>k\}$. Moreover, on this same event,
$\varkappa_M>k+1$ if and only if $\|Z_{k+1}\|_{X^s}\leq M$. Therefore,
\[\begin{aligned}
\mathbb P(\varkappa_M>k+1\mid\mathcal F_k)&=\mathbf 1_{\{\varkappa_M>k\}}\mathbb P\left(\|Z_{k+1}\|_{X^s}\leq M\mid\mathcal F_k\right)\\
&\leq(1-p_M)\mathbf 1_{\{\varkappa_M>k\}}.
\end{aligned}\]
Taking expectations, we obtain
\[\begin{aligned}
\mathbb P(\varkappa_M>k+1)
&=\mathbb E\left[\mathbb P(\varkappa_M>k+1\mid\mathcal F_k)\right]\\
&\leq(1-p_M)\mathbb E\left[\mathbf 1_{\{\varkappa_M>k\}}\right]\\
&=(1-p_M)\mathbb P(\varkappa_M>k).
\end{aligned}\]
Iterating this inequality, we obtain
\[\mathbb P(\varkappa_M>k)\leq(1-p_M)^k,\qquad k\geq0.\]
Since $p_M>0$, the right-hand side tends to zero as $k\to\infty$. Moreover, the events $\{\varkappa_M>k\}$ decrease to $\{\varkappa_M=+\infty\}$. Hence,
\[\mathbb P(\varkappa_M=+\infty)=\lim_{k\to\infty}\mathbb P(\varkappa_M>k)=0.\]
Therefore,
\[\mathbb P(\varkappa_M<+\infty)=1.\]
Applying this conclusion to every positive integer $M$ and taking the
countable intersection of the resulting events, we obtain
\[\mathbb P\left(\sup_{k\geq0}\|Z_k\|_{X^s}=+\infty\right)=1.\]
Since every $Z_k$ has finite $X^s$-norm, this is equivalent to
\[\mathbb P\left(\limsup_{k\to\infty}\|Z_k\|_{X^s}=+\infty\right)=1.\]
Finally, the continuous-time conclusion follows because the integer times
form a subsequence of $[0,\infty)$.

\end{proof}

\section*{Acknowledgments}
We thank Vahagn Nersesyan for suggesting the additional control profile used in the even-power saturation argument, which allowed us to simplify the proofs and the presentation of work.

\section*{Use of AI tools}
The results, proofs and writing of this paper are the authors' own. The authors used AI tools for detecting typos, editing the language, improving the exposition of the text and of the mathematics.
\appendix

\section{Global well-posedness results}
\subsection{\texorpdfstring{Proof of \Cref{prop:NLS-global-wellposedness}}{Proof of global well-posedness for NLS}}
\label[appendix]{app:NLS-GWP}
Since $E$ is finite-dimensional and contained in $C^\infty(M;\C)$, for
every $s\geq0$ there exists $C_{E,s}>0$ such that
\[\|\eta\|_{L^1(0,T;H^s)}\leq C_{E,s}T^{1/2}\|\eta\|_{L^2(0,T;E)}.\]
Thus the local well-posedness theory applies. Let $[0,T_{\max})$ be the maximal interval of existence. We first obtain a global $H^1$ bound. The mass identity gives
\[\frac{d}{dt}\|u(t)\|_{L^2}\leq \|\eta(t)\|_{L^2}\]
for almost every $t$, and hence
\begin{equation}
\label{eq:NLS-forced-mass-bound}
\sup_{0\leq t\leq T}\|u(t)\|_{L^2} \leq \|u_0\|_{L^2}+\|\eta\|_{L^1(0,T;L^2)}
\end{equation}
for every $T<T_{\max}$. Choose $c>0$ so that $W+c\geq1$ on $M$, and define
\[\mathcal E_c(u):=\frac12\|\nabla_g u\|_{L^2}^2+\frac12\int_M(W+c)|u|^2dx+\frac{\kappa}{4}\|u\|_{L^4}^4, \qquad \kappa>0.\]
This functional is coercive:
\begin{equation}
\label{eq:NLS-coercive-energy}
\mathcal E_c(u) \geq C^{-1}\bigl(\|u\|_{H^1}^2+\|u\|_{L^4}^4\bigr).
\end{equation}
Set $A_c(u):=(-\Delta_g+W+c)u+\kappa |u|^2u$. Since $i\partial_tu=A_c(u)-cu+\eta$, and both $\langle A_c(u),A_c(u)\rangle$ and
$\langle A_c(u),u\rangle$ are real, we obtain
\[\left|\frac{d}{dt}\mathcal E_c(u(t))\right|\leq\|A_c(u(t))\|_{H^{-1}}\|\eta(t)\|_{H^1}.\]
For $d\leq3$, Sobolev embedding and interpolation give
\[\||u|^2u\|_{H^{-1}}\leq C\||u|^2u\|_{L^{6/5}}=C\|u\|_{L^{18/5}}^3\leq C\|u\|_{L^2}^{1/3}\|u\|_{L^4}^{8/3}.\]
Combining this estimate with \eqref{eq:NLS-forced-mass-bound} and \eqref{eq:NLS-coercive-energy}, we get
\[\left|\frac{d}{dt}\mathcal E_c(u(t))\right|\leq C_T\|\eta(t)\|_{H^1}\bigl(1+\mathcal E_c(u(t))\bigr)^{2/3}.\]
Consequently,
\[\frac{d}{dt}\bigl(1+\mathcal E_c(u(t))\bigr)^{1/3}\leq C_T\|\eta(t)\|_{H^1},\]
and therefore
\begin{equation}
\label{eq:NLS-global-H1-bound}
\sup_{0\leq t\leq T}\|u(t)\|_{H^1}<\infty
\end{equation}
for every finite $T<T_{\max}$. It remains to propagate the $H^s$ regularity.
\subsubsection*{\texorpdfstring{Case $d=2$}{Case d=2}}
The energy-level Strichartz theory gives $u\in L^b_{\mathrm{loc}}([0,T_{\max});L^\infty(M))$ for some $b>2$. On every interval $I=[t_0,t_1]\subset[0,T_{\max})$, the Duhamel and Strichartz estimates yield
\[\|u\|_{Y_I^s}\leq C\left(\|u(t_0)\|_{H^s}+\|\eta\|_{L^1(I;H^s)}\right)+C\left(|I|+|I|^{1-\frac2b}\|u\|_{L^b(I;L^\infty)}^2\right)\|u\|_{Y_I^s}.\]
By \eqref{eq:NLS-global-H1-bound}, every finite interval can be divided into finitely many subintervals on which the last coefficient is smaller than $1/2$. Absorbing this term and iterating over the subintervals proves that
\[\sup_{0\leq t\leq T}\|u(t)\|_{H^s}<\infty\]
for every finite $T<T_{\max}$. 
\subsubsection*{\texorpdfstring{Case $d=3$}{Case d=3}}
In this case, we assume that
$W=0$. We follow the argument of
\cite[Section~3.3]{burq2004strichartz}. We first adapt their frequency-localized estimate to
the equation with additive forcing.

\begin{lemma}[Forced frequency-localized estimate]
\label{lem:NLS-forced-frequency-estimate}
Let $I=[t_0,t_1]\subset[0,T_{\max})$ with $t_1-t_0\leq1$, and let $\varphi\in C_c^\infty(\R\setminus\{0\})$. Set
\[P_h:=\varphi(-h^2\Delta_g),\qquad 0<h<1.\]
Assume that
\[\sup_{t\in I}\|u(t)\|_{H^1}\leq R.\]
Then there exist $\varepsilon>0$, a nondecreasing function
$\Lambda:\R_+\to\R_+$, and a constant $C>0$ such that
\begin{equation}
\label{eq:NLS-forced-frequency-estimate}
\|P_hu\|_{L^2(I;L^6)}
\leq C\left(h^{1/2}\|P_hu\|_{L^2(I;H^1)}+h^{1/2+\varepsilon}\Lambda(R)+\|P_h\eta\|_{L^2(I;L^2)}
\right).
\end{equation}
\end{lemma}

\begin{proof}
We follow the proof of
\cite[Lemma~3.6]{burq2004strichartz}. Let $J\subset I$ be an
interval of length comparable to $h$, let $J'$ be the corresponding enlarged
interval, and let $\chi_h$ be the time cut-off used in their proof. Set
$v_h(t):=\chi_h(t)P_hu(t)$. Since
\[i\partial_tu+\Delta_gu=\kappa|u|^2u+\eta,\]
we have
\[(i\partial_t+\Delta_g)v_h=i\chi_h'P_hu+\kappa\chi_hP_h(|u|^2u)+\chi_hP_h\eta.\]
The first two terms are exactly those appearing in \cite[Equation (3.12)]{burq2004strichartz}. The estimates \cite[Equations (3.13)--(3.16)]{burq2004strichartz}, combined with the inhomogeneous endpoint estimate of
\cite[Lemma~3.4]{burq2004strichartz}, give
\[\|P_hu\|_{L^2(J;L^6)}
\leq C\left( h^{1/2}\|P_hu\|_{L^2(J';H^1)}+h^{1+\varepsilon}\Lambda(R)+\|\chi_hP_h\eta\|_{L^2(J';L^{6/5})}
\right).\]
Since $M$ is compact, $L^2(M)$ is continuously embedded in $L^{6/5}(M)$.
Therefore,
\[\|\chi_hP_h\eta\|_{L^2(J';L^{6/5})}
\leq C\|P_h\eta\|_{L^2(J';L^2)}.\]
We now use the interval decomposition \cite[Equation (3.17)]{burq2004strichartz}. The enlarged intervals $J'$ have uniformly bounded
overlap, and hence
\[\sum_J\|P_h\eta\|_{L^2(J';L^2)}^2\leq C\|P_h\eta\|_{L^2(I;L^2)}^2.\]
Summing the squares of the estimates on the intervals $J$, and treating the
two boundary intervals exactly as in 
\cite[Equations (3.18)--(3.19)]{burq2004strichartz}, gives
\eqref{eq:NLS-forced-frequency-estimate}.
\end{proof}

Since $E$ is finite-dimensional and contained in $C^\infty(M;\C)$, for every
$r\geq0$ there exists $C_{E,r}>0$ such that
\begin{equation}
\label{eq:NLS-control-frequency-decay}
\|P_h\xi\|_{L^2} \leq C_{E,r}h^r\|\xi\|_E, \qquad \xi\in E,\quad 0<h<1.
\end{equation}
Indeed, on the support of the multiplier $\varphi(-h^2\Delta_g)$, the
frequencies are of order $h^{-1}$, and therefore
\[\|P_h\xi\|_{L^2}\leq Ch^r\|\xi\|_{H^r}\leq C_{E,r}h^r\|\xi\|_E.\]
Consequently,
\begin{equation}
\label{eq:NLS-forcing-frequency-decay}
\|P_h\eta\|_{L^2(I;L^2)} \leq C_{E,r}h^r\|\eta\|_{L^2(I;E)}.
\end{equation}
Let $\widetilde\varphi\in C_c^\infty(\R\setminus\{0\})$ be equal to one on the support of $\varphi$, and set $\widetilde P_h:=\widetilde\varphi(-h^2\Delta_g)$. Since $P_h=\widetilde P_hP_h$, the spectral Bernstein estimate of \cite[Corollary~2.2]{burq2004strichartz} gives
\[\|P_hu\|_{L^2(I;L^\infty)}\leq Ch^{-1/2}\|P_hu\|_{L^2(I;L^6)}.\]
Combining this inequality with
\eqref{eq:NLS-forced-frequency-estimate} and
\eqref{eq:NLS-forcing-frequency-decay}, we obtain
\[\|P_hu\|_{L^2(I;L^\infty)}\leq C\left(\|P_hu\|_{L^2(I;H^1)}+h^\varepsilon\Lambda(R)\right)+C_{E,r}h^{r-\frac12}\|\eta\|_{L^2(I;E)}.\]
Taking $r>1/2$, the last term is summable over the dyadic values $h=2^{-k}$. Repeating the low- and high-frequency decomposition in
\cite[Equations (3.22)--(3.26)]{burq2004strichartz}, we obtain the forced analogue of \cite[Equation (3.26)]{burq2004strichartz}:
\begin{equation}
\label{eq:NLS-forced-logarithmic-estimate}
\|u\|_{L^2(I;L^\infty)}
\leq C_R\left(\left(
|I|\log\left(2+\|u\|_{L^2(I;H^s)}\right)
\right)^{1/2}+1\right)+C_{E,r}\|\eta\|_{L^2(I;E)}.
\end{equation}
We now propagate the $H^s$ regularity. The standard Sobolev estimate gives
\[\||u|^2u\|_{H^s}\leq C\|u\|_{L^\infty}^2\|u\|_{H^s}.\]
Hence, for $t\in I=[t_0,t_1]$ by Duhamel's formulation,
\[\|u(t)\|_{H^s} \leq \|u(t_0)\|_{H^s}+\|\eta\|_{L^1(I;H^s)}+C\int_{t_0}^t\|u(\sigma)\|_{L^\infty}^2\|u(\sigma)\|_{H^s}d\sigma.\]
Gronwall's inequality yields
\begin{equation}
\label{eq:NLS-forced-Hs-Gronwall}
\|u\|_{L^\infty(I;H^s)} \leq \left(\|u(t_0)\|_{H^s}+\|\eta\|_{L^1(I;H^s)}\right)\exp\left(C\|u\|_{L^2(I;L^\infty)}^2\right).
\end{equation}
Using \eqref{eq:NLS-forced-logarithmic-estimate} in
\eqref{eq:NLS-forced-Hs-Gronwall}, and observing that
\[\|u\|_{L^2(I;H^s)}\leq|I|^{1/2}\|u\|_{L^\infty(I;H^s)},\]
we obtain the forced analogue of
\cite[Equation (3.27)]{burq2004strichartz}. More precisely,
\[\|u\|_{L^\infty(I;H^s)}\leq\left(\|u(t_0)\|_{H^s}+\|\eta\|_{L^1(I;H^s)}\right)\exp\left(C_R+C_{E,r}\|\eta\|_{L^2(I;E)}^2\right)\left(2+|I|^{1/2}\|u\|_{L^\infty(I;H^s)}
\right)^{C_R|I|}.\]
Choose $|I|\leq1$ sufficiently small so that $C_R|I|\leq1/2$.
Then, by Young's inequality,
\[\|u\|_{L^\infty(I;H^s)}\leq C\Bigg[1+\left(\|u(t_0)\|_{H^s}+\|\eta\|_{L^1(I;H^s)}\right)^2\exp\left(C_R+C_{E,r}\|\eta\|_{L^2(I;E)}^2\right)\Bigg].\]
The length condition imposed on $I$ depends only on the uniform $H^1$ bound
$R$. Suppose, for a contradiction, that $T_{\max}<+\infty$. By
\eqref{eq:NLS-global-H1-bound}, $R$ can be chosen uniformly on
$[0,T_{\max})$. Moreover,
\[\|\eta\|_{L^2(0,T_{\max};E)}<\infty,\qquad\|\eta\|_{L^1(0,T_{\max};H^s)}<\infty.\]
We divide $[0,T_{\max})$ into finitely many subintervals satisfying the
required length condition and iterate the preceding estimate, starting from
$u(0)=u_0\in H^s(M)$. This gives
\[\sup_{0\leq t<T_{\max}}\|u(t)\|_{H^s}<\infty,\]
which contradicts the local blow-up alternative. Hence $T_{\max}=+\infty$. Uniqueness follows from the local well-posedness theory. Continuous dependence is obtained by applying the local stability estimate on finitely many subintervals of $[0,T]$. 

The preceding estimates are first justified for smooth initial data and smooth $E$-valued controls. For general $u_0\in H^s(M)$ and $\eta\in L^2(0,T;E)$, choose sequences $u_0^n\in C^\infty(M)$, $\eta_n\in C^\infty([0,T];E)$, such that $u_0^n\to u_0$ in $H^s(M)$ and $\eta_n\to\eta$ in $L^2(0,T;E)$. The preceding bounds are uniform along these approximations. The local
continuous dependence of the resolving map then allows us to pass to the
limit, and the same estimates hold for the corresponding mild solution.\qed

\subsection{\texorpdfstring{Proof of \Cref{prop:KG-global-wellposedness}}{Proof of global well-posedness for KG}}
\label[appendix]{app:KG-GWP}
As in the Schrödinger case, for every $s\geq0$,
\begin{equation}
\label{eq:KG-finite-dimensional-source}
\|\eta\|_{L^1(0,T;H^s)} \leq C_{E,s}T^{1/2}\|\eta\|_{L^2(0,T;E)}.
\end{equation}
Let $[0,T_{\max})$ be the maximal interval of existence of the local solution. Choose $c>0$ so that $W+c\geq1$ on $M$, and set
\[\mathcal E_c(t):=\frac12\|\partial_ty(t)\|_{L^2}^2+\frac12\|\nabla_gy(t)\|_{L^2}^2+\frac12\int_M(W+c)|y(t)|^2dx+\frac{\kappa}{4}\|y(t)\|_{L^4}^4.\]
Then
\[\mathcal E_c(t)\geq C^{-1}\left(\|y(t)\|_{H^1}^2+\|\partial_ty(t)\|_{L^2}^2\right).\]
Multiplying the equation by $\partial_ty$ gives
\[\mathcal E_c'(t)=\langle\eta(t),\partial_ty(t)\rangle+c\langle y(t),\partial_ty(t)\rangle.\]
Therefore, with $X(t):=(1+\mathcal E_c(t))^{1/2}$,
\[X'(t)\leq CX(t)+C\|\eta(t)\|_{L^2}.\]
Gronwall's inequality and
\eqref{eq:KG-finite-dimensional-source} imply
\begin{equation}
\label{eq:KG-global-energy-bound}
\sup_{0\leq t\leq T}\left(\|y(t)\|_{H^1}+\|\partial_ty(t)\|_{L^2}\right)<\infty
\end{equation}
for every finite $T<T_{\max}$. We now propagate the higher regularity. Let $\Lambda:=(1-\Delta_g)^{1/2}$. For $s\geq1$, the fractional Leibniz rule and $H^1(M)\hookrightarrow L^6(M)$ give
\[\begin{aligned}
\|\Lambda^{s-1}(y^3)\|_{L^2} &\leq C\|\Lambda^{s-1}(y^2)\|_{L^3}\|y\|_{L^6}+C\|y^2\|_{L^3}\|\Lambda^{s-1}y\|_{L^6}\\
&\leq C\|y\|_{L^6}^2\|\Lambda^{s-1}y\|_{L^6}\\
&\leq C\|y\|_{H^1}^2\|y\|_{H^s}.
\end{aligned}\]
Thus
\begin{equation}
\label{eq:KG-tame-cubic-estimate}
\|y^3\|_{H^{s-1}} \leq C\|y\|_{H^1}^2\|y\|_{H^s}.
\end{equation}
The linear energy estimate for Klein--Gordon, applied in
$\mathcal H^s$, yields
\[\left\|\binom{y(t)}{\partial_ty(t)}\right\|_{\mathcal H^s} \leq C_T\left(\left\|\binom{y_0}{y_1}\right\|_{\mathcal H^s} +\|\eta\|_{L^1(0,t;H^{s-1})}+\int_0^t\|y(\sigma)^3\|_{H^{s-1}}d\sigma\right).\]
Using \eqref{eq:KG-tame-cubic-estimate}, we obtain
\[\left\|\binom{y(t)}{\partial_ty(t)}\right\|_{\mathcal H^s}
\leq C_T\left(\left\|\binom{y_0}{y_1}\right\|_{\mathcal H^s}+\|\eta\|_{L^1(0,t;H^{s-1})}+\int_0^t\|y(\sigma)\|_{H^1}^2\left\|\binom{y(\sigma)}{\partial_ty(\sigma)}\right\|_{\mathcal H^s}d\sigma\right).\]
The coefficient in the last integral is bounded on every finite interval by \eqref{eq:KG-global-energy-bound}. Gronwall's inequality therefore gives
\[\sup_{0\leq t\leq T}\left\|\binom{y(t)}{\partial_ty(t)}\right\|_{\mathcal H^s}<\infty\]
for every finite $T<T_{\max}$. Hence the local blow-up alternative implies $T_{\max}=+\infty$. Finally, let $y$ and $z$ be two solutions corresponding to $((y_0,y_1)^\top,\eta)$ and $((z_0,z_1)^\top,\zeta)$. On every bounded subset of $H^s(M)$, the fractional Leibniz rule gives
\[\|y^3-z^3\|_{H^{s-1}}\leq C_R\|y-z\|_{H^s}.\]
Applying the linear energy estimate to $y-z$ and using Gronwall's inequality, we obtain
\[\sup_{0\leq t\leq T}\left\|\binom{y(t)-z(t)}{\partial_ty(t)-\partial_tz(t)}\right\|_{\mathcal H^s}\leq C_{T,R}\left(\left\|\binom{y_0-z_0}{y_1-z_1}\right\|_{\mathcal H^s}+\|\eta-\zeta\|_{L^1(0,T;H^{s-1})}\right).\]
Together with \eqref{eq:KG-finite-dimensional-source}, this proves continuity with respect to the $L^2(0,T;E)$ topology.
\qed
\bibliography{biblio_sat}
\bibliographystyle{alpha}

\end{document}